\documentclass[a4paper,12pt]{amsart}
\usepackage{mathbbol}
\usepackage{tikz}

\usetikzlibrary{matrix,arrows,decorations.markings}
\usepackage[hyperindex=true]{hyperref}

\usepackage[T1]{fontenc}
\usepackage[utf8]{inputenc}
\usepackage{graphicx}

\def\DynkinNodeSize{2mm}
\def\DynkinArrowLength{3mm}
\tikzset{
  dnode/.style={
    circle,
    inner sep=0pt,
    minimum size=\DynkinNodeSize,
    fill=white,
    draw},
  middlearrow/.style={
    decoration={markings,
      mark=at position 0.6 with
      {\draw (0:0mm) -- +(+135:\DynkinArrowLength); \draw (0:0mm) -- +(-135:\DynkinArrowLength);},
    },
    postaction={decorate}
  },
  leftrightarrow/.style={
    decoration={markings,
      mark=at position 0.999 with
      {
      \draw (0:0mm) -- +(+135:\DynkinArrowLength); \draw (0:0mm) -- +(-135:\DynkinArrowLength);
      },
      mark=at position 0.001 with
      {
      \draw (0:0mm) -- +(+45:\DynkinArrowLength); \draw (0:0mm) -- +(-45:\DynkinArrowLength);
      },
    },
    postaction={decorate}
  },
  sedge/.style={
  },
  dedge/.style={
    middlearrow,
    double distance=0.5mm,
  },
  tedge/.style={
    middlearrow,
    double distance=1.0mm+\pgflinewidth,
    postaction={draw}, 
  },
  infedge/.style={
    leftrightarrow,
    double distance=0.5mm,
  },
}

 \newcommand\cX{{\mathring{X}}} 

\usepackage{amsmath,amssymb,amsfonts,mathscinet}
\usepackage{frcursive}

\newcommand\indlim{\displaystyle{\lim_{\longrightarrow}}}

\newcommand\bleq{{\preccurlyeq}}
\newcommand\rss{{\rm ss}}\newcommand\rs{{\rm s}}
\newcommand\dX{{\mathbb X}}\newcommand\dx{{\mathbb x}}
\newcommand\cI{{\mathcal I}}
\newcommand\Face{{\mathcal F}}

\newcommand\ul{\underline}
\newcommand\bkprod{{\odot_0}}
\newcommand\Gr{{\mathcal Gr}}
\newcommand\Ac{{\mathcal A}}
\newcommand\cE{{\mathcal E}}
\newcommand\KK{{\mathcal K}}\newcommand\Rr{{\mathcal R}}

\newcommand\cone{{\mathcal C}}
\newcommand\Cun{{\mathcal C}}\newcommand\Cung{{\Cun(\lg)}}
\newcommand\GP{{G/P}}\newcommand\GB{{G/B}}\newcommand\GBm{{G/B^-}}
\newcommand\GPCp{{G\times_P C^+}}\newcommand\GPCpbar{{G\times_P\bar C^+}}
\newcommand\Chi{{\aleph}}\newcommand\orb{{\varsigma}}
\newcommand\Tau{{\mathcal T}}
\newcommand\lh{{\mathfrak h}}\newcommand\lb{{\mathfrak b}}\newcommand\lln{{\mathfrak n}}
\renewcommand\lq{{\mathfrak q}}
\newcommand\dlh{{\dot
    \lh}}\newcommand\dlb{{\dot\lb}}
\newcommand\dtheta{{\dot\theta}}\newcommand\dvarpi{{\dot\varpi}}
\newcommand\Orb{{\mathcal O}}\newcommand\co{{\mathcal O}}
\newcommand\cs{{\mathcal S}}\newcommand\cf{{\mathcal F}}
\renewcommand\lg{{\mathfrak g}}\newcommand\dlg{{\dot \lg}}
\newcommand\codim{{\operatorname{codim}}}
\newcommand\ad{{\operatorname{ad}}}
\newcommand\End{{\operatorname{End}}}
\newcommand\Id{{\operatorname{Id}}}\newcommand\Ker{{\operatorname{Ker}}}
\newcommand\gr{{\operatorname{gr}}}
\newcommand\Ho{{\operatorname{H}}}

\newcommand\Lie{{\operatorname{Lie}}}

\newcommand\Aut{{\operatorname{Aut}}}
\newcommand\Exp{{\operatorname{Exp}}}
\newcommand\Proj{{\operatorname{Proj}}}
\newcommand\Spec{{\operatorname{Spec}}}
\newcommand\Hom{{\operatorname{Hom}}}\newcommand\Pic{{\operatorname{Pic}}}
\newcommand\SL{{\operatorname{SL}}}\newcommand\GL{{\operatorname{GL}}}

\newcommand\lk{{\mathfrak k}}\newcommand\lp{{\mathfrak p}}
\newcommand\lu{{\mathfrak u}}
\renewcommand\sl{{\mathfrak sl}}
\newcommand\ZZ{{\mathbb Z}}\newcommand\NN{{\mathbb N}}\newcommand\QQ{{\mathbb Q}}
\newcommand\RR{{\mathbb R}}\newcommand\CC{\mathbb C}
\renewcommand\AA{{\mathbb A}}\newcommand\PP{\mathbb P}
\newcommand\longto{\longrightarrow}
\newcommand\quot{/\hspace{-0.2em}/}
\newcommand\Li{{\mathcal L}}
\newcommand\bl{\text{\cursive l}}\newcommand\bh{\text{\cursive h}}
\newcommand\inv{^{-1}}
\newcommand{\ext}{\mathcal{E}xt}\newcommand{\tor}{\mathcal{T}or}

\newtheorem{theo}{Theorem}
\newtheorem{lemma}{Lemma}
\newtheorem{prop}{Proposition}
\newtheorem{coro}{Corollary}

\newenvironment{NB}{{\bf Remark.}}{}

\title{On the tensor semigroup of affine Kac-Moody Lie algebras}
\author{Nicolas Ressayre}
\address{Institut Camille Jordan (ICJ),
UMR CNRS 5208,
Université Claude Bernard Lyon I,
43 boulevard du 11 novembre 1918,
F - 69622 Villeurbanne cedex {\tt ressayre@math.univ-lyon1.fr}}

\begin{document}
\begin{abstract}
In this paper, we are interested in the decomposition of the tensor
product of two representations of 
a symmetrizable Kac-Moody Lie algebra $\lg$. Let $P_+$ be the set of dominant integral
weights. For $\lambda\in P_+$,  $L(\lambda)$ denotes the irreducible,
integrable, highest weight representation of $\lg$ with highest
weight    $\lambda$. Let $P_{+,\QQ}$ be the rational convex cone
generated by $P_+$. Consider the {\it tensor cone}
$$
\Gamma(\lg):=\{(\lambda_1,\lambda_2,\mu)\in P_{+,\QQ}^3\,|\,\exists N>1
\quad L(N\mu)\subset L(N\lambda_1)\otimes L(N\lambda_2)\}.
$$ 
If $\lg$ is finite dimensional, $\Gamma(\lg)$ is a polyhedral convex
cone described in \cite{BK} by an explicit finite list of
inequalities.
In general, $\Gamma(\lg)$ is nor polyhedral, nor closed. In this article
we describe the closure of  $\Gamma(\lg)$ by an explicit countable
family of linear inequalities, when $\lg$ is untwisted affine. This
solves a Brown-Kumar's conjecture \cite{BrownKumar} in this case.

We also obtain explicit saturation factors for the semigroup of
triples $(\lambda_1,\lambda_2,\mu)\in P_{+}^3$ such that 
$L(\mu)\subset L(\lambda_1)\otimes L(\lambda_2)$. Note that even the
existence of such saturation factors is not obvious since the
semigroup is not finitely generated.
For example, in case $\tilde A_n$, we prove that any integer $d_0\geq
2$ is a saturation factor, generalizing the case $\tilde A_1$ shown in \cite{BrownKumar}.
\end{abstract}

\maketitle

\section{Introduction}

Let $A$ be a symmetrizable irreducible GCM of size $l+1$.
Let $\lh\supset\{\alpha_0^\vee,\dots,\alpha_l^\vee\}$ and
$\lh^*\supset\{\alpha_0,\dots,\alpha_l\}=:\Delta$ be a realization of $A$. 
We fix an integral form $\lh_\ZZ\subset\lh$ containing each
$\alpha_i^\vee$, such that $\lh^*_\ZZ:=\Hom(\lh_\ZZ,\ZZ)$ contains $\Delta$
and such that $\lh_\ZZ/\oplus\ZZ\alpha_i^\vee$ is torsion free.
Set $\lh_\QQ^*=\lh_\ZZ^*\otimes\QQ\subset\lh^*$,  
$P_{+,\QQ}:=\{\lambda\in\lh_\QQ^*\,|\,\langle\alpha_i^\vee,\lambda\rangle\geq
0\quad\forall i\}$, and $P_+=\lh_\ZZ\cap P_{+,\QQ}$.

Let $\lg=\lg(A)$ be the associated Kac-Moody Lie algebra with Cartan
subalgebra $\lh$. 
For $\lambda\in P_+$,  $L(\lambda)$ denotes the irreducible, integrable,
highest weight representation of $\lg$ with highest weight $\lambda$. 
Define the {\it saturated tensor semigroup} as 
$$
\Gamma(A):=\{(\lambda_1,\lambda_2,\mu)\in P_{+,\QQ}^3\,|\,\exists N>1
\qquad L(N\mu)\subset L(N\lambda_1)\otimes L(N\lambda_2)\}.
$$ 

In the case that $\lg$ is a semisimple Lie algebra,
$\Gamma(A)$ (also denoted by $\Gamma(\lg)$) is a closed convex polyhedral cone given by an explicit set
of inequalities parametrized by some structure constants  of the cohomology rings of the flag varieties
for the corresponding algebraic group (see
\cite{BK,Kumar:survey,Kumar:surveyEMS,GITEigen}).

In the general case, $\Gamma(A)$ is no longer closed or polyhedral.
Nevertheless, in this paper we describe  the closure of $\Gamma(A)$ by
infinitely many explicit linear inequalities, if $\lg$ is affine
untwisted. 
Note that some intermediate results are true for any symmetrizable GCM
$A$.\\

Let $G$ be the minimal Kac-Moody group as in
\cite[Section~7.4]{Kumar:KacMoody} and $B$ its standard Borel subgroup.  
Fix
$(\varpi_{\alpha_0^\vee},\dots,\varpi_{\alpha_l^\vee})\subset\lh_\QQ$
be elements dual to the simple roots. 
Let $W$ be the Weyl group of $A$.
To any simple root $\alpha_i$, is associated  a maximal standard
parabolic subgroup $P_i$, its Weyl group $W_{P_i}\subset W$ and the
set $W^{P_i}$ of minimal length representative of elements of $W/W_{P_i}$.
We also consider the  partial flag ind-variety $X_i=G/P_i$ containing the Schubert varieties
$X_w=\overline{BwP_i/P_i}$, for $w\in W^{P_i}$.
Let
$\{\epsilon_w\}_{w\in W^{P_i}}
\subset \Ho^*(X_i,\ZZ)$
be the Schubert basis dual
to the basis of the singular homology of
$X_i$
given
by the fundamental classes of
$X_w$. 
As defined by Belkale–Kumar \cite[Section~6]{BK}
in the finite dimensional case, Brown-Kumar defined in \cite[Section~7]{BrownKumar} a deformed product
$\bkprod$
in
$\Ho^*(X_i,\ZZ)$, which is commutative and associative.

\begin{theo}
  \label{th:mainintro}
Let $\lg$ ba an affine untwisted Kac-Moody Lie algebra with central
element $c$. 
Let $(\lambda_1,\lambda_2,\mu)\in P_{+,\QQ}^3$ such that
$\lambda_1(c)>0$ and $\lambda_2(c)>0$. 

Then, 
$$(\lambda_1,\lambda_2,\mu)\in \Gamma(\lg)$$ 
if and only if
\begin{equation}
  \label{eq:centre}
\mu(c)=\lambda_1(c)+\lambda_2(c),
\end{equation}
and 
\begin{equation}
  \label{eq:8}
\langle\mu,v\varpi_{\alpha_i^\vee}\rangle\leq 
\langle\lambda_1,u_1\varpi_{\alpha_i^\vee}\rangle+\langle\lambda_2,u_2\varpi_{\alpha_i^\vee}\rangle
\end{equation}
for any $i\in\{0,\dots,l\}$ and any 
$(u_1,u_2,v)\in (W^{P_i})^3$ such that $\epsilon_v$ occurs with coefficient 1
in the deformed product
$$
\epsilon_{u_1}\bkprod\epsilon_{u_2}.
$$
\end{theo}

\bigskip
The statement of Theorem~\ref{th:mainintro} is very similar to
\cite[Theorem~22]{BK} that describes $\Gamma(\lg)$, if $\lg$ is finite
dimensional.
Nevertheless, the proof is very different.
Indeed, in the classical case the main ingredients are Kempf's
semistability theory and Hilbert-Mumford's theorem (see \cite{BK} or
\cite{GITEigen}).
These results have no known generalization in our situation. 
We overcome this difficulty by using a new strategy that we now
explain roughly speaking. 

Consider the cone $\Cung$ defined by equality~\eqref{eq:centre} and
inequalities~\eqref{eq:8}.
It remains to prove that, up to the assumption ``$\lambda_1(c)$  and
$\lambda_2(c)$ are positive'', the cone $\Cung$ is equal to
$\Gamma(\lg)$.
The proof proceeds in fives steps.

\bigskip
\noindent{\sc Step 1.}
$\Gamma(\lg)$ is convex.

This is a well-known consequence of Borel-Weil's theorem (see
Lemma~\ref{lem:BW}).

\bigskip
\noindent{\sc Step 2.} The set 
$\Gamma(\lg)$ is contained in $\Cung$.

This step is proved in \cite{BrownKumar} and reproved here. The first
ingredient is 
the easy implication in the Hilbert-Mumford's theorem. 
Indeed ``semistable $\Rightarrow$ numerically semistable'' is still
true for ind-varieties and ind-groups. 
In the finite dimensional case, the second argument is Kmeiman's
transversality theorem.
In \cite{BrownKumar}, this is by an argument in K-theory which express
the structure constants of $H^*(G/P_i,\ZZ)$ as the Euler
characteristic of sheaves supported by the intersection of three
translated Schubert or Birkhoff varieties. Here, we refine this
argument by proving a version of Kleiman's theorem that allows to
express these structure constants as the cardinality  of three
translated Schubert or Birkhoff varieties.

\bigskip
\noindent{\sc Step 3.} 
The cone $\Cung$ is locally polyhedral.

This is a consequence of Proposition~\ref{prop:locpol} below. We study
the inequalities~\eqref{eq:8} defining $\Cung$. 
In particular, we use some consequences of the nonvanishing of a
structure constant of the ring $H^*(G/P,\ZZ)$ (see
Lemmas~\ref{lem:deltaneg} and \ref{lem:calculdelta} below or
\cite{BrownKumar}).

 \bigskip
\noindent{\sc Step 4.}
Study of the boundary of $\Cung$.

Let $(\lambda_1,\lambda_2,\mu)$ be an integral point  in the boundary of $\Cung$. 
Step 3 implies that some inequality~\eqref{eq:8} has to be an
equality for $(\lambda_1,\lambda_2,\mu)$.
Then, one can use the following Theorem~\ref{th:reductionintro} to describe inductively the
multiplicity of $L(\mu)$ in $L(\lambda_1)\otimes L(\lambda_2)$.
Let $\alpha_i$ be a simple root and let $L_i$ denote the standard Levi
subgroup of $P_i$.
For $w\in W^{P_i}$ and $\lambda\in P_+$, $w\inv \lambda$ is a dominant
weight for $L_i$: we denote by $L_{L_i}(w\inv\lambda)$ the
corresponding irreducible highest weight $L_i$-module. 

\begin{theo}
  \label{th:reductionintro}
Here, $\lg$ is any symmetrizable Kac-Moody Lie algebra and $\alpha_i$
is a simple root.
Let $(\lambda_1,\lambda_2,\mu)\in P_{+}^3$.
Let $(u_1,u_2,v)\in (W^{P_i})^3$ such that $\epsilon_v$ occurs with coefficient 1
in the ordinary product
$
\epsilon_{u_1}.\epsilon_{u_2}.
$
We assume that 
\begin{equation}
  \label{eq:28}
\langle\mu,v\varpi_{\alpha_i^\vee}\rangle= 
\langle\lambda_1,u_1\varpi_{\alpha_i^\vee}\rangle+\langle\lambda_2,u_2\varpi_{\alpha_i^\vee}\rangle.
\end{equation}
Then the multiplicity of $L(\mu)$ in $L(\lambda_1)\otimes
L(\lambda_2)$ is equal to the multiplicity of $L_{L_i}(v\inv\mu)$ in $L_{L_i}(u_1\inv\lambda_1)\otimes
L_{L_i}(u_2\inv\lambda_2)$
\end{theo}

Note that Theorem~\ref{th:egalite} and its corollary in Section~\ref{sec:boundary} are a little bit
stronger than Theorem~\ref{th:reductionintro}. 

\bigskip
\noindent{\sc Step 5.} Induction.

Whereas there are numerous technical difficulties the basic idea is
simple. By convexity, it is sufficient to prove that the boundary of
$\Cung$ is contained in $\Gamma(\lg)$. 
Using Step 4, this can be proved by induction.

More precisely, consider a face $\Face$ of codimension one of $\Cung$
associated to some structure constant of $H^*(G/P_i,\ZZ)$ for
$\bkprod$ equal to one.
We have to prove that $\Face$ is contained in $\Gamma(\lg)$. 
By Theorem~\ref{th:reductionintro}, it remains to prove that the
points of $\Face$ satisfy the inequalities that characterize
$\Gamma(L_i)$.
Fix such an inequality associated to a structure constant of
$H^*(L_i/(P_j\cap L_i),\ZZ)$ for $\bkprod$ equal to one.
Consider the flags ind-varieties:
\begin{center}
\begin{tikzpicture}
  \matrix (m) [matrix of math nodes,row sep=2.5em,column sep=1.2em,minimum width=2em]
  {
     L_i/(P_j\cap L_i)&&G/(P_i\cap P_j)\\
&G/P_i&&G/P_j\\};
  \path[-stealth]
    (m-1-1) edge  (m-1-3)
    (m-1-3) edge  (m-2-2)
    (m-1-3) edge  (m-2-4);
\end{tikzpicture}  
\end{center} 

Proposition~\ref{prop:multiplicative} that shows a
property of multiplicativity for structure constants of the rings
$H^*(G/P,\ZZ)$ gives us a structure constant of $H^*(G/(P_i\cap
P_j),\ZZ)$ equal to one, for the ordinary product. 
An important point is Theorem~\ref{th:essineqBKBprod} that proves
that, if the considered inequality of $\Gamma(L_i)$ is ``useful'' then 
this structure constant of $H^*(G/(P_i\cap
P_j),\ZZ)$ is actually nonzero for $\bkprod$. 
Then it gives a  structure constant of $H^*(G/(
P_j),\ZZ)$ for $\bkprod$ equal to one. In particular, this gives an
inequality of $\Cung$ and the inequality we wanted to prove for the
points of $\Face$.


\bigskip
If we prove Theorem~\ref{th:mainintro} only for the untwisted affine
case, the general strategy should works more generally. For this
reason, we prove some intermediate results for any symmetrizable
Kac-Moody Lie algebra. In particular Steps 1, 2 and 4 works with this
generality. Proposition~\ref{prop:multiplicative} of multiplicativity
also holds in this context.

\bigskip
Let $Q$ denote the root lattice of $\lg$.
Consider the {\it tensor semigroup}
$$
\Gamma_\NN(\lg):=\{(\lambda_1,\lambda_2,\mu)\in P_{+}^3\,|\,\exists N>1
\quad L(\mu)\subset L(\lambda_1)\otimes L(\lambda_2)\}.
$$ 
It is actually a semigroup but it is not finitely generated for $\lg$
affine. Despite this, we obtain explicit saturation factors: a
positive integer $d_0$ is called a {\it saturation factor} for $\lg$
if for any $(\lambda_1,\lambda_2,\mu)\in\Gamma(\lg)\cap (P_+)^3$ such
that $\lambda_1+\lambda_2-\mu\in Q$, $L(d_0\mu)$ is a submodule of
$L(d_0\lambda_1)\otimes L(d_0\lambda_2)$. Observe that the condition 
$\lambda_1+\lambda_2-\mu\in Q$ is necessary to have 
$L(\mu)\subset L(\lambda_1)\otimes L(\lambda_2)$.

To describe our saturation factors, we need additional notation. 
Up to now, $\lg$ is the affine Lie algebra associated to the simple
Lie algebra $\dot\lg$. 
Let us define the constant $k_s$ to be the least common multiple of saturation factors
of maximal Levi subalgebras of $\lg$. The value of $k_s$ depends on
known saturation factors for the finte dimensional Lie algebras. With
the actual literature (see Section~\ref{sec:saturation}),  possible
values for $k_s$ are given in the following tabular.
$$
\begin{array}{|c|c|c|c|c|c|c|}
\hline
  {\rm Type\ of\ }\dot\lg&A_\bl&B_3\; B_4&B_\bl(\bl\geq 5)&C_\bl(\bl\geq
                                               2)&D_4&D_\bl(\bl\geq
                                                       5)\\[1.1ex]
\hline
k_s&1&2&4&2&
1&4\\
\hline\hline
 {\rm Type\ of\ } \dot\lg&E_6&E_7&E_8&F_4&G_2 &G_2\\[1.1ex]
\hline
k_s&36&144&3\,600&144&2&3\\
\hline
\end{array}
$$
Let $k_{\dot\lg}$ be the least common multiple of coordinates of
$\dot\theta$ written in terms of the simple roots. The values of
$k_{\dot\lg}$ are
$$
\begin{array}{|c|c|c|c|c|c|c|c|c|c|}
\hline
  {\rm Type}&A_\bl&B_\bl(\bl\geq 2)&C_\bl(\bl\geq 3)&D_\bl(\bl\geq
                                                      5)&E_6&E_7&E_8&F_4&G_2
  \\[1ex]
\hline
k_{\dot\lg}&1&2&2&2&6&12&60&12&6\\
\hline
\end{array}
$$

\begin{theo}
  \label{th:saturation}
Let $(\lambda_1,\lambda_2,\mu)\in (P_+)^3$ such that there exists $N>0$
  such that $L(N\mu)$ embeds in $L(N\lambda_1)\otimes L(N\lambda_2)$. 
We also assume that $\mu-\lambda_1-\lambda_2\in Q$.

Then, 
\begin{enumerate}
\item if $k_s=1$ then any integer $d\geq 2$, 
$L(dk_{\dot\lg}\mu)$ embeds in $L(dk_{\dot\lg}\lambda_1)\otimes
L(dk_{\dot\lg}\lambda_2)$;
\item  if $k_s>1$ then
$L(k_{\dot\lg}k_s\mu)$ embeds in $L(k_{\dot\lg}k_s\lambda_1)\otimes L(k_{\dot\lg}k_s\lambda_2)$. 
\end{enumerate}

\end{theo}

Observe that, in type $A$, $k_{\dot\lg}k_s=1$. The case $A_1$ was
obtained before in \cite{BrownKumar}.

Let $\delta$ denote the fundamental imaginary root. 
We also obtain the following variation.

\begin{theo}
  \label{th:saturation2}
Let $(\lambda_1,\lambda_2,\mu)\in (P_+)^3$ such that there exists $N>0$
  such that $L(N\mu)$ embeds in $L(N\lambda_1)\otimes L(N\lambda_2)$. 
We also assume that $\mu-\lambda_1-\lambda_2\in Q$.

Then, for any integer $d\geq 2$, 
$L(k_{\dot\lg}k_s\mu-d\delta)$ embeds in $L(k_{\dot\lg}k_s\lambda_1)\otimes L(k_{\dot\lg}k_s\lambda_2)$. 
\end{theo}

\bigskip
In Section~\ref{sec:tech}, we collect some technical lemmas used in
the paper.  

\bigskip
{\bf Acknowledgements.} I am pleased to thank Michael Bulois, Stéphane
Gaussent, Philippe Gille, Kenji Iohara,
Nicolas Perrin, Bertrand Remy for useful discussions. 

The author is partially supported by the French National Agency
(Project GeoLie ANR-15-CE40-0012).

\tableofcontents
 
\section{Ind-varieties}

\subsection{Ind-varieties}

In this section, we collect definitions, notation and properties of
ind-varieties. The results are certainly well-known, but we include
some proofs for the convenience of the reader.  
 
\subsubsection{The category}

Let $(X_n)_{n\in \NN}$ be a sequence of quasiprojective complex varieties given
with closed immersions $\iota_n\,:\,X_n\to X_{n+1}$. Consider the
inductive limit $X=\indlim X_n$. 
A subset $F$ in $X$ is said to be Zariski closed if $F\cap X_n$ is
closed for any $n\in\NN$. 
A continuous map $f\,:\, X\longto Y=\indlim Y_n$ between two
ind-varieties is a {\it morphism} if for any $n\in\NN$ there exists
$m\in\NN$ such that $f(X_n)\subset Y_m$ and the restriction
$f\,:\,X_n\longto Y_m$ is a morphism. 
Let $X'_n\subset X$ be closed subsets such that $X=\cup_{n\in\NN}X'_n$
and $X'_0\subset X'_1\subset\cdots X'_n\subset\cdots$. 
Then $X'=\indlim X'_n$ is an ind-variety. The filtrations $(X_n)_{n\in \NN}$ and
$(X'_n)_{n\in \NN}$ are said to be {\it equivalent} if the identity maps $X\longto
X'$ and $X'\longto X$  are morphisms.   
Actually, the filtrations on ind-varieties are regarded up to
equivalence; an ind-variety $X$ endowed with a filtration 
$X_0 \subset X_1\subset\cdots X_n\subset\cdots$ by closed
subsets is called a {\it filtered ind-variety} and simply denoted by $X=\cup_{n\in\NN}X_n$.

\begin{lemma}
  \label{lem:indvarunicity}
Let $X=\cup_{n\in\NN} X_n$ be a filtered ind-variety. 
Assume, we have a family $(X'_n)_{n\in\NN}$ of closed subsets in $X$
such that $X'_n\subset X'_{n+1}$ and $X=\cup_{n\in\NN} X'_n$. 

Then the two filtrations $(X_n)_{n\in\NN}$ and $(X'_n) _{n\in\NN}$ are equivalent. 
\end{lemma}

\begin{proof}
  Consider an irreducible component $C$ of some $X_{n_0}$. Then
  $C=\cup_nX'_n\cap C$ and $X'_n\cap C$ is closed in $C$. 
Assume that for any $n$, $X'_n\cap C\neq C$. Then, for any $n$,
$\dim(X'_n\cap C)<\dim C$. Hence $C$ is the union of countably many
subvarieties of smaller dimension. This is a contradiction since we
are working on the uncountable field of complex numbers:
there exists $n_C$ such that $X'_{n_C}\cap C= C$.

Since $X_{n_0}$ has finitely many irreducible components, there exists
$N_0$ such that $X_{n_0}\subset X'_{N_0}$.

Observe that $X_n\cap X'_m$ is closed in $X$ and hence in $X'_m$. Then
the same proof as above shows that for any $n_1$, there exists $N_1$
such that $X'_{n_1}\subset X_{N_1}$.
\end{proof}

\bigskip
For $x\in X$, we denote by $T_xX$ the tangent space of $X$ at $x$. 
By definition $T_xX=\indlim T_xX_n$.

\subsection{Irreducibility}
\label{sec:indirred}

An ind-variety $X$ is said to be {\it irreducible} if it is as a topological
space. A poset is said to be {\it directed} if for any two elements
$x,y$ there exists $z$ bigger or equal to $x$ and $y$.
If the poset of irreducible components of the $X_n$'s is directed for
inclusion then $X$ is irreducible. Unless
\cite[Proposition~1]{Sha81}, the converse of this assertion is not
true (see \cite{Kambayashi96:JofAlg,Stampfli:JofAlge} for
examples). Here, the ind-variety $X$ is said to be {\it ind-irreducible} if 
 the poset of irreducible components of the $X_n$'s is directed for
inclusion. The following lemma gives a more precise definition.

\begin{lemma}
  \label{lem:indirred}
Let $X$ be an ind-variety. The following are
equivalent:
\begin{enumerate}
\item  there exists a filtration $X=\cup_{n\in \NN}X_n$ such that the poset of
  irreducible components of the $X_n$'s is directed;
\item for any equivalent filtration $X=\cup_{n\in \NN}X_n$, the poset of
  irreducible components of the $X_n$ is directed;
\item there exists an equivalent  filtration $X=\cup_{n\in \NN}X_n$ with $X_n$
  irreducible, for any $n$.
\end{enumerate}
If $X$ satisfies these properties then $X$ is said to be {\it ind-irreducible}.
\end{lemma}

\begin{proof}
  We prove $(ii)\Rightarrow (i)\Rightarrow (iii)\Rightarrow (ii)$. The
  first implication is tautological.

Show $(i)\Rightarrow (iii)$. Since $X_n$ has finitely many
irreducible components the assumption implies that 
$$
\forall n\quad\exists N \mbox{ and an irreducible component $C_N$ of
  $X_N$ such that } X_n\subset C_N.
$$ 
Then one can construct by induction an increasing sequence
$\varphi\,:\,\NN\longto\NN$ and irreducible components $C_{\varphi (n)}$ of
$X_{\varphi(n)}$ such that
\begin{eqnarray}
  \label{eq:7}
  \forall n\quad X_{\varphi(n)}\subset C_{\varphi (n+1)}\subset X_{\varphi(n+1)}.
\end{eqnarray}
Note that the $C_{\varphi (n)}$'s are closed, satisfy $C_{\varphi (n)}\subset C_{\varphi (n+1)}$ and
$X=\cup_{n\in \NN}C_{\varphi (n)}$. Moreover, by \eqref{eq:7}, the filtrations by
$X_n$'s and $C_{\varphi (n)}$'s are equivalent.

\bigskip
Show $(iii)\Rightarrow (ii)$.
Fix a filtration $X=\cup_{n\in\NN} X_n$ with irreducible closed subsets
$X_n$. Let $X=\cup_{n\in\NN} X'_n$ be an equivalent filtration. 
Consider two irreducible components $C_{n_1}$ and $C_{n_2}$ of some
$X'_{n_1}$ and $X'_{n_2}$. 
There exist
$N_1$ and $N_2$ in $\NN$ such that $C_{n_1}\cap C_{n_2}\subset
X_{N_1}\subset X'_{N_2}$.
Since $X_{N_1}$ is irreducible, there exists  an irreducible
component $C'_{N_2}$ of $X'_{N_2}$ containing $ X_{N_1}$.
Hence the poset of irreductible components of the $X'_n$ is directed.
\end{proof}

\noindent{\bf Examples.}
\begin{enumerate}
\item The simplest examples $\AA^{(\infty)}=\cup_{n\in\NN}\AA^n$, $\PP^{(\infty)}\cup_{n\in\NN}\PP^n$ of
  ind-varieties are ind-irreducible. 
\item A nonempty open subset of an ind-irreducible ind-variety is
  ind-irreducible. A product of two ind-irreducible ind-varieties is
  ind-irreducible.
\item\label{ex:image} Consider a surjective morphism $f\,:\,X\longto Y$
  of ind-varieties. If $X$ is ind-irreducible then so is $Y$. 
Indeed, let $X_n$ be a filtration of $X$ by irreducible subvarieties.
Denote by $Y_n$ the closure in $Y$ of $f(X_n)$. Then, by Lemma~\ref{lem:indvarunicity}, $Y=\cup_nY_n$ is
an equivalent  filtration of $Y$ by irreducible subvarieties.
Hence $Y$ is ind-irreducible. 
\item If $G$ is a Kac-Moody group and $P$ a standard parabolic
  subgroup then $G/P$ is a projective ind-variety. Indeed,
a filtration of $G/P$ is given by the unions of Schubert
varieties of bounded dimension. Since the Bruhat order is directed (see
e.g. \cite[Proposition~2.2.9]{BjoBre:Coxgrp}), $G/P$ is ind-irreducible.
\item The Birkhoff subvarieties of $G/P$ are  ind-varieties. Indeed the
  Richardson varieties are irreducible.
\item Let $G$ be the minimal Kac-Moody group as defined in
  \cite[Section~7.4]{Kumar:KacMoody}. Then $G$ is an ind-variety
  ind-irreducible. 
Indeed, for each real root $\beta$, denote by
$U_\beta\,:\,\CC\longto G$ the radicial subgroup. 
Consider an infinite word
  $\underline w=\beta_1,\dots,\beta_n,\dots$ in the real roots of $\lg$ such that
  any finite word in these roots is a subword of $\underline w$. Consider the map
$$
\begin{array}{cccl}
 \theta\,:& \AA^{(\infty)}\times T&\longto&G\\
&((\tau_i)_{i\in \NN},t)&\longmapsto&(\prod_iU_{\beta_i}(\tau_i))t
\end{array}
$$
Since $G$ is an ind-group and $U_\beta$ are morphism of ind-groups,
$\theta$ is a morphism of ind-varieties. By definition, it is
surjective. By Example~\eqref{ex:image}, $G$ is ind-irreducible.
\end{enumerate}

\bigskip
Another result we need, is the following.

\begin{lemma}
  \label{lem:filtration}
Let $X$ be an ind-variety ind-irreducible. Let $\Omega$ be a nonempty
open subset of $X$.    

Let $(Y_n)_{n\in\NN}$ be a collection of closed subsets of $X$ such that $Y_n\subset Y_{n+1}$ and
$\cup_{n\in\NN} Y_n$ contains $\Omega$.

Then $X=\cup_{n\in\NN} Y_n$ is an equivalent filtration of $X$.
\end{lemma}

\begin{proof}
  Fix a filtration $X=\cup_{n\in\NN}X_n$ by closed irreducible subsets of $X$ intersecting $\Omega$. 

Fix $n_0\in \NN$. Observe that $X_{n_0}\cap\Omega\subset\cup_{n\in
  \NN}(X_{n_0}\cap Y_n\cap \Omega)$. But  $X_{n_0}\cap Y_n\cap
\Omega$ is a locally closed subvariety of $X_{n_0}$. Consider the
sequence $n\mapsto \dim(X_{n_0}\cap Y_n\cap \Omega)$. It is
nondeacreasing. 

Assume that $\dim(X_{n_0}\cap Y_n\cap
\Omega)<\dim(X_{n_0}\cap\Omega)$, for any $n$. Then $X_{n_0}\cap\Omega$ is the
union of contable many strict subvarieties. This is a contration since  
 the base field is uncountable.
Hence there exists $N$ such that $\dim(X_{n_0}\cap Y_N\cap
\Omega)=\dim(X_{n_0}\cap\Omega)$. 

Then, $X_{n_0}\cap\Omega$ being irreducible, it is contained in
$Y_N$. Since $Y_N$ is closed, it follows that $X_{n_0}$ is
contained in $Y_N$.
In particular,  $X=\cup_nY_n$.
We conclude using Lemma~\ref{lem:indvarunicity}.
\end{proof}

\subsection{Line bundles}

Let $X=\cup_{n\in \NN} X_n$ be an ind-variety. Denote by $\iota_n\,:\,X_n\longto
X$ the inclusion. 
A {\it line bundle} $\Li$ over $X$ is an ind-variety with a morphism
$\pi\,:\,\Li\longto X$ such that $\iota_n^*(\Li)$ is a line bundle
over $X_n$, for any $n$.

A section of $\Li$ is a morphism $\sigma\,:\,X\longto \Li$ such that
$\pi\circ\sigma=\Id_X$. 
We denote by $H^0(X,\Li)$ the vector space of sections.
Given a section $\sigma$, we consider the sequence of sections
$(\sigma_n=\iota_n^*(\sigma))_{n\in\NN}$.  
Then 
\begin{equation}
  \label{eq:310}
{\sigma_{n+1}}_{|X_n}=\sigma_n.
 \end{equation} 
Conversely, a sequence $\sigma_n$ of sections of $\iota^*_n(\Li)$ on $X_n$ satisfying
condition~\eqref{eq:310} induces a well defined section $\sigma$ of $\Li$.

\section{Using Borel-Weil Theorem}

\subsection{Tensor multiplicities}

Recall that $\lg$ is a symmetrizable Kac-Moody Lie algebra. For given
$\lambda_1$ and $\lambda_2$ in $P_+$, $L(\lambda_1)\otimes
L(\lambda_2)$ decomposes as a sum of integrable irreducible highest weights modules (see
\cite[Corrolary 2.2.7]{Kumar:KacMoody}), with finite multiplicities:
$$
L(\lambda_1)\otimes
L(\lambda_2)=\oplus_{\mu\in P_+}L(\mu)^{\oplus c_{\lambda_1\,\lambda_2}^\mu}.
$$

Let $M$ be a $\lg$-module in the category $\Orb$; under the action of $\lh$, $M$
decomposes as $\oplus_\mu M_\mu$ with finite dimensional weight spaces
$M_\mu$. Set $M^\vee=\oplus_\mu M_\mu^*$: it is a sub-$\lg$-module of
the dual space $M^*$. 

\subsection{Multiplicities as dimensions}

Recall that  $G$ is the minimal Kac-Moody group for a given irreducible
symmetrizable GCM $A$. 
Let $B$ be the standard Borel subgroup of $G$ and $B^-$ be the
opposite Borel subgroup. Consider $G/B$ and $G/B^-$ endowed with the
usual ind-variety structures.
Let $\underline{o}=B/B$ (resp.  $\underline{o}^-=B^-/B^-$) denote the
base point of $G/B$ (resp. $G/B^-$).
For $\lambda\in\lh_\ZZ^*=\Hom(T,\CC^*)=\Hom(B,\CC^*)=\Hom(B^-,\CC^*)$,
we consider the $G$-linearized line bundle $\Li(\lambda)$ (resp. $\Li_-(\lambda)$) on
$G/B$ (resp. $G/B^-$) such that $B$ (resp. $B^-$) acts on the fiber
over $\underline{o}$ (resp. $\underline{o}^-$) with weight $-\lambda$
(resp. $\lambda$). For $\lambda\in P_+$, we have 
$G$-equivariant isomorphisms
(see \cite[Section~VIII.3]{Kumar:KacMoody})
$$
\begin{array}{lcl}
  \Ho^0(G/B,\Li(\lambda))&\simeq&\Hom(L(\lambda),\CC),\\
\Ho^0(G/B^-,\Li_-(\lambda))&\simeq&\Hom(L(\lambda)^\vee,\CC).
\end{array}
$$
Set 
$$
\dX=(G/B^-)^2\times G/B.
$$
A significant part of the following lemma is contained in \cite[Proof
of Theorem~3.2]{BrownKumar}.
\begin{lemma}\label{lem:BW}
  Let $\lambda_1,\,\lambda_2$, and $\mu$ in $P_+$. Then the
  space 
$$
\Ho^0(\dX,\Li_-(\lambda_1)\otimes \Li_-(\lambda_2)\otimes \Li(\mu))^G
$$
of $G$-invariant sections has dimension
$c_{\lambda_1\,\lambda_2}^\mu$. In particular, this dimension is finite.
\end{lemma}

\begin{proof}
Set $\Li=\Li_-(\lambda_1)\otimes
  \Li_-(\lambda_2)\otimes\Li(\mu)$.
 We have the following canonical isomorphisms:
$$
\begin{array}{r@{\,}cll}
  \Ho^0(\dX,\Li)^G&
\simeq&\Hom(L(\lambda_1)^\vee\otimes L(\lambda_2)^\vee\otimes
                      L(\mu),\CC)^G&\\
 &\simeq&\Hom(
                      L(\mu),(L(\lambda_1)^\vee\otimes L(\lambda_2)^\vee)^*)^G&\\
 &\simeq&\Hom(
                      L(\mu),(L(\lambda_1)^\vee\otimes L(\lambda_2)^\vee)^\vee)^G&\mbox{by $\lh$-invariance}\\
 &\simeq&\Hom(
                      L(\mu),L(\lambda_1)\otimes L(\lambda_2))^G&
\end{array}
$$
 Thus this space of invariant sections has dimension
 $c_{\lambda_1\,\lambda_2}^\mu$.
We already mentioned that $c_{\lambda_1\,\lambda_2}^\mu$ is finite. 
Nevertheless, we prove independently that $\Ho^0(\dX,\Li)^G$ is finite
dimensional, reproving that $c_{\lambda_1\,\lambda_2}^\mu$ is finite.

Consider the $T$-equivariant map
  $\iota\,:\,G/B^-\longto \dX, x\longmapsto (\underline{o}^-,x,\underline{o})$. Then
  $\iota^*(\Li)$ is a $T$-linearized line bundle on $G/B^-$. Consider
$$
\iota^*\,:\,\Ho^0(\dX,\Li)\longto \Ho^0(G/B^-,\iota^*(\Li)).
$$
Since $G.(\underline{o}^-,\underline{o})$ is dense in $G/B^-\times
G/B$, the restriction of $\iota^*$
to $\Ho^0(\dX,\Li)^G$ is injective. Since $\iota$ is $T$-equivariant, $\iota^*(\Ho^0(\dX,\Li)^G)$
is contained in $\Ho^0(G/B^-,\iota^*(\Li))^T$. 
But $\iota^*(\Li)\simeq\Li_-(\lambda_2)\otimes(\lambda_1-\mu)$. 
Then 
$$
\begin{array}{l@{\,}l}
\Ho^0(G/B^-,\iota^*(\Li))^T&\simeq \Ho^0(G/B^-,\Li_-(\lambda_2))^{(T)_{\mu-\lambda_1}}\\[0.5em]
&\simeq\Hom(L(\lambda_2)^\vee,\CC)^{(T)_{\mu-\lambda_1}}\\[0.7em]
&\simeq L(\lambda_2)^{(T)_{\mu-\lambda_1}}
\end{array}
$$
Since $L(\lambda_2)$ belongs to the category ${\mathcal O}$, the
dimension of $L(\lambda_2)^{(T)_{\mu-\lambda_1}}$ is finite.
It follows that $\iota^*$ embeds $\Ho^0(\dX,\Li)^G$ in a finite
dimensional vector space.
\end{proof}

\section{Enumerative meaning of structure constants of
  \texorpdfstring{$H^*(G/P,\ZZ)$}{cohomology rings}}

\subsection{Richardson varieties}

Let $U$ be the usual ind-subgroup of $B$.
Let $P$ be a standard parabolic subgroup of $G$ and $X=G/P$ be the flag ind-variety.

For $u,v\in W^P$, 
set $X^u_\GP=\overline{B^-uP/P}$ and $X_v^\GP =\overline{BvP/P}$. 
Set also $\cX^u_\GP=B^-uP/P$ and $\cX_v^\GP=BvP/P$.
We sometimes forget the $G/P$ if there is no risk of confusion.
For $u,v\in W^P$, we denote by $u\bleq v$ the Bruhat order if
$X_u^\GP\subset X_v^\GP$.  
In the following lemma, we collect some well known facts about the
Schubert and Richardson varieties.

\begin{lemma}\label{lem:rappelsRichardson}
  \begin{enumerate}
  \item $X_v^\GP$ is a projective variety of dimension $l(v)$ and $\Aut(X_v^\GP)^\circ$ is an algebraic
    group of finite dimension.
\item The image of $U$ in $\Aut(X_v^\GP)^\circ$ is an unipotent group denoted by
  $U_v$.
\item The intersection $X^u_v=X^u_\GP\cap X_v^\GP$ is an irreducible
  closed normal subvariety (called Richardson variety) of $X_v^\GP$ of dimension
  $l(v)-l(u)$, if $u\bleq v$. It is empty otherwise.
\item  If $u\bleq v$ then $\cX^u_\GP\cap \cX_v^\GP$ is a nonempty open subset
 contained in the smooth locus of $X^u_v$. 
\item Let $u\bleq v$ such that $l(v)=l(u)+1$. Then the Richardson
  variety $X^u_v$ is isomorphic to $\PP^1$ and  $\cX^u_\GP\cap
  \cX_v^\GP$ is isomorphic to $\CC^*$.
\item Assume that $u\bleq v$ and $x\in\cX^u_\GP\cap \cX_v^\GP$. Then
  the sequence induced by the inclusions

\tikzset{square matrix/.style={outer sep=-6pt,matrix of nodes,
nodes={draw,
minimum height=0.7cm,
anchor=center,
text width=0.7cm,
align=center,
outer sep=-6pt}}}

\begin{center}
\begin{tikzpicture}
  \matrix (m) [matrix of math nodes,row sep=1.5em,column sep=1.7em,minimum width=2em,anchor=center]
  {
     0&T_x(\cX^u_\GP\cap \cX_v^\GP)&
T_x \cX_v^\GP&\frac{T_x\GP}{T_x\cX^u_\GP}&0\\};

\draw[->] (m-1-5-|m-1-3.east) -- (m-1-5-|m-1-4.west);
\draw[->] (m-1-5-|m-1-1.east) -- (m-1-5-|m-1-2.west);
\draw[->] (m-1-5-|m-1-2.east) -- (m-1-5-|m-1-3.west);
\draw[->] (m-1-5-|m-1-4.east) -- (m-1-5-|m-1-5.west);
\end{tikzpicture}  
\end{center} 
is exact.

 \end{enumerate}
\end{lemma}

\begin{proof}
The normality of Richardson's varieties is proved in \cite[Proposition~6.5]{Kumar:positivity}).
The last assertion is an easy consequence of
\cite[Lemma~7.3.10]{Kumar:KacMoody}).
The others assertions are banal (see \cite{Kumar:KacMoody}).
\end{proof}

\bigskip
Consider the homology group $\Ho_*(X,\ZZ)=\oplus_{v\in W^P}\ZZ[X_v^\GP]$. Then
$\Ho^*(X,\ZZ)\simeq\Hom(\Ho_*(X,\ZZ),\ZZ)$ has a ``basis'' $(\epsilon_u)_{u\in
  W^P}$ defined by
$$
\epsilon_u([X_v^\GP])=\delta^u_v,\qquad\forall v\in W^P.
$$
For $u_1,\,u_2,$ and $v\in W^P$, define $n_{u_1u_2}^v\in \ZZ$ by
$$
\epsilon_{u_1}.\epsilon_{u_2}=\sum_{v\in W^P} n_{u_1u_2}^v\epsilon_v.
$$
By \cite{KumarNori}, $n_{u_1u_2}^v\geq 0$. 

\subsection{Kleiman's lemma}

\begin{lemma}
\label{lem:Kl}  
Let $u_1,u_2$, and $v$ in $W^P$ such that $u_1\bleq v$ and $u_2\bleq v$. Assume
  that $l(v)=l(u_1)+l(u_2)$.

For general $h\in U_v$, $X^{u_1}_v\cap hX_v^{u_2}=\cX^{u_1}_\GP\cap
h(\cX^{u_2}_\GP\cap \cX_v^\GP)$ is finite and transverse.
More precisely, for any $x\in X^{u_1}_v\cap hX_v^{u_2}$ the following map
induced by inclusions
$$
T_xX_v^\GP\longto \frac{T_x\GP}{T_xX^{u_1}_\GP}\oplus
\frac{T_x\GP}{T_x\tilde hX^{u_2}_\GP}
$$
is an isomorphism, where  $\tilde h\in U$ satisfies $r(\tilde h)=h$. 
\end{lemma}

\begin{proof}
  It is an application of Kleiman's theorem.
Indeed, consider $X_v=\cup_{\sigma\bleq v} \cX_\sigma$.

Fix $\sigma\bleq v$ such that $\sigma\neq v$.
For any $h\in U_v$, we have
$X^{u_1}_v\cap hX_v^{u_2}\cap X_\sigma=X^{u_1}_\sigma\cap hX_\sigma^{u_2}$.

Assume first that $u_1\bleq \sigma$ and $u_2\bleq \sigma$.  Since
$l(\sigma)<l(v)$, $(\dim X_\sigma-\dim X^{u_1}_\sigma) + (\dim
X_\sigma-\dim X^{u_2}_\sigma)>\dim X_\sigma$. Kleiman's
theorem applied in the $U_v$-homogeneous space $\cX_\sigma$ shows
that  $X^{u_1}_v\cap hX_v^{u_2}\cap \cX_\sigma$ is empty for general
$h\in U_v$.

Otherwise, $X^{u_1}_\sigma$ or $X_\sigma^{u_2}$ is empty.

Similarly, for general $h\in U_v$, $\partial X^{u_1}\cap h X^{u_2}_v$
and $X^{u_1}\cap h (\partial X^{u_2}\cap X_v)$ is empty. Here
$\partial X^{u}=X^{u}-\cX^{u}$.
Indeed, only finitely many $B^-$-orbits $\cX^{u_1}\subset X^u$ intersect $X_v$.

Hence, for general $h\in U_v$, we have $X^{u_1}_v\cap hX_v^{u_2}=\cX^{u_1}\cap
h(\cX^{u_2}\cap \cX_v)$.\\

Now, by Kleiman's theorem in the
$U_v$-homogeneous space $\cX_v$, for general $h\in U_v$, $X^{u_1}_v\cap hX_v^{u_2}=\cX^{u_1}_\GP\cap
h(\cX^{u_2}_\GP\cap \cX_v^\GP)$ is finite and for any $x\in X^{u_1}_v\cap hX_v^{u_2}$ the map
$$
T_xX_v^\GP\longto \frac{T_xX_v^\GP}{T_xX^{u_1}_v}\oplus \frac{T_xX_v^\GP}{T_xhX^{u_2}_v}
$$
is an isomorphism.
Since $x\in \cX_v\cap\cX^{u_1}$, Lemma~\ref{lem:rappelsRichardson}
implies that the natural map $\frac{T_xX_v^\GP}{T_xX^{u_1}_v}\longto
\frac{T_x\GP}{T_xX^{u_1}_\GP}$ is an isomorphisms.
Similarly, 
$\frac{T_{h\inv x}X_v^\GP}{T_{h\inv x}X^{u_2}_v}\longto
\frac{T_{h\inv x}\GP}{T_{h\inv x}X^{u_2}_\GP}$ is an isomorphism.
Since both $\GP$ and $X_v$ are $\tilde h$-stable, by applying $\tilde
h$, one deduces that
$\frac{T_xX_v^\GP}{T_xhX^{u_2}_v}\longto
\frac{T_x\GP}{T_x\tilde hX^{u_2}_\GP}$ is an  isomorphism.
\end{proof}

\begin{lemma}
  \label{lem:cardcohom}
Let $h\in U_v$ satisfying  Lemma~\ref{lem:Kl}. 
Then
$$
\sharp (X^{u_1}_v\cap hX_v^{u_2})=n_{u_1u_2}^v.
$$
\end{lemma}

\begin{proof}
We first prove that one may assume that $B=P$.
Consider the projection $\pi\,:\,G/B\longto G/P$ and the associated
morphism $\pi^*\,:\,{\operatorname H}^*(G/P,\ZZ)\longto {\operatorname
  H}^*(G/B,\ZZ)$ in cohomology. Then, for any $u\in W^P$, 
$\pi^*(\epsilon_u(G/P))=\epsilon_u(G/B)$. 
In particular, $n_{u_1u_2}^v(G/B)=n_{u_1u_2}^v(G/P)$, for any $u_1$,
$u_2$ and $v$ in $W^P$.

Note that, for any $u\in W^P$, $X^u_\GB=\pi\inv (X^u_\GP)$ and $\pi$
maps $\cX_u^\GB$ bijectively onto $\cX_u^\GB$.
Then, for  any $h$  in $U_v$, $\pi$ maps $X^{u_1}_\GB\cap
h(X_v^\GB\cap X^{u_2}_\GB)$ bijectively onto $X^{u_1}_\GP\cap h(X_v^\GP\cap X^{u_2}_\GP)$.
In particular, if the lemma holds for $G/B$ it holds for $G/P$.

\bigskip
Let $\tilde G$ be the  Kac-Moody group completed
along
 the negative roots (as opposed to completed along the positive
 roots).  
Let $\tilde B^-$ be the Borel subgroup of $\tilde G$.
Let
$
\tilde{X} = \tilde G/B
$
be the `thick' flag variety which contains the standard KM-flag
variety $X=G/B$.
If $G$ is not of finite type, $\tilde{X}$ is an infinite
dimensional  non quasi-compact  scheme (cf. [Ka, \S4]). 
For $w\in W$, denote by $\tilde X^w$ the closure of $\tilde B^-w\underline
o$ in $\tilde X$. 
Observe that $\tilde X^w\cap X=X^w$. 
Let $K^0(\tilde{X})$ denote the Grothendieck group of
 coherent $\co_{\tilde{X}}$-modules  (see \cite[\S3.5]{BrownKumar}
 for details).
 Similarly,
define $K_0(X) := \lim_{n\to\infty} K_0(X_n)$, where $\{
X_n\}_{n\geq 1}$ is the filtration of $X$ and
$K_0(X_n)$ is the Grothendieck group of  coherent
sheaves on the projective variety $X_n$.
Then, $\bigl\{ [\co_{X_w}]\bigr\}_{w\in W}$ is a basis of $K_0(X)$ as a $\ZZ$-module.
Define a pairing
$$
\langle \, ,\, \rangle : K^0(\tilde{X}) \otimes K_0(X) \to \ZZ,\,\,
\langle[\cs], [\cf]\rangle  = \sum_i (-1)^i \chi \bigl (X_n, \tor_
i^{\co_{\tilde{X}}}
 (\cs,\cf ) \bigr),$$
if $\cs$ is a coherent sheaf on $\tilde{X}$ and $\cf$
is a coherent sheaf on ${X}$ supported in $X_n$ (for
some $n$), where $\chi$ denotes the 
Euler-Poincar\'{e} characteristic.
Let $\varpi_{\alpha_0},\dots, \varpi_{\alpha_l}$ be characters of $T$
dual of the coroots $\alpha_0^\vee,\dots,\alpha_l^\vee$.
Set $\rho=\sum_{i=0}^l\varpi_{\alpha_i}$.
Set the  sheaf on $\tilde{X}$ (see~\cite[Theorem~10.4]{Kumar:positivity}
$$
\xi^u = \Li(-\rho ) \ext^{\ell (u)}_{\co_{\tilde{X}}}
(\co_{X^u}, \co_{\tilde{X}} )=\Orb_{\tilde X^u}(-\partial \tilde X^u),
$$
where $\partial \tilde X^u=\tilde X^u-\tilde B^-uB/B$. 
Then,  as proved in  \cite[Proposition~3.5]{Kumar:positivity},  for any $u,w\in W$,
$$
\langle[\xi^u], [\co_{X_w}]\rangle = \delta_{u,w}.
$$

Let $\Delta: X \to X\times X$ be the diagonal map. Then, by
\cite[Proposition 4.1]{Kumar:positivity} and \cite[\S3.5]{BrownKumar},
for any $g_1,\,g_2\in G$
\begin{equation}
  \label{eq:nchi}
 n^w_{u_1\,u_2} 
 = \sum_i (-1)^i \chi (\tilde{X}\times \tilde{X}, \tor_i^{\co_{\tilde{X}\times \tilde{X}}} \Bigl(\xi^{u_1}\boxtimes\xi^{u_2},(g_1^{-1}, g_2^{-1}) \cdot \Delta_* (\co_{X_v})\Bigr) 
\end{equation}
Let $\tilde h$ in $U$ such that $r(\tilde h)=h$.

The support of  
$\tor_i^{\co_{\tilde{X}\times \tilde{X}}}
\Bigl(\xi^{u_1}\boxtimes\xi^{u_2},(1,\tilde h\inv) \cdot \Delta_*
(\co_{X_v})\Bigr) $
is contained in $(\tilde X^{u_1}\times \tilde X^{u_2})\cap
(1,\tilde h\inv)\Delta(X_v)$. The assumptions on $\tilde h$ implies that
this support is contained in 
$(\cX^{u_1}\cap \cX_v)\times  (\cX^{u_2}\cap \cX_v)$.
In particular, this $\tor$-sheaf is equal to 
\begin{equation}
  \label{eq:tori}
\tor_i^{\co_{\tilde{X}\times
     \tilde{X}}}
 \Bigl(\Orb_{{X}^{u_1}}\boxtimes\Orb_{{X}^{u_2}},(1,
     h\inv) \cdot \Delta_* (\co_{X_v})\Bigr). 
\end{equation}

The support of the $\tor$-sheaves in formula~\eqref{eq:tori} are
contained in $u_1\tilde B^-\underline o\times
u_2\tilde B^-\underline o$.
By \cite[Section~8]{KS:affflagman}, there exists an isomorphism 
$\iota\,:\,u_1\tilde B^-\underline o\times u_2\tilde B^-\underline o\longto
\AA^\infty=\CC^\NN$ such that 
$\tilde B^-u_1\underline{o}\times \tilde B^-u_2\underline{o}$ maps onto $\CC^{\NN_{\geq
    l(u_1)+l(u_2)}}$. Here $\CC^\NN$ denote the set $\CC$-valued
sequences view as $\Spec(\CC[T_0,\dots,T_n,\dots])$ and  $\CC^{\NN_{\geq
    l(u_1)+l(u_2)}}$ is the set of sequences starting with $
l(u_1)+l(u_2)$ zeros. 
We also set $\CC^{\NN_{\leq m}}:=\{(u_k)\in\CC^\NN\,:\,u_k=0\ \forall
k>m\}$. 
Then, there exists $m\geq l(u_1)+l(u_2)$ such that $(u_1\tilde B^-\underline
o\times u_2\tilde B^-\underline o)\cap (X_v\times X_v)$ is contained in
$\CC^{\NN_{\leq m}}$.
Now, for any $i\geq 0$,
$$
\tor_i^{\co_{\tilde{X}\times
     \tilde{X}}}
 \Bigl(\Orb_{\tilde{\cX}_{u_1}}\boxtimes\Orb_{\tilde{\cX}_{u_2}},(1,
     h\inv) \cdot \Delta_* (\co_{X_v})\Bigr)
$$
is the pullback of 
$$
\tor_i^{\Orb_{\CC^{\NN_{\leq m}}}}
 \Bigl(\Orb_{\CC^{\NN_{\leq m}}\cap    \CC^{\NN_{\geq
    l(u_1)+l(u_2)}}  },((1,
     h\inv) \cdot (\Delta_* (\co_{X_v}))_{|\CC^{\NN_{\leq m}}}\Bigr).
$$
Since the intersection $\iota((1,h\inv )\Delta(X_v)\cap (u_1\tilde B^-\underline o\times
u_2\tilde B^-\underline o))\cap (\CC^{\NN_{\leq m}}\cap    \CC^{\NN_{\geq
    l(u_1)+l(u_2)}})$ is transverse in $\CC^{\NN_{\leq m}}$ it follows
that these $\tor$-sheaves vanish for $i\geq 1$ and that
$n^w_{u_1\,u_2}$ is the cardinality of this intersection. The lemma is
proved.
\end{proof}

\subsection{The case \texorpdfstring{$n_{u_1\,u_2}^v=1$}{n=1}}
\begin{lemma}
  \label{lem:intermorph}
We keep notation and assumptions of Lemma~\ref{lem:Kl}. Moreover we
assume that $n_{u_1,\,u_2}^v=1$.

Then there exist a non empty open subset $\Omega$ of $U_v$ and a
regular map
$$
\psi\,:\,\Omega\longto \cX_v^\GP
$$
such that
$$
\forall h\in \Omega\qquad X_v^{u_1}\cap h  X_v^{u_2}=\{\psi(h)\},
$$
and 
$$
\psi(h)\in \cX^{u_1}_\GP\cap h(\cX_\GP^{u_2}\cap \cX_v^\GP).
$$
\end{lemma}

\begin{proof}
  Consider 
$$
\Chi=\{(x,h)\in \cX_v\times U_v\,:\,x\in X^{u_1}\cap h  X_v^{u_2}\},
$$
and the two projections $p\,:\,\Chi\longto X^{u_1}\cap \cX_v$ and
$\pi\,:\,\Chi\longto U_v$. 

\bigskip
We first prove that $\Chi$ is irreducible. 

Fix $x_0\in X^{u_1}\cap \cX_v$ and set $\orb\,:\,U_v\longto \cX_v,
h\longto hx$. 
The stabilizer of $x$ in the unipotent group
$U_v$ is unipotent, and hence connected.
Hence the fibers of $\orb$ are irreducible. 
Since  $X^{u_2}\cap\cX_v$
is irreducible, so is $\orb\inv(X^{u_2}\cap\cX_v)$.

Moreover for any $h\in U_v$, $(x_0,h)\in p\inv(x_0)$ if and only if $x_0\in hX^{u_2}$   if and only if 
$h\inv\in \orb\inv(X^{u_2}\cap \cX_v)$.
Hence $p\inv(x_0)$ is nonempty and irreducible. 

Since $X^{u_1}\cap \cX_v$ is also irreducible, we deduce that $\Chi$ is. 

\bigskip
By Lemmas~\ref{lem:Kl} and \ref{lem:cardcohom}, the general fiber of $\pi$ is a singleton. Since
we are working over the complex numbers, this implies that $\pi$ is
birational. Then, a partial converse map $\psi$ of $\pi$ satisfies the
lemma (at least, its restriction to an open subset of $h\in U_v$
satisfying Lemma~\ref{lem:Kl}).
\end{proof}

\section{Inequalities for \texorpdfstring{$\Gamma(A)$}{the tensor cone}}

In this section, we reprove \cite[Theorem~1.1]{BrownKumar} by similar methods in
the goal to introduce some useful notation. 
Fix once for all, a family
$\varpi_{\alpha_0^\vee},\dots,\varpi_{\alpha_l^\vee}$ in $\lh_\QQ$ such that
$$
\langle \varpi_{\alpha_i^\vee},\alpha_j\rangle=\delta_i^j,
$$
for any $i,j\in\{0,\dots,l\}$.
Similarly fix fundamental weights $\varpi_{\alpha_0},\dots,\varpi_{\alpha_l}$.

Let $\tau\,:\,\CC^*\longto G$ be a morphism of ind-groups. 
Let $\Li$ be a $G$-linearized line bundle on $\dX$ and $x\in\dX$. 
Since $\dX$ is ind-projective and the action of $\CC^*$ on $\dX$ is
given by a morphism of ind-varieties, $\lim_{t\to 0}\tau(t)x$
exists (i.e. the morphism $\CC^*\longto\dX,\,t\longmapsto \tau(t)x$
extends to $\CC$). Set $z=\lim_{t\to 0}\tau(t)x$. 
The point $z$ is fixed by $\tau(\CC^*)$ and $\tau(\CC^*)$ acts linearly on the
fiber $\Li_z$. There exists $m\in\ZZ$ such that
$$
\forall t\in\CC^*\qquad\forall \tilde z\in\Li_z\qquad \tau(t)\tilde
z=t^m\tilde z.
$$
Set $\mu^\Li(x,\tau)=-m$. 

\begin{prop}(see \cite[Theorem~1.1]{BrownKumar})
  \label{prop:ineg}
Let $P$ be a standard parabolic subgroup of $G$. Let $\alpha_i$ be a
simple root that does not belong to $\Delta(P)$. 
Let $u_1,u_2$, and $v$
in $W^P$ such that $n_{u_1,u_2}^v\neq 0$ in $\Ho^*(G/P,\ZZ)$.
  
If  $(\lambda_1,\lambda_2,\mu)\in \Gamma(A)$ then
\begin{eqnarray}
  \label{eq:4}
 \langle\mu,v\varpi_{\alpha_i^\vee}\rangle\leq  \langle\lambda_1,u_1\varpi_{\alpha_i^\vee}\rangle+\langle\lambda_2,u_2\varpi_{\alpha_i^\vee}\rangle.
\end{eqnarray}
\end{prop}

\begin{proof}
 Consider $C^+=Pu_1\inv \underline{o}^-\times Pu_2\inv \underline{o}^-\times Pv\inv
  \underline{o}$. As a locally closed subset of  $\dX$, it is an ind-variety. Set
$$
\GPCp:=\{(gP/P,x)\in G/P\times \dX\,:\, g\inv x\in  C^+\}.
$$
As a locally closed subset of  $G/P\times \dX$, it is an ind-variety.
Consider the maps
$$
\begin{array}{cccc}
  \eta\,:&\GPCp&\longto&\dX\\
&(gP/P,x)&\longmapsto&x
\end{array}
$$ 
and 
$$
\begin{array}{cccc}
  p\,:&\GPCp&\longto&G/P\\
&(gP/P,x)&\longmapsto&gP/P.
\end{array}
$$

\begin{lemma}
\label{lem:fibreeta}
Let $g_1,g_2$, and $g_3$ in $G$.
Then 
$$
p(\eta\inv
(g_1\underline{o}^-, g_2\underline{o}^-,g_3\underline{o}))=
g_1\cX^{u_1}_\GP\cap g_2 \cX^{u_2}_\GP\cap g_3 \cX_v^\GP.   
$$
\end{lemma}

\begin{proof}
The point $(gP/P,(g_1\underline{o}^-, g_2\underline{o}^-,g_3\underline{o})
)$ belongs to the fiber of $\eta$ if and only if
$$
\begin{array}{ll}
  &(g\inv g_1\underline{o}^-,g\inv  g_2\underline{o}^-,g\inv g_3\underline{o})\mbox{ belongs
    to}\\
& Pu_1\inv \underline{o}^-\times Pu_2\inv
\underline{o}^-\times Pv\inv \underline{o}\\[3pt]
\iff&(g\inv g_1,g\inv  g_2,g\inv g_3)\in Pu_1\inv B^-\times Pu_2\inv
B^-\times Pv\inv B\\[3pt]
\iff&g_1\inv g\in B^-u_1P,\, g_2\inv g\in
      B^-u_2P\mbox{ and }g_3\inv g\in BvP\\[3pt]
\iff&gP/P\in g_1B^-u_1P/P\cap
      g_2B^-u_2P/P\cap g_3BvP/P.
\end{array}
$$
\end{proof}

\bigskip
 Consider the morphism of ind-varieties
$r\,:\, U\longto U_v\subset\Aut(X_v^\GP)$ given by the action.  
Set $\mathring{G/B^-}=U\underline{o}^-$. Since $U\longto G/B,\,u\longmapsto
u\underline{o}^-$ is an open immersion, we have a morphism
$p^+\,:\,\mathring{G/B^-}\longto U$ such that
$p^+(x)\underline{o}^-=x$. Similarly, define  $\mathring{G/B}=U^-\underline{o}$ and
$p^-\,:\,\mathring{G/B}\longto U^-$.

Let $\Omega\subset U_v$ be a nonempty open subset of $h$'s satisfying
Lemma~\ref{lem:Kl}. 
Set
$$
\Omega_1=
\left\{
\begin{array}{ll}
 (x_1,x_2,g_3\underline{o})\in \dX\,: &g_3\inv x_1\in\mathring{G/B^-},
g_3\inv x_2\in\mathring{G/B^-}\mbox{ and }\\
&r(p^+( g_3\inv x_1)\inv p^+(g_3\inv x_2))\in\Omega
\end{array}
\right\}.$$
It is open and nonempty in $\dX$. Moreover, for
$(x_1,x_2,g_3\underline{o})\in\Omega_1$, Lemma~\ref{lem:fibreeta}
implies that 
\begin{equation}
  \label{eq:413}
  p(\eta\inv(x_1,x_2,g_3\underline{o}))=g_3 p^+( g_3\inv x_1)[\cX^{u_1}\cap
h(\cX^{u_2}\cap \cX_v)],
\end{equation}
where
$h=r(p^+( g_3\inv x_1)\inv p^+(g_3\inv x_2))$.
By Lemma~\ref{lem:cardcohom} this fiber is nonempty. 

\bigskip
Let $\Li$ be the line bundle on $\dX$ considered in Lemma~\ref{lem:BW}.
Since $(\lambda_1,\lambda_2,\mu)$ belongs to  $\Gamma(A)$, there exist $N>1$ and
$\sigma\in \Ho^0(\dX,\Li^{\otimes N})^G$ nonzero. 
Since $\dX$ is irreducible, $\Omega_1$ intersects the nonzero locus of
$\sigma$: there exists $\dx\in\dX$ such that $\sigma(\dx)\neq
0$. Since the fiber~\eqref{eq:413} is not empty, there exists $g\in G$
such that $(gP/P,\dx)$ belongs to $\GPCp$. 
Set $y=g\inv \dx$. 
Since $\sigma$ is $G$-invariant, $\sigma(y)\neq 0$. 

Let $\tau$ be a one parameter subgroup of $T$ belonging to $\oplus_{\alpha_j\not\in\Delta(P)}\ZZ_{>0}\varpi_{\alpha_j^\vee}$.
Consider 
$$
\theta\,:\,\AA^1\longto X,\, t\in \CC^*\longmapsto
\tau(t)y,\,0\longmapsto\lim_{t\to 0}\tau(t)y.
$$ 
Then $\theta^*(\sigma)$ is a $\CC^*$-invariant section of
$\theta^*(\Li)$ on $\AA^1$. It follows that $\mu^\Li(y,\tau)\leq 0$. 

Let $L$ denote the standard Levi subgroup of $P$. Set
$C=Lu_1\inv \underline{o}^-\times Lu_2\inv \underline{o}^-\times Lv\inv \underline{o}$. 
Since $y\in C^+$, $\theta(0)$ belongs to $C$.  
Then, a direct computation shows that 
\begin{eqnarray}
  \label{eq:6}
\mu^\Li(y,\tau)=-\langle\lambda_1,u_1\tau\rangle-\langle\lambda_2,u_2\tau\rangle+
\langle\mu,v\tau\rangle\leq 0.
\end{eqnarray}

Since inequality~\eqref{eq:6} is fulfilled for any sufficiently large $\tau\in
\oplus_{\alpha_j\not\in\Delta(P)}\ZZ_{>0}\varpi_{\alpha_j^\vee}$, the
inequality of the theorem follows by continuity.
\end{proof}

\begin{NB}
  We use notation of the proposition. Since
  $n_{u_1\,u_2}^v(G/P)=n_{u_1\,u_2}^v(G/B)$, Proposition~\ref{prop:ineg}
implies that inequality~\ref{eq:4} is fulfilled for any simple root
$\alpha_i$, even in $\Delta(P)$. 
\end{NB}

\section{Multiplicities on the boundary}
\label{sec:boundary}

In this section, we are interested in the multiplicity
$c_{\lambda_1\,\lambda_2}^\mu$ for some triple
$(\lambda_1,\lambda_2,\mu)$ of dominant weights such that
inequality~\eqref{eq:4} is an equality. If moreover
$n_{u_1\,u_2}^v=1$, we prove that $c_{\lambda_1\,\lambda_2}^\mu$  is
a multiplicity for the tensor product decomposition
for some strict Levi subgroup of $G$.

\begin{theo}\label{th:egalite}
 We use notation of Proposition~\ref{prop:ineg} and assume, in addition, that
 $n_{u_1,u_2}^v=1$. Let $L$ be the standard Levi subgroup of $P$.
Let $\tau\in \oplus_{\alpha_j\not\in\Delta(P)}\ZZ_{>0}\varpi_{\alpha_j^\vee}$. 

Let $(\lambda_1,\lambda_2,\mu)\in (P_+)^3$ such that
\begin{equation}
  \label{eq:egalite}
\langle\lambda_1,u_1\tau\rangle+\langle\lambda_2,u_2\tau\rangle=
\langle\mu,v\tau\rangle.
\end{equation}
Consider the line bundle
$$
\Li=\Li_-(\lambda_1)\otimes \Li_-(\lambda_2)\otimes \Li(\mu)
$$ on $\dX$, and 
$$
C=Lu_1\inv \underline{o}^-\times Lu_2\inv
\underline{o}^-\times Lv\inv \underline{o}$$ be the closed subset of $\dX$.

Then, the restriction map induces an isomorphism
\begin{eqnarray}
  \label{eq:5}
  \Ho^0(\dX,\Li)^G\longto  \Ho^0(C,\Li_{|C})^L.
\end{eqnarray}
\end{theo}

\bigskip
Before proving the theorem, we state a consequence.
Set $P_+(L)=\{\bar\lambda\in
X(T)\,:\,\langle\bar\lambda,\alpha_i^\vee\rangle\geq 0\quad\forall
\alpha_i\in\Delta(L)\}$. 
For any $\bar\lambda\in P_+(L)$, we have an irreducible $L$-module
$L(\bar\lambda)$ of highest weight $\bar\lambda$. 
Let ${\bar c}_{\bar\lambda_1\,\bar\lambda_2}^{\bar\mu}$ denote the
multiplicities of the tensor product decomposition for the group $L$.

\begin{coro}
  \label{cor:G2L}
With the notation and assumptions of the theorem, set
$\bar\lambda_1=u_1\inv\lambda_1$,  $\bar\lambda_2=u_2\inv\lambda_2$
and $\bar\mu=v\inv\mu$.
We have
$$
{c}_{\lambda_1\,\lambda_2}^{\mu}={\bar c}_{\bar\lambda_1\,\bar\lambda_2}^{\bar\mu}.
$$
\end{coro}

\begin{proof}
Since $u_i\in W^P$, $u_i\inv B^-u_i \cap L=B^-\cap L$. 
Similarly, $v\inv Bv \cap L=B\cap L$. Then, the action of $L^3$ on $C$ allows
to identify $C$ with $\dX_L:=(L/(B^-\cap L))^2\times L /(B\cap L)$. 
But $T$ acts with weight $u_1\inv\lambda_1$ on the fiber in
$\Li_-(\lambda_1)$ over $u_1\inv \underline{o}^-$. We deduce that $\Li_{|C}$
identifies with $\Li_-(\bar\lambda_1)\otimes
\Li_-(\bar\lambda_2)\otimes\Li(\bar\mu)$ on $\dX_L$. 
Now the corollary is obtained by applying Theorem~\ref{th:egalite} and Lemma~\ref{lem:BW}.
\end{proof}

Note that Theorem~\ref{th:reductionintro} in the introduction is a
particular case of Corollary~\ref{cor:G2L}.

\begin{proof}
[Proof of Theorem~\ref{th:egalite} up to the 5 lemmas below]
We use notation of Proposition~\ref{prop:ineg}. Let $\bar C^+$ be
  the closure of $C^+$ in $\dX$. Set
$$
\GPCpbar:=\{(gP/P,\dx)\in G/P\times \dX\,:\, g\inv \dx\in \bar C^+\}.
$$
As a closed subset of  $G/P\times \dX$, it is an ind-variety.
Consider the maps
$$
\begin{array}{cccc}
  \bar \eta\,:&\GPCpbar&\longto&\dX\\
&(gP/P,\dx)&\longmapsto&\dx
\end{array}
$$ 
and 
$$
\begin{array}{cccc}
\bar  p\,:&\GPCpbar&\longto&G/P\\
&(gP/P,\dx)&\longmapsto&gP/P.
\end{array}
$$ 

Consider the following commutative diagram

\begin{center}
\begin{tikzpicture}
  \matrix (m) [matrix of math nodes,row sep=2em,column sep=4em,minimum width=2em]
  {
     \Ho^0(\dX,\Li)^G&&  \Ho^0(C,\Li_{|C})^L \\
\Ho^0(\GPCpbar,\bar\eta^*(\Li))^G &  \Ho^0(\bar C^+,\Li_{|\bar C^+})^P
     & \Ho^0(C^+,\Li_{|C^+})^P\\};
  \path[-stealth]
    (m-1-1) edge  (m-1-3)
    
    (m-1-1) edge node [right]{$\bar\eta^*$} node [left] {\ref{lem:rest1}}(m-2-1) 
    (m-2-1)   edge node [below]{\rm rest.} node [above] {\ref{lem:rest2}}(m-2-2) 
    (m-2-2) edge node [below]{\rm rest.} node [above] {\ref{lem:rest3}} (m-2-3)  
 (m-2-3)   edge node [left]{\rm rest.} node [right] {\ref{lem:rest4}} (m-1-3);
\end{tikzpicture}  
\end{center}

It remains to prove that the top horizontal edge is an
isomorphism.
But, Lemmas~\ref{lem:rest1} to \ref{lem:rest4} below show that the four
other morphisms are isomorphisms.
\end{proof}

Before proving the four  mentioned lemmas, we construct a
partial converse map to $\bar\eta$.

\begin{lemma}
  \label{lem:inveta}
There exists a nonempty open subset $\Omega_1$ of $\dX$ such that the
restriction of $\bar\eta$ to $\bar\eta\inv(\Omega_1)$ is a bijection
onto $\Omega_1$. Moreover, the converse map $\zeta$ is a morphism of
ind-varieties, mapping $\Omega_1$ on
$\GPCp$.
\end{lemma}

\begin{proof}
Recall that $r\,:\,U\longto U_v\subset\Aut(X_v^\GP)^\circ$ denote the
action. 
By Lemma~\ref{lem:intermorph}, there exist a nonempty open subset
$\Omega\subset U_v$ and a morphism $\psi\,:\,\Omega\longto \cX_v$ such
that 
$$
\{\psi(h)\}=X^{u_1}\cap h (X^{u_2}\cap X_v).
$$

Consider
$$
  \Omega_2=
\left\{
  \begin{array}{l@{\,}l}
(x_1,x_2,g_3\underline{o})\in\dX\,: & g_3\inv x_1\in
\mathring{G/B^-},\,g_3\inv x_2\in \mathring{G/B^-}
\mbox{, and }\\
&g_3\underline{o}\in \mathring{G/B}
\end{array}
\right\}.
$$
Let $(x_1,x_2,x_3)\in\Omega_2$. Let $g_i\in G$ such that
$g_1\underline{o}^-=x_1$,
$g_2\underline{o}^-=x_2$ and $g_3\underline{o}=x_3$.
Observe that
$$
\begin{array}{r@{\,}l}
 &g_1X^{u_1}_\GP\cap g_2X^{u_2}_\GP\cap g_3X_v^\GP\\
=&g_1X^{u_1}_\GP\cap g_2X^{u_2}_\GP\cap
                                        p^-(x_3)X_v^\GP\\
=&p^-(x_3)\bigg(
(p^-(x_3)\inv g_1)X^{u_1}_\GP\cap (p^-(x_3)\inv g_2)X^{u_2}_\GP\cap X_v^\GP
\bigg)\\
=&p^-(x_3)\bigg(h_1X^{u_1}_\GP\cap h_2X^{u_2}_\GP\cap X_v^\GP
\bigg)\\
=&(p^-(x_3)h_1).\left[X^{u_1}_\GP\cap h X^{u_2}_\GP\cap X_v^\GP\right],
\end{array}
$$
where
$$
\begin{array}{l}
h_1=p^+(p^-(x_3)\inv x_1)\\
h_2=p^+(p^-(x_3)\inv x_2)\\
  h=h_1\inv h_2.
\end{array}
$$
Since $h\in U$,
$$
g_1X^{u_1}_\GP\cap g_2X^{u_2}_\GP\cap g_3X_v^\GP =(p^-(x_3)h_1).
\left[X^{u_1}_\GP\cap r(h) (X^{u_2}_\GP\cap X_v^\GP)\right].
$$
Let $\Omega_1$ be the set of $(x_1,x_2,x_3)\in\Omega_2$ such that $r(h)\in\Omega$.
It is a nonempty open subset of $\dX$.
Moreover, by Lemma~\ref{lem:intermorph}, for $(x_1,x_2,x_3)\in\Omega_1$, 
\begin{equation}
  \label{eq:210}
g_1X^{u_1}_\GP\cap g_2X^{u_2}_\GP\cap g_3X_v^\GP =(p^-(x_3)h_1).\{\psi\circ r(h)\},
\end{equation}
and 
$$
(p^-(x_3)h_1).\psi\circ r(h)\in g_1\cX^{u_1}_\GP\cap g_2\cX^{u_2}_\GP\cap g_3\cX_v^\GP.
$$
Then the formula
$$
(x_1,x_2,x_3)\in\Omega_1\longmapsto
((p^-(x_3)h_1).\psi\circ r(h),(x_1,x_2,x_3))
$$
defines a morphism $\zeta$ from $\Omega_1$ to $\GPCp$ such that
$\eta\circ\zeta$ is the identity of $\Omega_1$.

Lemma~\ref{lem:fibreeta} with $\bar\eta$ in place of $\eta$ and
formula~\eqref{eq:210} show that the fiber of $\bar\eta$ over any point
of $\Omega_1$ is a singleton.   
\end{proof}

\bigskip
We now go from $\dX$ to $\GPCpbar$. 

\begin{lemma}
  \label{lem:rest1}
The linear map
$$
\bar\eta^*\,:\,\Ho^0(\dX,\Li)\longto \Ho^0(\GPCpbar,\bar\eta^*(\Li))
$$
is a $G$-equivariant isomorphism.
\end{lemma}

\begin{proof}
Since the image of $\bar\eta$ contains the dense subset $\Omega_1$ of $\dX$,
$\bar\eta^*$ is injective.
  Fix a filtration $\dX=\cup_{n\in\ZZ_{\geq 0}}\dX_n$ such that each $\dX_n$
  is a product of three finite dimensional Schubert varieties
  (i.e. $(B^-\times B^-\times B)$-orbit closures), $\dX_n$
  intersects $\Omega_1$ and $\dX_n\subset \dX_{n+1}$.
Set $Y=\GPCpbar$ and
$$
\mathring{Y_n}:=\bar\eta\inv(\dX_n\cap \Omega_1)=\zeta(\dX_n\cap\Omega_1).
$$
Let $Y_n$ be the closure of $\mathring{Y_n}$ in $Y$. Then $Y_n$ is
closed, irreducible, of finite dimension and projective. 

As the image of 
$G\times \bar C^+\longto\GPCpbar,\, (g,x)\longmapsto (gP/P,gx)$,
$Y$ is ind-irreducible (see examples of Section~\ref{sec:indirred}).
Moreover,  
$\cup_{n\in\ZZ_{\geq 0}}\mathring{Y}_n$  and hence
$\cup_{n\in\ZZ_{\geq 0}}Y_n$ contains the nonempty open subset
$\bar\eta\inv(\Omega_1)$. 
Then, Lemma~\ref{lem:filtration} implies that 
$Y=\cup_{n\in\ZZ_{\geq 0}}Y_n$.  

\bigskip
We now prove the surjectivity of $\bar\eta^*$. 
Let $\sigma\in \Ho^0(\GPCpbar,\bar\eta^*(\Li))$.
Consider the restriction $\bar\eta_n\,:\,Y_n\longto \dX_n$ of $\bar\eta$.
Then $\bar\eta_n$ is proper, birational and $\dX_n$ is normal. Zariski's
main theorem implies that the fibers of $\bar\eta_n$ are connected. Then
$$
\bar\eta_n^*\,:\,\Ho^0(\dX_n,\Li_{|\dX_n})\longto \Ho^0(Y_n,\bar\eta^*(\Li)_{|Y_n})
$$
is an isomorphism (see e.g. \cite[Chap IV, Corollary
5]{Peters:coursAG}).

Let $\tilde\sigma_n\in \Ho^0(\dX_n,\Li_{|\dX_n})$ such that $\bar\eta_n^*(\tilde\sigma_n)=\sigma_{|Y_n}$.
Since $\Omega_1\cap \dX_n$ is dense in $\dX_n$, the restriction of
$\tilde\sigma_{n+1}$ to $\dX_n$ is equal to $\tilde\sigma_n$.
Hence $(\tilde\sigma_n)_{n\in\NN}$ is a global section $\tilde\sigma$
of $\Li$ on $\dX$. Moreover, $\bar\eta^*(\tilde\sigma)=\sigma$.
\end{proof}

\bigskip
We now go from $\GPCpbar$ to $\bar C^+$. 

\begin{lemma}
  \label{lem:rest2}
The linear map
$$
\Ho^0(\GPCpbar,\bar\eta^*(\Li))^G\longto  \Ho^0(\bar C^+,\Li_{|\bar C^+})^P
$$
is an isomorphism.
\end{lemma}

\begin{proof}
Embed $\bar C^+$ in $\GPCpbar$ by mapping $x\in \bar C^+$ to
$(P/P,x)$. 
  Since $\bar\eta^*(\Li)_{|\bar C^+}=\Li_{|\bar C^+}$, the map of the lemma
  is well defined. 

Since $\bar C^+$ intersects any $G$-orbit in $\GPCpbar$, the map is
injective. 
Let $\sigma\in  \Ho^0(\bar C^+,\Li_{|\bar C^+})^P$. Set
$$
\begin{array}{lccl}
\tilde\sigma\,:&\GPCpbar&\longto& \bar\eta^*(\Li)\\
 &(gP/P,x)&\longmapsto& g\sigma(g\inv x).
\end{array}
$$
Note that $\tilde\sigma$ is well defined as a map since $\bar\eta^*(\Li)$ is
$G$-linearized and $\sigma$ is $P$-invariant. Moreover $\tilde
\sigma$ is $G$-invariant.
It is regular, since the morphism $G\longto G/P$ is locally trivial in
Zariski topology, because of Birkhoff's decomposition.
Hence $\tilde\sigma\in \Ho^0(\GPCpbar,\bar\eta^*(\Li))^G$ and 
${\tilde\sigma}_{|\bar C^+}=\sigma$.
\end{proof}

\bigskip
We now go from $\bar C^+$ to $C^+$. 
Let $\tilde\tau\,:\,\CC^*\longto T$ such that $\tilde\tau\in\ZZ_{>0}\tau$.

\begin{lemma}\label{lem:rest3}
Recall that $\mu^\Li(C,\tau)=0$. Then
  the restriction map
$$
\Ho^0(\bar C^+,\Li_{|\bar C^+})^{\tilde\tau(\CC^*)}\longto  \Ho^0(C^+,\Li_{|C^+})^{\tilde\tau(\CC^*)}
$$
is an isomorphism.
\end{lemma}

\begin{proof}
  Since $C^+$ is dense in $\bar C^+$, the restriction is injective. 
For $w\in W$, we recall that 
$$
  X_w^{G/B}=\overline{Bw\underline{o}}
\qquad X^w_{G/B}=\overline{B^-w\underline{o}},
$$
and set
$$
X_w^{G/B^-}=\overline{B^-w\underline{o}^-}
\qquad X^w_{G/B^-}=\overline{Bw\underline{o}^-}
$$
in such a way $\dim(X_w^{G/B})=\dim(X_w^{G/B^-})=l(w)$ and
$\codim(X^w_{G/B})=\codim(X^w_{G/B^-})=l(w)$.
Note also that $\overline{Pu_1\inv \underline{o}^-}=X^{u_1\inv}_{G/B^-}$ and
$\overline{Pu_2\inv \underline{o}^-}=X^{u_2\inv}_{G/B^-}$, whereas
$X_{v\inv}^\GB$ is finite dimensional and closed in $\overline{Pv\inv
  \underline o}$.

Consider a sequence $(w_1^n,w_2^n,w_3^n)\in W^3$ such that, for $i=1,2$
and for any $n\in\NN$,
$u_i\inv\bleq w_i^n\bleq w_i^{n+1}$ and $v\inv\bleq w_3^n\bleq w_3^{n+1}$  and
$$
\bar C^+_n:=(X_{w_1^n}^{G/B^-}\cap X^{u_1\inv}_{G/B^-})\times (X_{w_2^n}^{G/B^-}\cap
X^{u_2\inv}_{G/B^-})\times X_{w_3^n}^{G/B}
$$ 
is a filtration of $\bar C^+$.
In particular $X_{w_3}^\GB\subset \overline{Pv\inv \underline{o}}$.
Note that if $P$ has finite type, $(w_3^n)_{n\in\NN}$ can be chosen to be
constant. 
Note also that $C^+\cap\bar C_n^+$ is open in $\bar C_n^+$ and
nonempty for any $n$.
Let $\sigma\in \Ho^0(C^+,\Li_{|C^+})^{\tilde\tau}$ and let $\sigma_n$ denote its
restriction to $C^+\cap \bar C^+_n$. 
We have to prove that $\sigma$ extends to a section $\bar\sigma$ on
$\bar C^+$. 
It remains to prove that each $\sigma_n$ extends to a section $\bar\sigma_n$ on
$\bar C^+_n$. 
Indeed, then $\bar \sigma=(\bar\sigma_n)_{n\in\NN}\in
\Ho^0(\bar C^+,\Li_{|\bar C^+})^{\tilde\tau}$ extends $\sigma$.

\bigskip
Fix $n\in \NN$. 
By Lemma~\ref{lem:rappelsRichardson}, $\bar C^+_n$ is normal. Then, to prove that
$\sigma_n$ extends to $\bar C^+_n$ it is sufficient to prove that it has
no pole along the divisors of   $\bar C^+_n-C^+$. 
Let $D_n$ be such a divisor. Then, either
\begin{enumerate}
\item[(a)] $D_n=(X_{w_1^n}\cap X^{\tilde u_1\inv})\times (X_{w_2^n}\cap
X^{u_2\inv})\times X_{w_3^n}$, for some $\tilde u_1\in W^P$ such that $u_1\inv
\bleq \tilde u_1\inv\bleq w_1^n$ and $l(\tilde u_1)=l(u_1)+1$; or 
\item[(a')] $D_n=(X_{w_1^n}\cap X^{u_1\inv})\times (X_{w_2^n}\cap
X^{\tilde u_2\inv})\times X_{w_3^n}$, for some $\tilde u_2\in W^P$ such that $u_2\inv
\bleq \tilde u_2\inv\bleq w_2^n$ and $l(\tilde u_2)=l(u_2)+1$; or
\item[(b)] $D_n=(X_{w_1^n}\cap X^{u_1\inv})\times (X_{w_2^n}\cap
X^{u_2\inv})\times X_{\tilde w_3}$ for some $\tilde w_3\in W$ such that
$\tilde w_3
\bleq w_3^n$,  $l(\tilde w_3)=l(w_3^n)-1$ and $X_{\tilde w_3}\subset
X_{w_3^n}-Pv\inv \underline{o}$.
\end{enumerate}

In each case, we will apply Lemma~\ref{lem:GIT} to an affine neighborhood
of $D_n$ in   $\bar C_n^+$. We first construct such  neighborhoods and
check the assumptions of  Lemma~\ref{lem:GIT}. We do not
consider Case (a') that is similar to the first one.  We also skip the
power $n$ from $w_i^n$. Note that the action of $\CC^*$ in
Lemma~\ref{lem:GIT} is given by $\tilde\tau$.

\bigskip
\noindent
\underline{Case (a).} 
Set
$$
X=(X_{w_1}\cap X^{u_1\inv}\cap \tilde u_1\inv B\underline{o}^-)\times (X_{w_2}\cap
\cX^{u_2\inv})\times \cX_{w_3},
$$
$D=D_n\cap X$ and $U=X-D$.

The fact that $X$ is $T$-stable is obvious.
Since $\tilde u_1\inv B\underline{o}^-$, $\cX^{u_2\inv}$ and $\cX_{w_3}$ are open in 
$G/B^-$, $X^{u_2\inv}$ and $X_{w_3}$ respectively, $X$
is open in $\bar C_n^+$.
By
\cite[Lemma~7.3.5]{Kumar:KacMoody}, $X^{u_1\inv}\cap \tilde u_1\inv
B\underline{o}^-$ is affine. Then  the first factor of $X$ is affine. But,
\cite[Lemma~7.3.5]{Kumar:KacMoody} also implies that the two other
factors are affine, and $X$ is affine.

Moreover, by Lemma~\ref{lem:rappelsRichardson}, $X$ is normal. 

\bigskip
{\it Check Assumption $(i)$.}
It is sufficient to prove it for the first factor. 
Let $x\in (X_{w_1}\cap X^{u_1\inv}\cap \tilde u_1\inv
B\underline{o}^-)-X^{\tilde u_1\inv}$.
Set $y=\lim_{t\to 0}\tilde\tau(t)x$. We have to prove that $y\not\in (X_{w_1}\cap X^{u_1\inv}\cap \tilde u_1\inv
B\underline{o}^-)$. 

Let $w\in W$ such that $x\in \cX^w_{G/B^-}$. Since $x\in X^{u_1\inv}$,
$u_1\inv\leq w$. Since $x\in \tilde u_1\inv
B\underline{o}^-$, $\tilde u_1\inv \underline{o}^-$ belongs to $X^w_\GBm$ and
$w\bleq\tilde u_1\inv$. But $l(\tilde u_1)=l(u_1)+1$, hence $w=\tilde
u_1\inv$ or $u_1\inv$. Now, $x\not\in X^{\tilde u_1\inv}$ and
$w=u_1\inv$.

Now, Lemma~\ref{lem:flow1} implies that $y$ does not belong to $\tilde u_1\inv
B\underline{o}^-$. The claim is proved.

\bigskip
{\it Check Assumption $(ii)$.}
We work successively on each factor of $X$.

Let $x\in \cX_{w_3}$ such that $\lim_{t\to \infty}\tilde\tau(t)x\in
\cX_{w_3}$.
For any $b\in B$, $\lim_{t\to \infty}\tilde\tau(t)b\tilde\tau(t\inv)$ exists in
$B$. 
Then $\lim_{t\to 0}\tilde\tau(t)x\in\cX_{w_3}$.
If $x$ is not fixed by $\tau(\CC^*)$, $\overline{\tilde\tau(\CC^*)x}$ is
isomorphic to $\PP^1$. Hence it cannot be contained in the affine
variety $\cX_{w_3}$. Hence $y\not\in \cX_{w_3}$. Contradiction.
It follows that $x$ is fixed by $\tilde\tau(\CC^*)$. 

Similarly for $x\in X_{w_2}\cap
\cX^{u_2\inv}$, if $\lim_{t\to \infty}\tilde\tau(t)x\in \cX^{u_2\inv}$ then
$x$ is fixed by $\tilde\tau(\CC^*)$. 

Now, it is sufficient to prove Assumption $(ii)$ for the first factor
of $X$. Let $x_1$, $x_2\in (X_{w_1}\cap X^{u_1\inv}\cap \tilde u_1\inv
B\underline{o}^-)-X^{\tilde u_1\inv}$ such that $y:=\lim_{t\to
  \infty}\tilde\tau(t)x_1=\lim_{t\to \infty}\tilde\tau(t)x_2\in D$. 

We already noticed that $x_1$ and $x_2$ belong to $\cX^{u_1\inv}$.
By assumption, $y\in \tilde u_1\inv B\underline{o}^-\cap X^{\tilde u_1\inv}=
\cX^{\tilde u_1\inv}$. 
Then Lemma~\ref{lem:flow2} shows that $\tilde\tau(\CC^*)x_1=\tilde\tau(\CC^*)x_2$.

\bigskip
{\it 
Check Assumption $(iii)$.} 
This can be proved component by component.

Let $x\in \cX_{w_3}$ (resp. $X_{w_2}\cap
\cX^{u_2\inv}$). We already observe that $\lim_{t\to 0}\tilde\tau(t)x$
belongs to $ \cX_{w_3}$ (resp. $X_{w_2}\cap
\cX^{u_2\inv}$).

Let now $x\in (X_{w_1}\cap X^{\tilde u_1\inv}\cap \tilde u_1\inv
B\underline{o}^-)=(X_{w_1}\cap \cX^{\tilde u_1\inv})$. 
Hence, we are in the situation of the second factor. 

Finally,  for any $x\in D$, $\lim_{t\to 0}\tilde\tau(t)x$
belongs to $D$.

\bigskip
{\it 
Check Assumption $(iv)$.} 
It is sufficient to consider the first factor. Let $y\in (X_{w_1}\cap X^{\tilde u_1\inv}\cap \tilde u_1\inv
B\underline{o}^-)^{\tilde\tau}=(X_{w_1}\cap \cX^{\tilde u_1\inv})^{\tilde\tau}$. We have to find $x\in
 (X_{w_1}\cap X^{u_1\inv}\cap \tilde u_1\inv
B\underline{o}^-)-X^{\tilde u_1\inv}$ such that $\lim_{t\to\infty}\tilde\tau(t)x=y$.

Assume first that $y=y_0=\tilde u_1\inv \underline{o}^-$. 
Since $l(\tilde u_1)=l(u_1)+1$, $X_{\tilde u_1\inv}\cap X^{u_1\inv}$
is a Richardson variety of dimension one. 
By Lemma~\ref{lem:rappelsRichardson},   $\cX_{\tilde u_1\inv}\cap
\cX^{u_1\inv}$ is dense in $X_{\tilde u_1\inv}^{u_1\inv}$. 
Let $x_0\in \cX_{\tilde u_1\inv}\cap
\cX^{u_1\inv}$.

Since $\tilde u_1\in W^P$, $\tilde u_1\inv B^-\tilde u_1\cap L=B^-\cap
L$. Moreover, $B^-=P^{u,-}(B^-\cap L)$. Hence
$$
\cX_{\tilde u_1\inv}=B^-\tilde u_1\inv \underline{o}^-=P^{u,-}(B^-\cap L) \tilde u_1\inv \underline{o}^-=P^{u,-}\tilde u_1\inv \underline{o}^-.
$$
In particular, $\lim_{t\to\infty}\tilde\tau(t)x_0=y_0$.

Since $\tilde u_1\inv\bleq w_1$, $x_0\in X_{w_1}$.
Moreover, $x_0\in\cX^{u_1\inv}$; thus $x_0\not\in X^{\tilde u_1\inv}$.
And $x_0\in\cX_{\tilde u_1\inv}\subset \tilde u\inv B\underline{o}^-$. 
Finally, $x_0\in  (X_{w_1}\cap X^{u_1\inv}\cap \tilde u_1\inv
B\underline{o}^-)-X^{\tilde u_1\inv}$.

\bigskip
Now $y\in (\cX^{\tilde u_1\inv})^{\tilde\tau}=(B\cap L) \tilde u_1\inv
\underline{o}^-$. Let $l\in B\cap L$ such that $y=l y_0$. Set $x=lx_0$.
Since $l\in L$, $\lim_{t\to\infty}\tilde\tau(t)x=y$.
Since $l\in L$ and $X^{u_1\inv}$ and $X^{u_1\inv}$ are $L$-stable, $x\in
X^{u_1\inv}-X^{\tilde u_1\inv}$. 
Since $\tilde u_1\in W^P$, $\tilde u_1\inv B\tilde u_1\cap L=B\cap
L$. Hence $x\in \tilde u_1\inv B \underline{o}^-$.
Recall that $x_0\in P^{u,-}y_0$. 
Since $l$ normalizes $P^{u,-}$, this emplies that $x\in P^{u,-}y$. 
But $X_{w_1}$ is $B^-$-stable, and $x\in X_{w_1}$.
Finally, $x$ works. 

\bigskip
\noindent
\underline{Case (b).} 
Set
$$
X=(X_{w_1}\cap \cX^{u_1\inv})\times (X_{w_2}\cap
\cX^{u_2\inv})\times (X_{w_3}\cap \tilde w_3 B^-\underline{o}),
$$
$D=D_n\cap X$ and $U=X-D$.

 Lemma~\ref{lem:rappelsRichardson} and
 \cite[Lemma~7.3.5]{Kumar:KacMoody} imply that $X$ is open in $\bar
 C^+$, $T$-stable, affine and normal.

\bigskip
{\it 
Check Assumption $(i)$.}
It is sufficient to prove it for the third factor. 
Let $x\in (X_{w_3}\cap \tilde w_3
B^-\underline{o})-X^{\tilde w_3}$.
Set $y=\lim_{t\to 0}\tilde\tau(t)x$. 
We have to prove that $y\not\in \tilde w_3 B^-\underline{o}$. 

Like in Case a - $(i)$, one has
$x\in\cX_{w_3}$.
Hence $y\in (B\cap L).w_3\underline{o}$.
But $(\tilde w_3 B^-\underline{o} )^{\tilde\tau}=((\tilde w_3B^-\tilde w_3\inv)\cap L)
\tilde w_3\underline{o}$.
In particular, if $y\in  \tilde w_3 B^-\underline{o}$ then $Pv\inv
\underline{o}=P w_3\underline{o}
=P \tilde w_3\underline{o}$. Contradiction.

\bigskip
{\it 
Check Assumption $(ii)$.}
Let $x_1,x_2\in U$ such that $\lim_{t\to \infty}\tilde\tau(t)x_1=\lim_{t\to
  \infty}\tilde\tau(t)x_2\in D$.
In the proof of Case a - $(ii)$, we proved that the two first
components of each $x_i$ is fixed by $\CC^*$. We assume now that
$x_1,x_2\in (X_{w_3}\cap \tilde w_3
B^-\underline{o})-X^{\tilde w_3}$.

Since $X_{w_3}\cap \tilde w_3
B^-\underline{o}$ is contained in $\cX_{w_3}\cup\cX_{\tilde w_3}$, $x_1,x_2\in
\cX_{w_3}$.
Similarly $y\in\cX_{\tilde w_3}$.
Then Lemma~\ref{lem:flow2} implies that 
$\tilde\tau(\CC^*)x_1=\tilde\tau(\CC^*)x_2$. 

\bigskip
{\it 
Check Assumption $(iii)$.}
We can work on each component separately. The two first one was
treated in Case a-$(iii)$. The last one works since $X_{\tilde
  w_3}\cap\tilde w_3 B^-\underline{o}=\cX_{\tilde w_3}$.

\bigskip
{\it 
Check Assumption $(iv)$.}
 We can work in the last factor. 
Let $y\in \cX_{\tilde w_3}^{\tilde\tau}$. 

Assume first that $y=y_0=\tilde w_3\underline{o}$. 
Consider the Richardson variety $C=X_{w_3}\cap X^{\tilde w_3}$ of
dimension one.
Pick $x_0\in \cX_{w_3}\cap \cX^{\tilde w_3}$. 
Then $x_0\in X_{w_3}\cap \tilde w_3B^-\underline{o}$. Let $\gamma$ be the
character of the action of $T$ on $T_{y_0}C$. It is a root of
$\lg$. 
Since $\overline{Py_0}$ does not contain $w_3\underline{o}$, $\gamma$ is not a
root of $P$. Then $\gamma$ is a root of $P^{u,-}$. 
This implies that $\lim_{t\to\infty}\tilde\tau(t)x_0=y_0$. 

Now $y\in\cX_{\tilde w_3}$ and there exists $l\in B\cap L$ such that
$y=ly_0$. 
Consider the curve $lC$: it is $\tilde\tau(\CC^*)$-stable, contained in $X_{w_3}$ and contains
$y$. Since $X$ is open in $X_{w_3}$, $C\cap X_{w_3}\cap \tilde w_3 B^-\underline{o}$ is a nonempty open
subset of $C$. Then $lx\in X$. But $\lim_{t\to\infty}\tilde\tau(t)x=y$. 
  
\bigskip
\noindent{\sc The line bundle on the affine subvarieties.}
Since $\Li_-(\lambda_1)$ is $G$-linearized and the action of $U$ on
$\underline{o}^-$ is free, $\Li_-(\lambda_1)$ is trivial as a line bundle on
$U\underline{o}^-$. Similarly $\Li(\mu)$ is trivial on $U^-\underline{o}$. 
As a consequence, $\Li$ is trivial as a line bundle, on each affine
variety $X$ we have considered. 

To determine $\Li_{|X}$ as a $\CC^*$-linearized line bundle, it is
sufficient to determine the action of $\CC^*$ on one $\CC^*$-fixed
point. 
Set
$$
x_0=(u_1\inv \underline{o}^-,u_2\inv \underline{o}^-,w_3 \underline{o}),
$$
and 
$$
\begin{array}{ll}
x_a=(\tilde u_1\inv \underline{o}^-,u_2\inv \underline{o}^-,w_3
  \underline{o}),&\mbox{and}\\
x_b=(u_1\inv
\underline{o}^-,u_2\inv \underline{o}^-,\tilde w_3 \underline{o}).
\end{array}
$$
By assumption, $\CC^*$ acts trivially on the fiber $\Li_{x_0}$. 
In Case~(a), we have constructed a copy of $\PP^1$, containing $x_0$
and $x_a$ such that for $x\in\PP^1-\{x_0\}$, $\lim_{t\to\infty}\tilde\tau(t)x=x_a$.
Moreover, $\Li_{\PP^1}$ is nonnegative as a line bundle. 
Now, a computation in $\PP^1$ shows that the action of $\CC^*$ on
$\Li_{x_a}$ is given by a nonpositive weight $k_a$.

Similarly, the action of $\CC^*$ on
$\Li_{x_b}$ is given by a nonpositive weight $k_b$.

Since $\Li$ is trivial on $X$, we deduce that, for any consider affine
open subsets $X$, we have
\begin{equation}
  \label{eq:H0open}
  H^0(X,\Li)^{\tilde\tau}\simeq \CC[X]^{(k)},
\end{equation}
for some nonnegative integer $k$.

\bigskip
We are now in position to apply Lemma~\ref{lem:GIT}.
By assumption the restriction $(\sigma_n)_{|U}$ belongs to
$H^0(U,\Li)^{\tilde\tau}\simeq \CC[U]^{(k)}$. 
By  Lemma~\ref{lem:GIT}, $\CC[U]^{(k)}=\CC[X]^{(k)}$. Hence
$(\sigma_n)_{|U}$ extends to a regular section on $X$. In particular,
it has no pole along $D_n$. Then $\sigma_n$ extends to a regular section on $\bar
C^+_n$ by normality. This ends the proof of the lemma.
\end{proof}

\bigskip
The last step goes from $C^+$ to $C$.

\begin{lemma}\label{lem:rest4}
Recall that $\mu^\Li(C,\tilde\tau)=0$. Then
  the restriction map
$$
\Ho^0(C^+,\Li_{| C^+})^{\tilde\tau}\longto  \Ho^0(C,\Li_{C})^{\tilde\tau}
$$
is an isomorphism.
\end{lemma}

\begin{proof}
We first prove the injectivity. Let $\sigma\in \Ho^0(C^+,\Li_{|
  C^+})^{\tilde\tau}$ such that $\sigma_{|C}=0$ and $x\in C^+$. Consider the
morphism
$$
\begin{array}{r@{\;}cclcl}
  \theta_x\,:&\CC&\longto&C^+\\
&t&\longmapsto&\tilde\tau(t)x&\mbox{if}&t\neq 0,\\
&0&\longmapsto&\lim_{t\to 0}\tilde\tau(t)x.
\end{array}
$$
It is $\CC^*$-equivariant for the natural actions of $\CC^*$. 
Moreover, $\theta_x^*(\Li)$ is trivial as $\CC^*$-linearized line bundle,
since $\mu^\Li(C,\tilde\tau)=0$.
But $\theta_x^*(\sigma)(0)=0$ and $\theta_x^*(\sigma)$ is $\CC^*$-invariant. 
Thus $\theta_x^*(\sigma)=0$. 
In particular $\sigma(x)=0$.

\bigskip
Consider now the map 
$\pi\,:\,C^+\longto C,\, x\longmapsto \lim_{t\to  0}\tilde\tau(t)x$.
We claim that $\pi^*(\Li_{|C})$ is isomorphic to $\Li$. 

We can work on each factor of $C^+$ separately. So assume for proving the claim
that $X=G/B$, $C^+=Pv\inv \underline{o}$ and $C=Lv\inv \underline{o}$. 
Let $\iota_\mu\,:\,X\longto\PP(L(\mu))$ induced by the action of $G$
on the highest weight line of $\PP(L(\mu))$.
Set $k=\langle\tilde\tau,v\inv \mu\rangle$. Recall that $L(\mu)$ has a
weight space decomposition $L(\mu)=\oplus_{\chi\in
  X(T)}L(\mu)_\chi$. Set
$$
L(\mu)^k=\bigoplus_{\langle\chi,\tilde\tau\rangle=k}L(\mu)_\chi
\qquad
L(\mu)^{>k}=\bigoplus_{\langle\chi,\tilde\tau\rangle>k}L(\mu)_\chi.
$$ 
Define $$
\cE=\{[v_0+v_+]\,:\, v_0\in L(\mu)^k-\{0\}\mbox{ and } v_+\in
L(\mu)^{>k}\}
$$
as a subset of $\PP(L(\mu)^k\oplus L(\mu)^{>k})$ and so of
$\PP(L(\mu))$.
Then $\iota_\mu(C)$ is contained in $\PP(L(\mu)^k)$ and $\iota_\mu(C^+)$ is contained in
$\cE$.
Moreover, $\pi$ is the restriction of the canonical projection  
$\cE\longto
\PP(L(\mu)^k)$.
But, $\Li(\mu)^*$ is the restriction to $X$ of the tautological bundle
on $\PP(L(\mu))$. The claim follows.

\bigskip
Let us prove the surjectivity. Let $\sigma\in
\Ho^0(C,\Li_{|C})^{\tilde\tau}$. By the claim, we have to prove that $\sigma$
extends to a $\CC^*$-invariant section of $\pi^*(\Li)$.
The morphism
$$
\begin{array}{ccl}
  C^+&\longto&\Li_{|C}\\
x&\longmapsto&\sigma(\pi(x))
\end{array}
$$
induces such an extension.
\end{proof}

\section{The Belkale-Kumar-Brown product}

In this section, we purpose a construction of the BKB product
$\bkprod$ (see \cite{BrownKumar}) and prove some properties of it.

\subsection{Preliminaries of linear algebra}

Let $E$ be a complex vector space filtered by 
linear subspaces
$$
\{0\}=E^0\subset E^1\subset E^2\subset\dots\subset E^n\subset\cdots
$$
such that
$E=\cup_n E^n$.
Let $T$ and $U$ be two linear subspaces of $E$ such that $T$ has finite
dimension, $U$ has finite codimension and $\dim(T)=\codim(U)$.
Consider the linear map
$$
\begin{array}{cccl}
  \Theta\,:&E&\longto&E/T\oplus E/U\\
&v&\longmapsto&(v+T,v+U).
\end{array}
$$
Define the induced filtrations from $E$ to $T$ and $E/U$: 
$$
T^n=E^n\cap T\quad\mbox{and}\quad (E/U)^n=E^n/(U\cap E^n).
$$
Set
$$
\delta=\sum_{n> 0}n\bigg (\dim(T^n/T^{n-1})-\dim((E/U)^n/(E/U)^{n-1})\bigg ).
$$

\begin{lemma}
\label{lem:deltanegLA}
  If $\Theta$ is an isomorphism then
$$
\dim(T^n)\leq \dim(E/U)^n\qquad\forall n\in \ZZ_{\geq 0}.
$$
In particular $\delta\geq 0$.
\end{lemma}

\begin{proof}

  Consider 
$$
\begin{array}{cccl}
  \bar\Theta\,:&T&\longto& E/U\\
&v&\longmapsto&v+U.
\end{array}
$$
Since $\Theta$ is an isomorphism,  so is $\bar\Theta$.
Moreover $\bar\Theta(T^n)\subset (E/U)^n$.
Then the first inequality of the lemma is a consequence of the injectivity
of the restriction of $\bar\Theta$ to $T^n$.

Since $T$ and $E/U$ are finite dimensional, there exists $N$ such that
$T^n=T$ and $(E/U)^n=E/U$ for any $n\geq N$.
Then
$$
\begin{array}{ll}
\delta&=\sum_{n=1}^N n\bigg
    (\dim(T^n/T^{n-1})-\dim((E/U)^n/(E/U)^{n-1})\bigg )\\
&=\sum_{n=1}^N n\bigg
    (\dim(T^n)-\dim(T^{n-1})-\dim((E/U)^n)+\dim(E/U)^{n-1})\bigg )\\
&=\sum_{n=0}^{N-1} 
    \dim((E/U)^n)-\dim(T^n),
\end{array}
$$
since $\dim(T)=\dim(E/U)$.
In particular, $\delta\geq 0$.
\end{proof}

Consider the graded vector spaces
$$
\gr T=\oplus_{n\in\ZZ_{>0}}T^n/T^{n-1}\quad\mbox{and}\quad
\gr (E/U)=\oplus_{n\in\ZZ_{>0}}(E/U)^n/(E/U)^{n-1}.
$$
The map $\bar\Theta$ induces a graded linear map
$$
\gr\bar\Theta\,:\,\gr T\longto \gr(E/U).
$$

\begin{lemma}
\label{lem:deltanulLA}
   Assume that $\Theta$ is an isomorphism. The following assertions
   are equivalent 
   \begin{enumerate}
   \item $\gr\bar\Theta$ is an
   isomorphism;
\item 
\label{ass:02}
$
\dim(T^n)=\dim(E/U)^n\qquad\forall n\in \ZZ_{\geq 0};
$
\item $\delta=0$.
   \end{enumerate}
\end{lemma}

\begin{proof}
The second assertion implies the last one by the proof of Lemma~\ref{lem:deltanegLA}.
  If $\gr\bar\Theta$ is an isomorphism then for any $n$,
$\dim(T^n)-\dim(T^{n-1})=\dim(E/U)^n-\dim(E/U)^{n-1}$.
Since $T^0$ and $(E/U)^0$ are trivial the equalities of assertion~\eqref{ass:02}
follow.

Assume now that $\delta=0$.
Since $\delta=\sum_{n\geq 0}\dim(E/U)^n-\dim(T^n)$, 
Lemma~\ref{lem:deltanegLA} shows that $\dim(E/U)^n=\dim(T^n)$, 
for any $n$. 
Then, the injectivity of $\Theta$ implies that $\bar\Theta$ induces
isomorphisms from $T^n$ onto $(E/U)^n$, for any $n$. 
It follows that $\gr\bar\Theta$ is an isomorphism.
\end{proof}

\subsection{Definition of the BKB product}
\label{sec:bkbprod}

Let $P$ be a standard parabolic subgroup of $G$. Let $u_1$, $u_2$, and $v$
in $W^P$ such that $l(v)=l(u_1)+l(u_2)$ and $n_{u_1u_2}^v\neq 0$.
Set
$$
\begin{array}{ccc}
  &\Tau=T_{P/P}G/P&\\
\quad \Tau^{u_1}=T_{P/P}{u_1}\inv X^{u_1}_\GP&
\Tau^{u_2}=T_{P/P}u_2\inv X^{u_2}_\GP& \Tau_v=T_{P/P}v\inv X_v^\GP.
\end{array}
$$
Fix $\tau$ be a one parameter subgroup of $T$ belonging to $\oplus_{\alpha_j\not\in\Delta(P)}\ZZ_{>0}\varpi_{\alpha_j^\vee}$.
Observe that $P$ acts on $\Tau$. Under the action of $\tau$, $\Tau$
decomposes as $\Tau=\bigoplus_{n\in \ZZ}\Tau_{n}$, where
$\Tau_k=\{\xi\in\Tau\,:\,\tau(t)\xi=t^k\xi\quad\forall t\in \CC^*\}$.
Note 
 that $\Tau_n=\{0\}$
for all $n\geq 0$. 
Set
$$
\Tau^n=\oplus_{k\leq n}\Tau_{-k}.
$$
Then
 $(\Tau^n)_{n\in\ZZ_{\geq 0}}$ forms a 
$P$-stable filtration of $\Tau$. 
Moreover $\Tau^0=\{0\}$. Consider also the induced filtrations 
$(\Tau/\Tau^{u_1})^n$, $(\Tau/\Tau^{u_2})^n$ and $\Tau_v^n$ on $\Tau/\Tau^{u_1}$,
$\Tau/\Tau^{u_2}$, and $\Tau_v$.
Set
$$
\begin{array}{l@{}l}
\delta_{u_1\,u_2}^v=\sum_{n\geq 0}n\bigg(
\dim(\Tau_{v}^n/\Tau_{v}^{n-1})&-
                                 \dim(((\Tau/\Tau^{u_1})^n)/(\Tau/\Tau^{u_1})^{n-1})\\
&-\dim(((\Tau/\Tau^{u_2})^n)/(\Tau/\Tau^{u_2})^{n-1})
\bigg ).
\end{array}
$$

\begin{lemma}
\label{lem:deltaneg}
  If $n_{u_1u_2}^v\neq 0$ then $\delta_{u_1\,u_2}^v\geq 0$.
\end{lemma}

\begin{proof}
  Consider the map

$$
\eta\,:\,G\times_P C^+\longto \dX,
$$ 
as in the proof of Proposition~\ref{prop:ineg}.

By Lemma~\ref{lem:Kl}, there exists $b\in B$ such that $X_v^{u_1}\cap bX_v^{u_2}=\cX_v^\GP\cap {\cX^{u_1}_\GP}\cap
b{\cX^{u_2}}_\GP$ is transverse and nonempty by Lemma~\ref{lem:cardcohom}.
Let $g\in G$ such that $gP/P$ belongs to this intersection. 
There exist $p_1$, $p_2$ and $p_3$ in $P$ such that
$$
(g P/P,(g p_1u_1\inv \underline{o}^-, g p_2u_2\inv \underline{o}^-,g p_3v\inv
\underline{o}))
$$ 
belongs to the fiber $\eta^{-1}(\underline{o}^-,b\underline{o}^-,\underline{o})$.
Observe that 
$$
g(p_1u_1\inv X^{u_1}_\GP\cap p_2u_2\inv X^{u_2}_\GP\cap p_3v\inv X_v^\GP)=X_v^{u_1}\cap b X_v^{u_2}
$$
is transverse. By Lemma~\ref{lem:Kl}, the canonical map
\begin{eqnarray}
  \label{eq:isoTau}
p_3\Tau_v\longto \frac\Tau{p_1\Tau^{u_1}}\oplus \frac\Tau{p_2\Tau^{u_2}}
  \end{eqnarray}
is an isomorphism.

The lemma follows by applying Lemma~\ref{lem:deltanegLA} with
$E=\Tau\oplus\Tau$, $E^n=\Tau^n\oplus\Tau^n$, $U=p_1\Tau^{u_1}\oplus p_2\Tau^{u_2}$
and $T\simeq \Tau_v$ embedded in $E$ by $x\mapsto (p_3x,p_3x)$.
\end{proof}

\bigskip
For $w\in W$, we denote by $\Phi_w=w\inv\Phi^+\cap\Phi^-$ the
inversion set of $w$. 
Then, $\Phi_w$ consists in $l(w)$ real roots.
Recall that $\rho=\sum_{i=0}^l\varpi_{\alpha_i}$.

\begin{lemma}
  \label{lem:calculdelta}
With above notation, we have
$$
\delta_{u_1\,u_2}^v=\langle -v\inv\rho +u_1\inv\rho+u_2\inv \rho-\rho,\tau\rangle.
$$
\end{lemma}

\begin{proof}
Since 
$$
\Tau/\Tau^{u_1}\simeq\oplus_{\alpha\in\Phi_{u_1}}\lg_\alpha\quad
  \Tau/\Tau^{u_2}\simeq\oplus_{\alpha\in\Phi_{u_2}}\lg_\alpha 
\quad{\rm and}\quad
\Tau_v \simeq\oplus_{\alpha\in\Phi_v}\lg_\alpha, 
$$
we have
$$
\delta_{u_1\,u_2}^v=-\sum_{\alpha\in\Phi_v}\langle\alpha,\tau\rangle+\sum_{\alpha\in\Phi_{u_1}}\langle\alpha,\tau\rangle+\sum_{\alpha\in\Phi_{u_2}}\langle\alpha,\tau\rangle.
$$
But by \cite[Lemma~1.3.22]{Kumar:KacMoody},
$w\inv\rho-\rho=\sum_{\alpha\in\Phi_w}\alpha$ and the lemma follows.
\end{proof}

\bigskip
Because of Lemma~\ref{lem:calculdelta}, Lemma~\ref{lem:deltaneg} can
be restated as: by assigning the degree $\langle
v\inv\rho-\rho,\tau\rangle\in\ZZ$ to $\epsilon_v$, one obtains a
filtration of the cohomology ring  $\Ho^*(G/P,\ZZ)$.
Then $\bkprod$ is defined to be the product of the associated graded
ring.

Then $\bkprod$ satisfies (with obvious identifications):
$$
\forall u_1,u_2\in W^P\qquad
\epsilon_{u_1}\bkprod\epsilon_{u_2}=\sum_{v\in
  W^P} {}^\bkprod n_{u_1\,u_2}^v\,\epsilon_v,
$$
where
$$
\begin{array}{lll}
  {}^\bkprod n_{u_1\,u_2}^v&=0&\mbox{ if $\delta_{u_1\,u_2}^v\neq 0$},\\[0.8em]
&=n_{u_1\,u_2}^v&\mbox{ if $\delta_{u_1\,u_2}^v=0$}.
\end{array}
$$

\bigskip
Let $Z(L)$ denote the center of $L$ and $Z(L)^\circ$ denote its
neutral component.
Given an $L$-module and a character $\chi\in X(Z(L)^\circ)$ we denote
by $V_\chi=\{v\in V\,:\,\forall t\in Z(L)^\circ\quad tv=\chi(t)v\}$
the associated weight space.
A priori, this construction of $\bkprod$ depends on our choice of an element 
$\tau\in
\oplus_{\alpha_j\not\in\Delta(P)}\ZZ_{>0}\varpi_{\alpha_j^\vee}$.
Actually, it does not depends on:

\begin{lemma}\label{lem:indeptau}
   Let $u_1$, $u_2$, and $v$
in $W^P$ such that ${}^\bkprod n_{u_1u_2}^v\neq 0$. 
Then, for any $\chi\in X(Z(L)^\circ)$ 
$$
\dim((\Tau_v)_\chi)=\dim(\left(
\frac\Tau{\Tau^{u_1}}
\right)_\chi)+\dim(\left(
\frac\Tau{\Tau^{u_2}}
\right)_\chi).
$$
\end{lemma}

\begin{proof}
As in the proof of Lemma~\ref{lem:deltaneg}, choose $p_1$, $p_2$, and
$p_3$ in $P$ such that the linear map \eqref{eq:isoTau} is an
isomorphism. Up to multiplying by $p_3\inv$ we assume that $p_3$ is
trivial. 
Then
$$
\bar \Theta\,:\,\Tau_v\longto \frac\Tau{p_1\Tau^{u_1}}\oplus \frac\Tau{p_2\Tau^{u_2}}
$$
is an isomorphism.
For $i=1,2$, write $p_i=g^u_il_i$ with $g^u_i\in P^u$ and $l_i\in L$. 

Let $k\in \ZZ$ and $x\in \Tau_{-k}$. Then 
\begin{eqnarray}
  \label{eq:actionn}
  g^u_ix\in x+\Tau^{k+1}. 
\end{eqnarray}
Fix bases of $\Tau_v$, $\frac\Tau{p_1\Tau^{u_1}}$ and
$\frac\Tau{p_2\Tau^{u_2}}$ adapted to the filtrations. 
Using~\eqref{eq:actionn}, one can check that the matrices of
$\gr\bar\Theta$ and of 
$$
\tilde \Theta\,:\,\Tau_v\longto \frac\Tau{l_1\Tau^{u_1}}\oplus \frac\Tau{l_2\Tau^{u_2}}
$$
coincide.
Now Lemma~\ref{lem:deltanulLA} shows that $\delta_{u_1\,u_2}^v=0$ if
and only if $\tilde\Theta$ is an isomorphism. Since $\tilde\Theta$ is
$Z(L)^\circ$-equivariant, we have 
$$
\dim((\Tau_v)_\chi)=\dim(\left(
\frac\Tau{l_1\Tau^{u_1}}
\right)_\chi)+\dim(\left(
\frac\Tau{l_2\Tau^{u_2}}
\right)_\chi),
$$
for any $\chi$.
Since $Z(L)^\circ$ is central in $L$, we have for $i=1,2$
$$
\dim(\left(
\frac\Tau{l_i\Tau^{u_i}}
\right)_\chi=\dim(\left(
\frac\Tau{\Tau^{u_i}}
\right)_\chi.
$$
The lemma is proved.
\end{proof}

\subsection{On Levi movability}

\begin{prop}
  \label{prop:LmovBKBprod}
We keep notation and assumptions of Section~\ref{sec:bkbprod} and
assume in addition that $P$ is of finite type. Recall that
$n_{u_1\,u_2}^v\neq 0$. 
We assume that there exist $l_1,l_2,l\in L$ such that 
$$
l_1u_1\inv\cX^{u_1}\cap l_2u_2\inv\cX^{u_2}\cap lv\inv\cX_v 
$$
is finite. Then
$$
\delta_{u_1\,u_2}^v
=0.
$$
\end{prop}

\begin{proof}
Identify $P^{u,-}$ with an $L$-stable open subset of $G/P$.
For any $n\geq 2$, consider the normal group $P^{u,-}_{\geq n}$ of
$P^{u,-}$ such that $P^{u,-}/P^{u,-}_{\geq n}$ is a finite dimensional
unipotent algebraic group with Lie algebra
$$
\oplus_{\alpha\in\Phi,\,1\leq
  \langle\alpha,\tau\rangle<n}\lg_\alpha$$
(see~\cite[Lemma~6.1.11]{Kumar:KacMoody}). 
Let $\pi_n\,:\,P^{u,-}\longto P^{u,-}/P^{u,-}_{\geq n}$ be the quotient
map.  
Observe that the actions of $\tau$ and $L$ commute and that  $P^{u,-}_{\geq n}$ is $L$-stable.

Since $\Phi_{u_i}$  and $\Phi_v$ are finite, there exists $N$ such
that for any $i=1,2$, 
$$u_i\inv B^-u_i\cap P\supset P^{u,-}_{\geq N}$$
and 
$$
v\inv Bv\cap P^{u,-}_{\geq N}=\{e\}.
$$
Then $\pi_N(u_i\inv\cX^{u_i})$ has codimension $l(u_i)$, for $i=1,2$;
and $\pi_N(v\inv\cX_v)$ has dimension $l(v)$.
Moreover, $\pi_N$ maps bijectively $l_1u_1\inv\cX^{u_1}\cap l_2u_2\inv\cX^{u_2}\cap
lv\inv\cX_v$ onto
\begin{eqnarray}
  \label{eq:18}
\pi_N( l_1u_1\inv\cX^{u_1})\cap \pi_N( l_2u_2\inv\cX^{u_2})\cap \pi_N( lv\inv\cX_v).
\end{eqnarray}

Consider now the exponential map
$$
\Exp\,:\,\oplus_{\alpha\in\Phi,\,1\leq
  \langle\alpha,\tau\rangle<n}\lg_\alpha\longto 
 P^{u,-}/P^{u,-}_{\geq N}.
$$
Since $P^{u,-}/P^{u,-}_{\geq N}$ is unipotent, $\Exp$ is an isomorphism of
varieties. 
Let $\lk_1$, $\lk_2$ and $\lk_3$ denote the Lie algebras of $\pi_N(
l_1u_1\inv\cX^{u_1}=l_1u_1\inv U^-u_1l_1\inv\cap P^{u,-})$,  $\pi_N(
l_1u_2\inv\cX^{u_2})$ and $\pi_N(
lv\inv\cX_v)$ respectively.
These subspaces are stable by the action of $\tau$ and decompose as $\lk_i=\oplus_{n<0}\lk_i^n$.
Since the intersection~\eqref{eq:18} is finite, $\lk_1\cap
\lk_2\cap\lk_3=\{0\}$.
Since $\dim(\lk_3)=\codim(\lk_1)+\codim(\lk_2)$, it follows that the
natural map
$$
\lk_3\longto\frac{\oplus_{\alpha\in\Phi,\,1\leq
  \langle\alpha,\tau\rangle<n}\lg_\alpha}{\lk_2}
\oplus
\frac{\oplus_{\alpha\in\Phi,\,1\leq
  \langle\alpha,\tau\rangle<n}\lg_\alpha}{\lk_3}
$$
is an isomorphism $\tau$-equivariant. 
Then for any integer $n$ the $\tau$-eigenspace 
$
\lk_3^n$
has dimension $codim(\lk_1^n)+\codim(\lk_2^n)
$.
Since the actions of $L$ and $\tau$ commute, one deduces that $\delta_{u_1\,u_2}^v=0$.
\end{proof}

\section{Multiplicativity in cohomology}

\subsection{The multiplicativity}

Let $B\subset P\subset Q$ be two standard parabolic subgroups of $G$.
Let $L^P$ and $L^Q$ denote the Levi subgroups of $P$ and $Q$
containing $T$.
Then $L^Q\cap P$ is a parabolic subgroup of $L^Q$ and $Q/P=L^Q/(L^Q\cap
P)$. 

In this section, we study relations between structure constants of
$\Ho^*(G/P,\ZZ)$, $\Ho^*(G/Q,\ZZ)$ and $\Ho^*(Q/P,\ZZ)$.
To be more precise, we extend results of \cite{Rich:mult,multi} from
the classical case to the Kac-Moody case.

Let $W_Q^P$ be the
set of minimal length representative in $W_Q$ of the classes
$W_Q/W_P$.

\begin{lemma}
  \label{lem:WQP}
The map
$$
\begin{array}{ccc}
  W^Q\times W_Q^P&\longto&W^P\\
(\bar w,\tilde w)&\longmapsto&\bar w\tilde w
\end{array}
$$
is bijective.
\end{lemma}

\begin{proof}
  Recall that (see \cite[Exercice 1.3.E]{Kumar:KacMoody}) 
$W^P=\{w\in W\,:\,w\inv \Phi^-\cap
  \Phi^+(L^P)=\emptyset\}$.
We first check that $w=\bar w\tilde w$ belongs to $W^P$; this shows that the
map of the lemma is well defined.
Write 
$$w\inv \Phi^-\cap  \Phi^+(L^P)=w\inv\bar w(\bar w\inv\Phi^-\cap
\tilde w\Phi^+(L^P)).$$
Note that $\Phi^+(L^P)\subset\Phi(L^Q)=\tilde w\Phi(L^Q)$ and that 
$\bar w\inv\Phi^-\cap\Phi(L^Q)=\Phi^-(L^Q)$
(since $\bar w\in W^Q$). Hence
$$
\begin{array}{ccl}
  w\inv \Phi^-\cap  \Phi^+(L^P)&=&\tilde w\inv\bigg(
\bar w\inv \Phi^-\cap\Phi(L^Q)\cap\tilde w \Phi^+(L^P)
\bigg)\\
&=&\tilde w\inv\bigg(\Phi^-(L^Q)\cap\tilde w \Phi^+(L^P) \bigg)\\
&=&\tilde w\inv\Phi^-(L^Q)\cap\Phi^+(L^P).
\end{array}
$$
This last intersection is empty 
since $\tilde w\in W_Q^P$. Then $w\in W^P$.

Fix now $w\in W^P$. If $w=\bar w\tilde w$ (with $\bar w\in W^Q$ and
$\tilde w\in W_Q^P$) then $wW_Q=\bar wW_Q$ and $\bar w$ is necessarily
the unique
representative of $wW_Q$ in $W^Q$. Since $\tilde w=\bar w\inv w$, this
proves that the map is injective. 

Consider now the
representative $\bar w$ of $wW_Q$ in $W^Q$ and set $\tilde w=\bar
w\inv w$.
To prove the surjectivity, it remains to prove that $\tilde w\in W_Q^P$.
The equality $\bar wW_Q=wW_Q$ implies that $\tilde w\in W_Q$. Moreover
$$
\begin{array}{c@{\,}c@{\,}l}
  \tilde w\inv \Phi^-(L^Q)\cap  \Phi^+(L^P)&=&w\inv\bar w\Phi^-(L^Q)\cap\Phi^+(L^P)\\
&\subset&w\inv\Phi^-\cap\Phi^+(L^P),
\end{array}
$$
since $\bar w\Phi^-(L^Q)\subset \Phi^-$ $(\bar w\in W^Q)$. 
This last intersection is empty since $w\in W^P$.
The lemma is proved.
\end{proof}

\bigskip
Recall that, for $w\in W$
$$
\begin{array}{l@{\hspace{2cm}}l}
  X_w^\GP=\overline{BwP/P}& \cX_w^\GP=BwP/P\\
X^w_\GP=\overline{B^-wP/P}& \cX^w_\GP=B^-wP/P
\end{array}
$$

\begin{lemma}(See \cite[Lemma~1]{BK})
  \label{lem:gtop}
Let $w\in W^P$ and $g\in G$.
\begin{enumerate}
\item If $g\cX_w^\GP$ contains $P/P$ then there exists $p\in P$ such
  that $g\cX_w^\GP=pw\inv \cX_w^\GP$.
\item If $g\cX^w_\GP$ contains $P/P$ then there exists $p\in P$ such
  that $g\cX^w_\GP=pw\inv \cX^w_\GP$.
\end{enumerate}
\end{lemma}

\begin{proof}
Fix a representative $\dot w$ of $w$ in $N(T)$.
  Let $b\in B$ and $p\in P$ such that $gb\dot w=p$. Then
  $g\cX_w^\GP=p\dot w\inv b\inv \cX_w^\GP=pw\inv \cX_w^\GP$.
The second assertion works similarly.
\end{proof}

\bigskip
\begin{lemma}
\label{lem:Xinter}
  Let $\tilde w\in W_Q^P$ and $\bar w\in W^Q$. Set $w=\bar w\tilde w$.
Then
\begin{enumerate}
\item $\bar w\inv \cX^{G/P}_w\cap Q/P=\cX^{Q/P}_{\tilde w}$;
\item $\bar w\inv \cX_{G/P}^w\cap Q/P=\cX_{Q/P}^{\tilde w}$.
\end{enumerate}
\end{lemma}

\begin{proof}
  Note that $\bar w\inv \cX^{G/P}_w\cap Q/P$ is stable by the action
  of $\bar w\inv B\bar w\cap Q$. Since $\bar w\in W^Q$, $\bar w\inv
  B\bar w\cap Q$ contains $L^Q\cap B$.
Moreover each $(L^Q\cap B)$-orbit in $Q/P$ contains a
$T$-fixed point. Hence
$$
\bar w\inv \cX^{G/P}_w\cap Q/P=\bigcup_{x\in (\bar w\inv \cX^{G/P}_w\cap Q/P)^T}(L^Q\cap B).x.
$$
But $(\bar w\inv \cX^{G/P}_w\cap Q/P)^T\subset \bar w\inv(
\cX^{G/P}_w)^T=\{\tilde wP/P\}$. The first assertion of the lemma follows. 
The second one works similarly.
\end{proof}

\bigskip
We can now prove the main result of this section.

\begin{prop}
\label{prop:multiplicative}
  Let $u_1$, $u_2$, and $v$ in $W^P$. Write
$u_1=\bar u_1\tilde u_1$, $u_2=\bar u_2\tilde u_2$, and $v=\bar
v\tilde v$
as in Lemma~\ref{lem:WQP}.
We assume that $l(v)=l(u_1)+l(u_2)$ and  $l(\bar v)=l(\bar u_1)+l(\bar u_2)$.

Consider the structure constants $n_{u_1\,u_2}^v$, $n_{\bar u_1\bar u_2}^{\bar v}$,
and $n_{\tilde u_1\,\tilde u_2}^{\tilde v}$ in $\Ho^*(G/P,\ZZ)$,
$\Ho^*(G/Q,\ZZ)$, and $\Ho^*(Q/P,\ZZ)$ respectively.

Then
$$
n_{u_1\,u_2}^v=n_{\bar u_1\bar u_2}^{\bar v}n_{\tilde u_1\,\tilde u_2}^{\tilde v}.
$$
\end{prop}

\begin{proof}
  Since $B$ is irreducible, Lemma~\ref{lem:Kl} implies
  that there exists $b\in B$ such that 
  \begin{enumerate}
  \item $X^{u_1}_{G/P}\cap b X_v^{u_2}=\cX^{u_1}_{G/P}\cap \cX_v^{G/P}\cap
    b\cX^{u_2}_{G/P}$ is transverse, and
\item $X^{\bar u_1}_{G/Q}\cap b X_{\bar v}^{\bar u_2}=\cX^{\bar
    u_1}_{G/Q}\cap \cX_{\bar v}^{G/Q}\cap
    b\cX^{\bar u_2}_{G/Q}$ is transverse.
  \end{enumerate}
By Lemma~\ref{lem:cardcohom}, it remains to determine the cardinality
of the intersection~(i). We do this by counting in each fiber of the
$G$-equivariant projection $\pi\,:\,G/P\longto G/Q$.

Fix $g\in G$ such that $gQ/Q\in X^{\bar u_1}_{G/Q}\cap b X_{\bar v}^{\bar u_2}$. Then
$Q/Q\in g\inv \cX^{\bar
    u_1}_{G/Q}\cap g\inv \cX_{\bar v}^{G/Q}\cap
    g\inv b\cX^{\bar u_2}_{G/Q}$.
As in Lemma~\ref{lem:gtop},
there exist  $q_1$, $q_2$, and $q$ in $Q$ such that 
$g\inv\in q\dot{\bar v}\inv B$, 
$g\inv\in q_1\dot{\bar u}_1\inv B^-$, and
$g\inv b\in q_2\dot{\bar u}_2\inv B^-$. 
Let $l_1$, $l_2$, and $l$ in $Q$ such that $q_1l_1\inv$, $q_2l_2\inv$, and
$ql\inv$ belong to $Q^u$.  
Then
$$
\begin{array}{r@{\;}ll}
  I:=&g\inv (\cX^{u_1}_{G/P}\cap b\cX^{ u_2}_{G/P} \cap  \cX_{v}^{G/P}
   \cap \pi\inv(gQ/Q))\\[0.8em]
=&
q_1\bar u_1\inv\cX_{G/P}^{u_1}\cap q_2\bar u_2\inv\cX_{G/P}^{u_2}\cap q\bar
                                         v\inv\cX^{G/P}_{v}\cap
                                         Q/P\\[0.8em]
=&q\cX^{Q/P}_{\tilde v}\cap
q_1\cX_{Q/P}^{\tilde u_1}\cap q_2\cX_{Q/P}^{\tilde u_2}
 &\mbox{by Lemma~\ref{lem:Xinter}}\\[0.8em]
=&l\cX^{Q/P}_{\tilde v}\cap
l_1\cX_{Q/P}^{\tilde u_1}\cap l_2\cX_{Q/P}^{\tilde u_2}
\end{array}
$$
The last equality holds since $Q^u$ acts trivially on $Q/P$.
Moreover, since $X^{u_1}_{G/P}\cap b X_v^{u_2}=\cX^{u_1}_{G/P}\cap \cX_v^{G/P}\cap
    b\cX^{u_2}_{G/P}$, we also have 
$$I=l X^{Q/P}_{\tilde v}\cap
l_1 X_{Q/P}^{\tilde u_1}\cap l_2 X_{Q/P}^{\tilde u_2}.
$$


\bigskip
\noindent
\underline{Claim}. The intersection $l\cX^{Q/P}_{\tilde v}\cap
l_1\cX_{Q/P}^{\tilde u_1}\cap l_2\cX_{Q/P}^{\tilde u_2 }$ is
transverse.\\

Let $x$ be a point in this intersection. 
By Lemma~\ref{lem:Kl}, 
the map

$$
T_xg\inv\cX_v^{G/P}\longto 
\frac{T_xG/P}{T_xg\inv\cX^{u_1}_{G/P}}\oplus \frac{T_xG/P}{T_xg\inv b\cX^{u_2}_{G/P}}
$$
is  an  isomorphism.
Hence, the natural map
$$
T_xg\inv\cX_v^{G/P}\cap T_xQ/P \longto 
\frac{T_xQ/P}{T_xg\inv\cX^{u_1}_{G/P}\cap T_xQ/P}\oplus \frac{T_xQ/P}{T_xg\inv b\cX^{u_2}_{G/P}\cap T_xQ/P}
$$
is injective. 
Since $T_xl\cX_{\tilde v}^{Q/P}\subset
T_xg\inv\cX_v^{G/P}\cap T_xQ/P$ (and similar inclusions hold for $u_1$
and $u_2$), we deduce that the natural map
$$
T_xl\cX_{\tilde v}^{Q/P}\longto 
\frac{T_xQ/P}{T_xl_1\cX^{\tilde u_1}_{Q/P}}\oplus \frac{T_xQ/P}{T_xl_2\cX^{\tilde u_2}_{Q/P}}
$$
is injective. 
The assumption on the length of elements of $W_Q^P$ implies that it is
in fact an isomorphism. The claim is proved.

\bigskip
The claim and Lemma~\ref{lem:cardcohom} imply that the cardinal of $I$
is $n_{\tilde u_1\,\tilde u_2}^{\tilde v}$. Since this holds for any
of the $n_{\bar u_1\,\bar u_2}^{\bar v}$ points in $X^{\bar
  u_1}_{G/Q}\cap b X_{\bar v}^{\bar u_2}$, we get that
$X^{u_1}_{G/P}\cap b X_v^{u_2}$ has cardinality $n_{\tilde u_1\,\tilde
  u_2}^{\tilde v}n_{\bar u_1\,\bar u_2}^{\bar v}$. 
We conclude by applying Lemma~\ref{lem:cardcohom} in $G/P$.
\end{proof}

\subsection{Application to the BKB-product}

The following lemma allows to apply
Proposition~\ref{prop:multiplicative} to any structure constants of
the BKB-product.

\begin{lemma}
\label{lem:bkhyp}
  Let $P\subset Q$ be two standard parabolic subgroups of $G$.
 Let $u_1$, $u_2$, and $v$
in $W^P$ such that ${}^\bkprod n_{u_1u_2}^v\neq 0$.

Then $
l(\bar v)=l(\bar u_1)+l(\bar u_2)$.
\end{lemma}

\begin{proof}
Let $Q^-$ be the opposite subgroup of $Q$ and $Q^{u,-}$ its
``unipotent'' subgroup.
Let $\lq^{u,-}$ be the Lie algebra of  $Q^{u,-}$. 
Note that 
$$
l(\bar v)=\dim(\Tau_v\cap \lq^{u,-})\mbox{ and }
l(\bar u_i)=\dim(\frac{\lq^{u,-}}{\Tau^{u_i}\cap \lq^{u,-}})
\ \forall i=1,2.
$$
There exists a one parameter subgroup $\tau_Q$ of $Z(L^Q)$ such that $\lq^{u,-}$
is the sum of the negative weight spaces for $\tau_Q$. 
Since $Z(L^Q)^\circ$ is contained in $Z(L^P)^\circ$:
 $$\Tau_w\cap
\lq^{u,-}=\bigoplus_{\chi\in X(Z(L^P))^\circ\ \langle
  \chi,\tau_Q\rangle<0}(\Tau_w)_\chi$$
for any $w\in W$.
Now the equality of the lemma is a direct consequence of Lemma~\ref{lem:indeptau}.
\end{proof}

\section{The untwisted  affine case}

\subsection{Notation}
\label{sec:defaff}

Let $\dot \lg$ be a complex finite dimensional simple Lie algebra with
Cartan subalgebra $\dlh$ and Borel subalgebra $\dlb\supset\dlh$.
Let $\dot\alpha_1,\dots,\dot\alpha_l$ denote the simple roots,
$\dot\alpha_1^\vee,\dots,\dot\alpha_l^\vee$  the simple coroots,
$\dtheta$ the highest root and $\dtheta^\vee$ the highest coroot.
For any simple root $\dot\alpha$, we denote by $\varpi_{\dot\alpha}$ the
corresponding fundamental weight and by $\varpi_{\dot\alpha^\vee}$ the
corresponding fundamental coweight.
Let $\dot P_+$ denote the set of dominant integral weights for $\dot\lg$.
Set $\dot\rho=\sum_{i=1}^l\varpi_{\dot \alpha_i}$.

Endow $\lg=\dlg\otimes \CC[z,z\inv]\oplus\CC c\oplus
\CC d$ with the usual Lie bracket (see e.g. \cite[Chap XIII]{Kumar:KacMoody}).
Set $\lh=\dlh\oplus\CC c\oplus\CC d$.
Define $\Lambda$ and $\delta$ in $\lh^*$ by
$$
\begin{array}{l}
\delta\,:\, \dot\lh\longmapsto 0, c\longmapsto 0,
d \longmapsto 1;\\
  \Lambda\,:\, \dot\lh\longmapsto 0, c\longmapsto 1,
d \longmapsto 0.
\end{array}
$$
We identify $\dlh^*$ with the orthogonal of $\CC c\oplus\CC d$ in
$\lh^*$ in such a way that $
\lh^*=\dlh^*\oplus\CC\Lambda\oplus\CC\delta$.
The simple roots of $\lg$ are 
$$
\alpha_0=\delta-\dtheta,\dot\alpha_1,\dots,\dot\alpha_l.
$$
The simple coroots of $\lg$ are 
$$
\alpha_0^\vee=c-\dtheta^\vee,\dot\alpha_1^\vee,\dots,\dot\alpha_l^\vee.
$$
For any simple root $\dot\alpha_i$ of $\dlg$, set
$\varpi_{\alpha_i}=\varpi_{\dot\alpha_i}+\varpi_{\dot\alpha_i}(\dtheta^\vee)\Lambda\in\lh^*$.
Set $\varpi_{\alpha_0}=\Lambda$.
A choice of fundamental weights for $\lg$ is
$\varpi_{\alpha_0},\dots,\varpi_{\alpha_l}$.
In particular
\begin{equation}
  \label{eq:312}
  \rho=\dot\rho+\bh^\vee\Lambda,
\end{equation}
where $\bh^\vee=1+\langle\dot\rho,\dot\theta^\vee\rangle$ is the dual
Coxeter number. 
Set
$$
\lh_\ZZ^*=\ZZ \varpi_{\alpha_0}\oplus\cdots\oplus
\ZZ \varpi_{\alpha_l}\oplus\ZZ\delta,
$$
and
$$
\begin{array}{l@{\,}l}
P_+&=\ZZ_{\geq 0}\varpi_{\alpha_0}\oplus\cdots\oplus
\ZZ_{\geq 0}\varpi_{\alpha_l}\oplus\ZZ\delta,\\[0.6em]
&=\{\dot\lambda+\bl\Lambda+b\delta\,:\,\dot\lambda\in\dot P_+\mbox{
  and }
\langle\dot\lambda,\dot\theta^\vee\rangle\leq \bl
\}.
\end{array}
$$
Denote by $P_{++}=\ZZ_{> 0}\varpi_{\alpha_0}\oplus\cdots\oplus
\ZZ_{> 0}\varpi_{\alpha_l}\oplus\ZZ\delta$, the set of {\it regular}
dominant weights.
The chosen fundamental coweights are
$$
\varpi_{\alpha_0^\vee}=d\qquad   \varpi_{\alpha_i^\vee}=\varpi_{\dot\alpha_i^\vee}+\langle
\varpi_{\dot\alpha_i^\vee},\dtheta\rangle d.
$$
Set $\dot Q^\vee=\oplus_{i=1}^l\ZZ\dot\alpha_i^\vee$. Then $W=\dot
Q^\vee.\dot W$. Moreover $h\in\dot Q^\vee$ acts on 
$\lh$
by
\begin{equation}
  \label{eq:actionWh}
h\cdot(x+kd+\bl c)=x+kh+kd+\bigg(\bl-(x,h)-k\frac{(h,h)}2\bigg )c.
\end{equation}

\subsection{Essential inequalities and BKB-product}

We are now interested in the inequlities~\eqref{eq:4} of
Proposition~\ref{prop:ineg} that are equalities for some regular
elements of $\Gamma(\lg)$. 
We prove that such inequalities necessarily appear in Theorem~\ref{th:mainintro}.

\begin{theo}\label{th:essineqBKBprod}
We use notation of Proposition~\ref{prop:ineg} and assume that
 $n_{u_1,u_2}^{v}=1$.
Let $\tau\in \oplus_{\alpha_j\not\in\Delta(P)}\ZZ_{>0}\varpi_{\alpha_j^\vee}$.
Let $(\lambda_1,\lambda_2,\mu)\in (P_{+})^3$ such that
\begin{equation}
  \label{eq:424}
\langle\lambda_1,u_1\tau\rangle+\langle\lambda_2,u_2\tau\rangle=
\langle\mu,v\tau\rangle.
\end{equation}
Assume that 
$\mu$ is regular and that
\begin{eqnarray}
  \label{eq:425}
  \exists N>0\qquad L(N\mu)\subset L(N\lambda_1)\otimes L(N\lambda_2).
\end{eqnarray}

Then, $\epsilon_v$ appears with multiplicity 1 in $\epsilon_{u_1}\bkprod\epsilon_{u_2}$.
\end{theo}


Consider the line bundle
$\Li=\Li_-(\lambda_1)\otimes \Li_-(\lambda_2)\otimes \Li(\mu)$ on
$\dX$, and  the closed subset
$C=Lu_1\inv \underline{o}^-\times Lu_2\inv
\underline{o}^-\times Lv\inv \underline{o}$ of $\dX$. 

For $i=1, 2$, consider the maximal parabolic subgroup $Q_i$ containing $B^-$
such that $\lambda_i$ extends to $Q_i$. 
Set $\ul\dX=G/Q_1\times G/Q_2\times G/B$ and
$\pi\,:\,\dX\longto\ul\dX$.
Set $\ul C=\pi(C)$ and $\ul C^+=\pi(C^+)$.
Let $\ul\Li$ be the $G$-linearized line bundle on $\underline\dX$ such that $\pi^*(\ul\Li)=\Li$.

In this paper, we denote by $C^{\rm ss}(\Li_{|C},L)$ the set of points
$x\in C$ such that there exists a $L$-invariant section $\sigma$ of
some positive power $\Li_{|C}^{\otimes N}$ of $\Li_{|C}$ such that
$\sigma(x)\neq 0$. Note that this definition is the standard one
(see~\cite{GIT}) only if $\Li_{|C}$ is ample.
Moreover, $C^{\rm ss}(\Li_{|C},L)=C\cap\pi\inv(\ul C^{\rm ss}(\ul\Li_{|\ul C},L))$.

The following lemma is a consequence of Theorem~\ref{th:egalite}. 
We include here a more direct proof. 

\begin{lemma}
\label{lem:ZLtrive}
With the assumptions of Theorem~\ref{th:essineqBKBprod}, the set $
C^{\rm ss}(\Li_{|C},L)$
is not empty. In particular, $Z(L)^\circ$ acts trivially on $\Li_{|C}$.
\end{lemma}

\begin{proof}
  By Lemma~\ref{lem:BW}, there exists a $G$-invariant section
  $\sigma\in\Ho^0(\dX,\Li^{\otimes N})^G$.
By Lemma~\ref{lem:inveta}, the image of $\eta\,:\,\GPCp\longto\dX$ contains a nonempty open
subset of $\dX$. Since $\dX$ is irreducible and $G$-invariant, we
deduce that there exists $x\in C^+$ such that $\sigma(x)\neq 0$.

Set $y=\lim_{t\to 0}\tau(t)x$. 
Consider the map $\theta_x$ defined in the proof of
Lemma~\ref{lem:rest4}.
Since $\mu^\Li(C,\tau)=0$, $\theta_x^*(\Li)$ is trivial as a
$\CC^*$-linearized line bundle.
We deduce that  
$\tilde
y:=\lim_{t\to 0}\tau(t)\sigma(x)$ exists and belongs to
$\Li_{y}-\{y\}$. 
But, $\sigma$ being $G$-invariant, $\tilde y=\sigma(y)$. 
In particular $\tilde y$ belongs to $C^{\rm ss}(\Li_{|C},L)$.

Since $Z(L)$ acts trivially on $C$ and $\sigma$ is $G$-invariant, it
fixes $\tilde y$.
Then $Z(L)$ acts trivially on $\Li^{\otimes N}_{|C}$ and $Z(L)^\circ$ acts
trivially on $\Li _{|C}$.
\end{proof}

We now prove a lemma on the points of $\ul C^\rss(\ul \Li_{|\ul C},L)$.

\begin{lemma}
  \label{lem:stabCs}
Let $x\in \ul C^\rss(\ul \Li_{|\ul C},L)$. Then
$G_x\cap P^{u,-}$ is trivial.
\end{lemma}

\begin{proof}
Fix $\tilde x\in C$ such that $\pi(\tilde x)=x$.
By Theorem~\ref{th:egalite}, one can find a $G$-invariant section $\sigma$ of
    some positive power $\Li^{\otimes N}$ of $\Li$ such that
    $\sigma(\tilde x)\neq 0$.
For any $(w_1,w_2,w_3)\in W$, $H^0(\pi(X_{w_1}^\GBm\times
X_{w_3}^\GB),\ul \Li)$ is isomorphic to  $H^0(X_{w_1}^\GBm\times
X_{w_3}^\GB,\Li)$.
In particular, $\sigma$ descends to a $G$-invariant section
$\ul\sigma$ of $\ul\Li$ on $\ul\dX$. 
The set $\ul\dX_\sigma=\{y\in \ul\dX\,:\,\ul\sigma(y)\neq 0\}$ is a $G$-stable affine
ind-variety containing $x$. 

Write $x=(l_1u_1\inv Q_1/Q_1,l_2u_2\inv Q_2/Q_2,lv\inv \underline{o})$, with
$l_1$, $l_2$ and $l$ in $L$. Then 
 $G_{x}\cap P^{u,-}$ is contained in $l(v\inv Bv\cap P^{u,-})l\inv$. 
By \cite[Example~6.1.5.b]{Kumar:KacMoody}, $v\inv Bv\cap P^{u,-}$ is a
finite dimensional unipotent group. In particular, $G_{x}\cap P^{u,-}$
is connected and it is sufficient to prove that its Lie algebra is
trivial.

Assume that there exists a nonzero vector $\xi\in\Lie(G_{x}\cap P^{u,-})$.
Consider a morphism
$$
\phi\,:\,\SL_2(\CC)\longto G,
$$
such that $T_1\phi(E)=\xi$, given by Proposition~\ref{prop:JM}.
Look the induced $\SL_2(\CC)$-action on $\ul\dX$. The unipotent subgroup
$
U_2=\begin{pmatrix}
  1&*\\0&1
\end{pmatrix}
$
of $\SL_2(\CC)$ fixes the point $ x$.
Since $\SL_2(\CC)/U_2\simeq\CC^2-\{(0,0)\}$, one gets a
regular map 
$$\bar\phi\,:\,\CC^2-\{(0,0)\}\longto \ul\dX.
$$
Since $\ul\dX_\sigma$ is $G$-stable, the image of $\bar\phi$ is contained
in $\ul\dX_\sigma$. 
Since $\ul\dX_\sigma$ is an affine ind-variety,
Arthog's lemma implies that $\bar\phi$ extends to a regular map
$$\tilde\phi\,:\,\CC^2\longto \ul\dX.
$$
By density, $\tilde\phi$ is $\SL_2(\CC)$-equivariant. 
In particular the point
$\tilde\phi(0,0)$ is fixed by $\SL_2(\CC)$. This is a contradiction, since
$v\inv \lb v$ contains no copy of $\sl_2(\CC)$.
  \end{proof} 

\begin{proof}[Proof of Theorem~\ref{th:essineqBKBprod}]
Consider the group
$X(T)^{Z(L)^\circ}$ of characters $\chi$ of $T$ such that
$\chi_{|Z(L)^\circ}$ is trivial. 
By Lemma~\ref{lem:ZLtrive}, $\Li_-(\lambda_1)\otimes
                   \Li_-(\lambda_2)\otimes\Li(\mu)$ belongs to
                   $\Pic^{L/Z(L)^\circ}(\underline C)$. 
Set
$$
\begin{array}{r@{\,}ccl}
  \gamma\,:&X(T)^{Z(L)^\circ}\otimes
             \QQ&\longto&\Pic^{L/Z(L)^\circ}(\ul C)\otimes\QQ\\
&\mu'&\longmapsto&(\Li_-(\lambda_1)\otimes
                   \Li_-(\lambda_2)\otimes\Li(\mu+\mu'))_{|\ul C}.
\end{array}
$$
The set of $\mu'$ such that some positive power of $\gamma(\mu')$ is
ample and $\ul C^\rss(\gamma(\mu'),L/Z(L)^\circ)$ is not empty is a
convex set denoted $\Cun^L(\ul C)$.
But, the image of $\gamma$ is abundant in the sense of
\cite[Section~4.1]{DH}. 
As a consequence, for $\mu'$ general in $\Cun^L(\ul C)$, there exist
stable points in $\ul C$ for $\gamma(\mu')$ and the action of
$L/Z(L)^\circ$.
Fix such a $\mu'$ and $N'$ such that $N'\mu'\in X(T)$.  Then,
$(N'\lambda_1,N'\lambda_2,N'(\mu+\mu'))$ still satisfies
equality~\eqref{eq:424}. 
Moreover, Theorem~\ref{th:egalite} and the descent argument at the
beginning of the proof of Lemma~\ref{lem:stabCs} show that it also satisfies
condition~\eqref{eq:425} for some $N$.
Then, to prove the theorem, one may assume that 
$\ul C^\rs(\ul\Li,L/Z(L)^\circ)$ is not empty.

  The graded algebra $\oplus_k\Ho^0(\ul C,\ul\Li_{|\ul C}^{\otimes k})^L$ is finitely
  generated. Fix $d>0$ such  that  $\Ho^0(\ul C,\ul\Li_{|\ul C}^{\otimes 
d})^L$ generates $\oplus_k\Ho^0(\ul C,\ul\Li_{|\ul C}^{\otimes dk})^L$.
Let $\sigma_0,\dots,\sigma_N$ be a $\CC$-basis of  $\Ho^0(\ul C,\ul
\Li_{|\ul C}^{\otimes 
d})^L$.
By Theorem~\ref{th:egalite}, for any $i$, $\sigma_i$ extends to a
$G$-invariant section $\tilde\sigma_i$ on $\ul\dX$. Set
$$
\ul\dX^{\rm ss}(\Li):=\{x\in \ul\dX\,:\,\exists i \quad\tilde\sigma_i(x)\neq 0\},
$$
and 
$$
\begin{array}{lccl}
  \pi\,:&\ul\dX^{\rm ss}(\ul\Li)&\longto&\CC\PP^N\\
&x&\longmapsto&[\tilde\sigma_0(x):\cdots:\tilde\sigma_N(x)].
\end{array}
$$
Consider now $\ul\eta\,:\,G\times_P \ul C^+\longto \ul\dX$.
Let $\ul C^{+,\,\rm ss}(\ul \Li,L)$ denote the set of points $y\in \ul C^+$ such that
$\lim_{t\to 0}\tau(t)y\in \ul C^{\rm ss}(\ul \Li,L)$. 
Since the $\tilde\sigma_i$ are $G$-invariant, $\ul\eta(G\times_P\ul
C^{+,\,\rm  ss}(\ul \Li,L))$ is contained in $\dX^{\rm ss}(\Li)$. By the proof of
Proposition~\ref{prop:ineg}, this set is dense in $\ul\dX^{\rm ss}(\ul\Li)$, and
hence in $\ul\dX$ by irreducibility.
But $\pi(\ul\eta((gP/P,x)))=\pi(g\inv x)$ for any $g\in G$, $x\in\ul\dX^{\rm
  ss}(\ul\Li,L)$ such that $ (gP/P,x)\in G\times_P\ul C$. 
Then $\pi\circ\ul\eta(G\times_P\ul C^{+,\rm
  ss}(\ul \Li),L)$ is contained in $\pi(\ul C^{\rm ss}(\ul\Li,L))=\ul C^{\rm
  ss}(\ul\Li,L)\quot L$. 
Since $\ul C^{\rm
  ss}(\ul\Li,L)\quot L$ is projective, this implies that 
$$
\pi(\ul\dX^{\rm ss}(\ul\Li))=\ul C^{\rm
  ss}(\ul\Li,L)\quot L:=\Proj(\oplus_k\Ho^0(\ul C,\ul\Li_{|\ul C}^{\otimes k})^L).
$$ 

Let us fix $\xi\in \ul C^{\rm ss}(\ul\Li)\quot L$ general. Since
$\ul C^s(\ul\Li_{|\ul C},L/Z(L))$ is not empty, $\pi\inv(\xi)\cap \ul C$ is an
$L$-orbit $L.x_0$. 
Since $L.(u_2\inv Q_2/Q_2,v\inv \underline{o})$ is open in 
$Lu_2\inv Q_2/Q_2\times Lv\inv \underline{o}$, one may assume that
$x_0=(l_1u_1\inv Q_1/Q_1,u_2\inv Q_2/Q_2,v\inv
\underline{o})$, for some $l_1\in L$.
Let $i$ such that $\sigma_i(x_0)\neq 0$.

We claim that $\ul\eta\inv(x_0)$ is finite.
Let $g\in G$ such that $(gP/P,x_0)\in \ul\eta\inv(x_0)$. 
Set $y=g\inv x_0\in \ul C^+$ and $z=\lim_{t\to 0}\tau(t)y$.
Consider 
$\tilde y:=\tilde\sigma_i(y)=g\inv \tilde\sigma_i(x_0)\in (\ul\Li^{\otimes
  d})_y-\{y\}$.
Since $\mu^\Li(C,\tau)=0$, $\tilde z:=\lim_{t\to 0}\tau(t)\tilde y\in (\Li^{\otimes
  d})_z-\{z\}$.
  But $\tilde\sigma_i$ being $G$-invariant, $\tilde\sigma_i(z)=\tilde z$.
Hence $z\in \ul\dX^{\rm ss}(\ul\Li)\cap \ul C=\ul C^{\rm ss}(\Li)$.

Moreover, $\pi$ being $G$-invariant, we have $\pi(z)=\pi(x_0)$. 
Hence $z\in L.x_0$. Then Lemma~\ref{lem:Cplushom} implies that
$y=g\inv x_0\in Px_0$. Hence $g\in G_{x_0}P$. 

Note that $G_{x_0}\subset u_2\inv Q_2u_2\cap v\inv Bv$ is finite
dimensional. Moreover, it contains $\tau(\CC^*)$. Then, the Lie
algebra of $G_{x_0}$ decomposes as 
$\Lie(G_{x_0})=(\Lie(G_{x_0})\cap \Lie(P^{u,-})\oplus
(\Lie(G_{x_0})\cap Lie(P))$. Using Lemma~\ref{lem:stabCs}, we deduce
that the neutral component $G_{x_0}^\circ$ of $G_{x_0}$ is contained
in $P$. Then $G_{x_0}P/P$ is finite and the claim is proved.

\bigskip
Consider now  $\eta\,:\,G\times_P C^+\longto \dX$.
Let $\tilde x_0\in C$ such that $\pi(\tilde x_0)=x_0$. 
Since $\pi\inv(\ul C^+)=C^+$, the claim implies that $\eta\inv(\tilde
x_0)$ is finite.
With Lemma~\ref{lem:fibreeta}, the claim implies that
$lu_1\inv\cX_{u_1}^{G/P}\cap u_2\inv\cX_{u_2}^{G/P}\cap
v\inv\cX_v^{G/P}$ is finite. 
Then Proposition~\ref{prop:LmovBKBprod} allows to conclude.
\end{proof}

\subsection{About \texorpdfstring{$\Gamma(\lg)$}{the tensor cone}}

For any $n\in \ZZ$, $L(n\delta)$ is one dimensional acted on by the
character $n\delta$ of $\lg$. It follows that 
$$\Gamma(\lg)=\Gamma_{\rm red}(\lg)+\QQ(\delta,0,\delta)+\QQ(0,\delta,\delta),$$ 
where
$$
\Gamma_{\rm red}(\lg)=\{(\lambda_1,\lambda_2,\mu)\in\Gamma(\lg)\,:\,\lambda_1(d)=\lambda_2(d)=0\}.
$$
For any $\lambda\in P_+$,
the center $\CC c$ of $\lg$ acts on $L(\lambda)$ with weight
$\lambda(c)\in\ZZ$. Then
$$
\Gamma_{\rm red}(\lg)\subset
\Gamma(\lg)\subset\{(\lambda_1,\lambda_2,\mu)\in(\lh^*_\QQ)^3\,:\, \mu(c)=\lambda_1(c)+\lambda_2(c)\}.
$$

As an application of the GKO construction \cite{GKO} of
representations of Virasoro algebras, Kac-Wakimoto obtained in \cite{KW} the following
properties of decomposition of $L(\lambda_1)\otimes
L(\lambda_2)$.

\begin{lemma}
\label{lem:Gammaepigraph}
Let $\lambda_1$, $\lambda_2$ in $P_+$ such that
$\lambda_1(d)=\lambda_2(d)=0$, $\lambda_1(c)>0$ and $\lambda_2(c)>0$.
Let $\dot \mu\in\dot P$ and set
$\bar\mu:=\dot\mu+(\lambda_1(c)+\lambda_2(c))\Lambda\in P_+$.
 
Then, there exists $b\in\ZZ$ such that  $L(\bar \mu+b\delta)$ is a
submodule of $L(\lambda_1)\otimes L(\lambda_2)$ if and only if
$\bar\mu-\lambda_1-\lambda_2\in Q$.

Moreover, if
$\bar\mu-\lambda_1-\lambda_2\in Q$ then one of the following two assertions holds:
\begin{enumerate}
\item there exists $b_0\in \ZZ$ such that
   $
L(\bar \mu+b\delta)\subset L(\lambda_1)\otimes L(\lambda_2)
$
if and only if  
$
b\leq b_0
$;
\item   there exists $b_0\in \ZZ$ such that $
L(\bar \mu+b\delta)\subset L(\lambda_1)\otimes L(\lambda_2)
$
if and only if  
$
b=b_0$ or $b\in b_0-\ZZ_{>1}
$.
\end{enumerate}
Set
$$
b_0(\lambda_1,\lambda_2,\bar\mu)=\max\{b\in\ZZ\,:\,L(\bar \mu+b\delta)\subset L(\lambda_1)\otimes L(\lambda_2)\}.
$$
\end{lemma}

\begin{proof}
  The first assertion is proved in \cite[p.~194]{KW}. 
The fact that $\{b\in\ZZ\,:\,L(\bar \mu+b\delta)\subset L(\lambda_1)\otimes L(\lambda_2)\}
$ has an upper bound is proved on \cite[p.~171]{KW}.
Let $b_0$ be the maximum of such $b\in\ZZ$.
It remains to prove that $b_0-n\delta$ for all $n\geq 2$.
This is a direct consequence of \cite[Proof of Proposition~3.2]{KW}.
See also \cite[Proposition~4.2]{BrownKumar}.
\end{proof}

\begin{NB}
\begin{enumerate}
\item
 \cite[Inequality~2.4.1]{KW} implies that 
$$
b_0\leq
\frac{(\dot\lambda_1+2\dot\rho,\dot\lambda_1)}{2(\bl_1+\bh^\vee)}+
\frac{(\dot\lambda_2+2\dot\rho,\dot\lambda_2)}{2(\bl_2+\bh^\vee)}
-\frac{(\dot\mu+2\dot\rho,\dot\mu)}{2(\bl_1+\bl_2+\bh^\vee)}.
$$
This inequality is quadratic is $(\lambda_1,\lambda_2,\bar\mu)$. In
this paper, we show stronger linear inequalities.  
\item
 If one takes $\bl_1=0$ in Lemma~\ref{lem:Gammaepigraph}, one get
  $\dot\lambda_1=0$. 
Set $\lambda_1=\dot\lambda_1+\bl_1\Lambda=0$,
$\lambda_2=\dot\lambda_2+\bl_2\Lambda$ and
$\mu=\dot\mu+\bl_2\Lambda$. 
We have $L(N0)\otimes L(N\lambda_2)=L(N\lambda_2)$, for
  any positive integer $N$. Hence
  $(0,\lambda_2,\mu)$ belongs to $\Gamma(\lg)$ if and only if
  $\mu=\lambda_2$. 
In particular, the assumption ``$\bl_1$ positive'' is necessary in Lemma~\ref{lem:Gammaepigraph}.
Observe that this implies that $\Gamma(\lg)$ is not closed.
\end{enumerate}
\end{NB}

\bigskip
Set 
$$
\Gamma_{\rm red}^\circ(\lg)=\{(\lambda_1,\lambda_2,\mu)\in
\Gamma_{\rm red}(\lg)\,:\,\lambda_1(c)>0\quad{\rm and}\quad \lambda_2(c)>0\},
$$
and
$$
\begin{array}{lll}
\Ac=\{(\lambda_1,\lambda_2,\bar\mu)\in (X(\dot
T)_\QQ\oplus\QQ\Lambda)^3\,:&
\lambda_1,\lambda_2\mbox{ and }\bar\mu
\mbox{ are dominant}\\
&\lambda_1(c)>0,\,\lambda_2(c)>0\\
&\bar\mu(c)=\lambda_1(c)+\lambda_2(c)&\}.
\end{array}
$$
Define a function $\Psi\,:\,\Ac\longto\RR$ by
$$
\Psi (\lambda_1,\lambda_2,\bar\mu)=\sup_{
  \begin{array}{c}
    N\in\ZZ_{>0}\mbox{ s.t.}\\
N\lambda_1,N\lambda_2,N\bar \mu\in\lh^*_\ZZ\\
N\bar\mu-N\lambda_1-N\lambda_2\in Q
  \end{array}}
\frac{b_0(N\lambda_1,N\lambda_2,N\bar\mu)}{N},
$$
where $b_0$ is defined in Lemma~\ref{lem:Gammaepigraph}.
This lemma implies that $(\lambda_1,\lambda_2,\mu+b\delta)$ belongs to
the closure of  $\Gamma_{\rm red}^\circ(\lg)$ in $\Ac\times\RR$ if and only
if
$$
b\leq \Psi (\lambda_1,\lambda_2,\bar\mu).
$$

\subsection{A cone defined by inequalities}

Consider the cone $\Cun$ of points 
$(\lambda_1,\lambda_2,\mu)\in(\lh^*_\QQ)^3$
such that 
\begin{enumerate}
\item $\lambda_1(c)>0$ and $\lambda_2(c)>0$;
\item $\lambda_1,\lambda_2$, and $\mu$ are dominant;
\item $\lambda_1(d)=\lambda_2(d)=0$;
\item $\lambda_1(c)+\lambda_2(c)=\mu(c)$;
\item the inequality 
\begin{eqnarray}
  \label{eq:1}
  \langle\mu,v\varpi_{\alpha_i^\vee}\rangle\leq
\langle\lambda_1,u_1\varpi_{\alpha_i^\vee}\rangle+
\langle\lambda_2,u_2\varpi_{\alpha_i^\vee}\rangle
\end{eqnarray}
holds for any $i\in\{0,\dots,l\}$ and any $(u_1,u_2,v)\in W^{P_i}$ such that
\begin{eqnarray}
  \label{eq:2}
  {}^\bkprod n_{u_1u_2}^v=1 \qquad\mbox{in } \Ho^*(G/P_i,\ZZ).
\end{eqnarray}
\end{enumerate}

The aim of this section is to prove Theorem~\ref{th:mainintro} or equivalently
Theorem~\ref{th:GammaC} below. We first study the cone $\Cun$. 

\subsection{Realisation of \texorpdfstring{$\Cun$}{the cone} as an hypograph}

For $\mu\in X(\dot
T)_\QQ\oplus\QQ\Lambda\oplus \QQ\delta$, we denote by $\dot\mu$
(resp. $\bar\mu$) its
projection on $X(\dot T)_\QQ$ (resp. $X(\dot
T)_\QQ\oplus\QQ\Lambda$). 

Let $\cI$ be the set of $(u_1,u_2,v,i)\in (W^{P_i})^3\times\{0,\dots
l\}$ satisfying condition~\eqref{eq:2}. 
Fix $(u_1,u_2,v,i)\in \cI$. Assume first that $i=0$. Let $h_1,h_2$ and $h$ in $\dot Q^\vee$ such that
$u_1W_{P_0}=h_1W_{P_0}$, $u_2W_{P_0}=h_2W_{P_0}$ and
$vW_{P_0}=hW_{P_0}$.
Define the (restriction of) linear function
$\varphi_{(u_1,u_2,v,0)}\,:\,\Ac\longto\QQ$ that maps 
$(\lambda_1,\lambda_2,\bar\mu)$ to 
\begin{multline}\label{eq:3}
\langle h_1,\dot\lambda_1\rangle+\langle
          h_2,\dot\lambda_2\rangle-\langle h,\dot\mu\rangle\\
+\frac{\bl_1}2(\Vert h\Vert^2-\Vert h_1\Vert^2)
+\frac{\bl_2}2(\Vert h\Vert^2-\Vert h_2\Vert^2),
\end{multline}
where $\bl_1=\lambda_1(c)$ and $\bl_2=\lambda_2(c)$.
Note  that $\varpi_{\alpha_0^\vee}=d$ and for $h\in \dot Q^\vee$, 
by equation~\eqref{eq:actionWh}, we have
$
h\cdot d=h+d-\frac{(h,h)}2c.
$
Let $(\lambda_1,\lambda_2,\bar\mu)\in \Ac$.
Then inequality~\eqref{eq:1} with $i=0$, is fulfilled by
$(\lambda_1,\lambda_2,\bar\mu+b\delta)$ if and only if   
$$
b\leq \varphi_{(u_1,u_2,v,0)}(\lambda_1,\lambda_2,\bar\mu).
$$
Assume now that $i\in\{1,\dots,l\}$. Write
$u_1=\dot u_1h_1$, $u_2=\dot u_2h_2$ and $v=\dot vh$ with $\dot
u_1,\dot u_2,\dot v\in \dot W$ and $h_1,h_2,h\in \dot Q^\vee$. 
Define the linear function
$\varphi_{(u_1,u_2,v,i)}\,:\,\Ac\longto\QQ$ that maps 
$(\lambda_1,\lambda_2,\bar\mu)$ to 
\begin{equation}
  \label{eq:33}
\begin{array}{l}
\langle \dot u_1(h_1+\frac{\varpi_{\alpha_i^\vee}}{\langle
\dvarpi_{\alpha_i^\vee},\dtheta\rangle}),\dot\lambda_1\rangle+\langle
          \dot u_2(h_2+\frac{\varpi_{\alpha_i^\vee}}{\langle
\dvarpi_{\alpha_i^\vee},\dtheta\rangle}),\dot\lambda_2\rangle-\langle \dot v(h+\frac{\varpi_{\alpha_i^\vee}}{\langle
\dvarpi_{\alpha_i^\vee},\dtheta\rangle}),\dot\mu\rangle\\
+\frac{\bl_1}2(\Vert h\Vert^2-\Vert
h_1\Vert^2+2\frac{(\varpi_{\alpha_i^\vee},h-h_1)}{\langle
\dvarpi_{\alpha_i^\vee},\dtheta\rangle})\\
+\frac{\bl_2}2(\Vert h\Vert^2-\Vert h_2\Vert^2+2\frac{(\varpi_{\alpha_i^\vee},h-h_2)}{\langle
\dvarpi_{\alpha_i^\vee},\dtheta\rangle}).
\end{array}
\end{equation}
Recall that $\varpi_{\alpha_i^\vee}=\dvarpi_{\alpha_i^\vee}+\langle
\dvarpi_{\alpha_i^\vee},\dtheta\rangle d$. Moreover, for $w=\dot wh\in W$, by equation~\eqref{eq:actionWh}, we have
$$
(\dot wh)\cdot \varpi_{\alpha_i^\vee}=\dot w \dvarpi_{\alpha_i^\vee}+\langle
\dvarpi_{\alpha_i^\vee},\dtheta\rangle\dot wh+
\langle
\dvarpi_{\alpha_i^\vee},\dtheta\rangle  d-\bigg (\langle
\dvarpi_{\alpha_i^\vee},\dtheta\rangle\frac{(h,h)}2+(\dvarpi_{\alpha_i^\vee},h)\bigg)c.
$$
Then inequality~\eqref{eq:1}, is fulfilled by
$(\lambda_1,\lambda_2,\bar\mu+b\delta)$ if and only if   
$$
b\leq \varphi_{(u_1,u_2,v,i)}(\lambda_1,\lambda_2,\bar\mu).
$$

Define
$$
\begin{array}{lccl}
  \varphi\,:&\Ac&\longto&\RR\cup\{-\infty\}\\
&(\lambda_1,\lambda_2,\bar\mu)&\longmapsto&\inf_{(u_1,u_2,v,i)\in \cI}\varphi_{(u_1,u_2,v,i)}(\lambda_1,\lambda_2,\bar\mu).
\end{array}
$$
Then $\varphi$ is a concave function and 
$\Cun$  is the hypograph of
$\varphi$:

$$
\Cun=\{(\lambda_1,\lambda_2,\bar\mu+b\delta)\,:\,(\lambda_1,\lambda_2,\bar\mu)\in
\Ac\mbox{ and }b\leq \varphi(\lambda_1,\lambda_2,\bar\mu)\}.
$$

\subsection{The convex set \texorpdfstring{$\Cun$}{cone} is locally polyhedral} 

\begin{prop}
\label{prop:locpol}
Let $x_0\in\Ac$. Then
  $$\forall M\in \RR\quad\exists \mbox{ an open }U\ni x_0\quad \exists
  J\subset\cI\mbox{ finite}
$$
such that 
$$
\forall x\in U\qquad \forall a\in \cI\qquad
\varphi_a(x)<M\,\Rightarrow\,a\in J.
$$
\end{prop}

\begin{proof}
Fix $M\in\RR$.
Let $(u_1,u_2,v,0)\in \cI$.
By Proposition~\ref{prop:ineg}, the inequality~\eqref{eq:1} is satisfied
by $(\bl\Lambda,0,\bl\Lambda)$ for any $\bl>0$. Hence
$$
\Vert h\Vert^2-\Vert h_1\Vert^2\geq 0.
$$
Similarly, $
\Vert h\Vert^2-\Vert h_2\Vert^2\geq 0
$.
Hence, the image of $(\lambda_1,\lambda_2,\bar \mu)\in\Ac$ by
$\varphi_{(u_1,u_2,v,0)} $ is greater or equal to
\begin{equation}\label{eq:lb0}
\langle h_1,\dot\lambda_1\rangle+\langle
          h_2,\dot\lambda_2\rangle-\langle h,\dot\mu\rangle
+\frac{\bl}2(2\Vert h\Vert^2-\Vert h_1\Vert^2-\Vert h_2\Vert^2),
\end{equation}
where $\bl=\min(\lambda_1(c),\lambda_2(c))$.
Using Lemma~\ref{lem:calculdelta}, one gets
$\delta_{u_1\,u_2}^v=\langle\rho,-d-h\cdot d+h_1\cdot d+h_2\cdot d\rangle$.
Hence
$$
\delta_{u_1\,u_2}^v=\langle\dot\rho,h_1+h_2-h\rangle+
\frac{\bh^\vee}2
(\Vert h\Vert^2-\Vert h_1\Vert^2-\Vert h_2\Vert^2).
$$
Lemma~\ref{lem:deltaneg} implies that $
\delta_{u_1\,u_2}^v\geq 0$ and
$$
\Vert h\Vert^2-\Vert h_1\Vert^2-\Vert h_2\Vert^2\geq\frac
2{\bh^\vee}\langle h-h_1-h_2,\dot
\rho\rangle.
$$
Then the image of $(\lambda_1,\lambda_2,\bar\mu)\in\Ac$ by
$\varphi_{(u_1,u_2,v,0)} $ is greater or equal to
\begin{equation}
  \label{eq:lb1}
\frac{\bl}2\Vert h\Vert^2
-\Vert h_1\Vert. \Vert \dot\lambda_1\Vert-\Vert h_2\Vert. \Vert
\dot\lambda_2\Vert- \Vert h\Vert. \Vert \dot\mu\Vert-
\frac{2 \Vert \dot\rho\Vert}{\bh^\vee}
(\Vert h_1\Vert+ \Vert h_2\Vert+\Vert h\Vert).
\end{equation}
By construction there exist $\dot u_1$, $\dot u_2$ and $\dot v$ in
$\dot W$ such that $u_1=h_1\dot u_1$, $u_2=h_2\dot u_2$ and $v=h\dot v$.
But $l(v)=l(u_1)+l(u_2)$. Then Lemma~\ref{lem:lvsnorm} implies that 
$$
N+\sqrt 2N\Vert h\Vert\geq
l(v)\geq l(u_1)\geq
K\Vert h_1\Vert-N.
$$
Then
\begin{equation}
  \label{eq:maxh}
\max(\Vert h_1\Vert,\Vert h_2\Vert)\leq\frac{N} K (2+\sqrt 2 \Vert
  h\Vert).
\end{equation}
The point is that this implies that 
$\varphi_{(u_1,u_2,v,0)}(\lambda_1,\lambda_2,\bar\mu)$ is greater or equal
to $\frac{\bl}2\Vert h\Vert^2$ minus terms that are linear in $\Vert h\Vert$.
 One can easily deduce that there exist an open neighborhood $U_0$ of
 $x_0$ and $A_0\in\RR$ such that 
$$
\forall x\in U_0\quad\forall a=(u_1,u_2,v,0)\in\cI\qquad l(v)\geq A_0\,\Rightarrow\,
\varphi_a(x)\geq M.
$$ 

\bigskip
Let $(u_1,u_2,v,i)\in \cI$ with $i>0$ and consider the associated linear
function $\varphi_{(u_1,u_2,v,i)}$. Since
$(\bl\Lambda,0,\bl\Lambda)\in\Gamma(\lg)$ for any $\bl>0$,
Proposition~\ref{prop:ineg} implies that 
$$
\Vert h\Vert^2-\Vert h_1\Vert^2+2\frac{(\varpi_{\alpha_i^\vee},h-h_1)}{\langle
\dvarpi_{\alpha_i^\vee},\dtheta\rangle}\geq 0.
$$
Set
$C:=\frac{\Vert\dvarpi_{\alpha_i^\vee}\Vert}{\langle\dvarpi_{\alpha_i^\vee},\dot\theta\rangle}$
and 
$D=\langle\dvarpi_{\alpha_i^\vee},\dot\theta\rangle$.
Then, the image of $(\lambda_1,\lambda_2,\bar \mu)\in\Ac$ by
$\varphi_{(u_1,u_2,v,i)} $ is greater or equal to
\begin{equation}\label{eq:lb6}
  \begin{array}{r@{\,}l}
    -&(\Vert h_1\Vert .\Vert \dot\lambda_1\Vert+\Vert h_2\Vert .\Vert
\dot\lambda_2\Vert
+\Vert h\Vert .\Vert \dot\mu\Vert)
-C(\Vert \dot\lambda_1\Vert+\Vert \dot\lambda_2\Vert+\Vert
\dot\mu\Vert)\\
+\frac{\bl}2&(2\Vert h\Vert^2-\Vert h_1\Vert^2-\Vert h_2\Vert^2
-2C(\Vert h-h_1\Vert+\Vert h-h_2\Vert)).
  \end{array}
\end{equation}
where $\bl=\min(\lambda_1(c),\lambda_2(c))$.
But $
\delta_{u_1\,u_2}^v\geq 0$ implies that
$$
\begin{array}{r@{\,}l}

\Vert h\Vert^2-\Vert h_1\Vert^2-\Vert h_2\Vert^2\geq&
\frac 2{D\bh^\vee}(1-\langle\dot\rho,
-\dot v\dvarpi_{\alpha_i^\vee}+\dot u_1\dvarpi_{\alpha_i^\vee}+\dot u_2\dvarpi_{\alpha_i^\vee}
\rangle)\\[0.7em]
&
-\frac 2{\bh^\vee}(\langle\dot\rho,
-\dot vh+\dot u_1h_1+\dot u_2h_2
\rangle)\\[0.7em]
&
-\frac 2{D}(\dvarpi_{\alpha_i^\vee},h-h_1-h_2).
\end{array}
$$
Combining these inequalities with inequality~\eqref{eq:maxh} one can get
a lower bound for $\varphi_{(u_1,u_2,v,i)}(\lambda_1,\lambda_2,\bar\mu)$  equals
to $\frac{\bl}2\Vert h\Vert^2$ minus terms that are linear in $\Vert
h\Vert$.
One can easily deduce that there exist an open neighborhood $U$ of
 $x_0$ and $A\in\RR$ such that 
$$
\forall x\in U\quad\forall a=(u_1,u_2,v,i)\in\cI\qquad 
l(v)\geq A\,\Rightarrow\,
\varphi_a(x)\geq M.
$$ 

\bigskip
But there exist only finitely many triples $(h_1\dot
u_1,h_2\dot u_2,h\dot v)$ with $l(v)<A$ and 
$l(v)=l(
u_1)+l(u_2)$. 
The proposition follows.
\end{proof}

\bigskip
\begin{NB}
  Proposition~\ref{prop:locpol} is still true with the family of
  equations corresponding to $P$ any parabolic subgroup,
  $\alpha_i\in\Delta-\Delta(P)$ and $n_{u_1\,u_2}^v\neq 0$. The same
  proof works.  
\end{NB}

\bigskip
For any $(u_1,u_2,v,i)\in\cI$, we set
$$
\Ac_{(u_1,u_2,v,i)}=\{(\lambda_1,\lambda_2,\bar\mu)\in
\Ac\,:\,\varphi(\lambda_1,\lambda_2,\bar\mu)=\varphi_{(u_1,u_2,v,i)}(\lambda_1,\lambda_2,\bar\mu)
\}.
$$

The properties of these sets are summarized in the following lemma.
The {\it dimension} of a convex set is the dimension of the generated
affine space.

\begin{lemma}
  With above notation, 
  \begin{enumerate}
  \item The sets $\Ac_{(u_1,u_2,v,i)}$ are convex.
\item Set 
$
\cI_1=\{(u_1,u_2,v,i)\in \cI\,|\,\dim(\Ac_{(u_1,u_2,v,i)})=\dim(\Ac)\}.
$
Then
\begin{equation}
  \label{eq:decA}
  \Ac=\cup_{(u_1,u_2,v,i)\in \cI_1}\Ac_{(u_1,u_2,v,i)}.
\end{equation}
\item The sets associated to two different elements of $\cI_1$  only intersect along their boundaries. 
  \end{enumerate}
\end{lemma}

\begin{proof}
The first assertion follows from the linearity of functions
$\varphi_{(u_1,u_2,v,i)}$.
  Proposition~\ref{prop:locpol} implies that $\Cun$ is locally
  polyhedral and the two last assertions.
\end{proof}

For any $(u_1,u_2,v,i)\in \cI$, we set
\begin{equation}
  \label{eq:defFi}
  \begin{array}{ll}
  \Face_{(u_1,u_2,v,i)}=\{(\lambda_1,\lambda_2,\bar\mu+b\delta)\,:\,&(\lambda_1,\lambda_2,\bar\mu)\in\Ac_{(u_1,u_2,v,i)}\\
&b=\varphi_{(u_1,u_2,v,i)}(\lambda_1,\lambda_2,\bar\mu)\}.
  \end{array}
\end{equation}
Then
$$
\{\Face_{(u_1,u_2,v,i)}\,:\,(u_1,u_2,v,i)\in \cI_1\}
$$
is the collection of codimension one faces of $\Cun$.
Moreover,
$$
\begin{array}{ll}
\Cun=\{(\lambda_1,\lambda_2,\bar\mu+b\delta)\,:\,&(\lambda_1,\lambda_2,\bar\mu)\in\Ac\\
&b\leq
\varphi_{(u_1,u_2,v,i)}(\lambda_1,\lambda_2,\bar\mu)\qquad\forall
(u_1,u_2,v,i)\in\cI_1\}.
\end{array}
$$
Two convex sets $\Ac_{(u_1,u_2,v,i)}$ and $\Ac_{(u'_1,u'_2,v',i)}$ are
said to be {\it adjacent} if their intersection has codimension 1.

\subsection{An example of a face of codimension 1}

Consider the element $(e,e,e,0)\in\cI$. The associated
inequality~\eqref{eq:1} is $b\leq 0$.
Moreover, $G/P_0$ is the affine Grassmannian $\Gr_{\dot G}$, the
semi-simple component of the Levi
subgroup $L_0$ is $\dot G$. 

\begin{lemma}
  \label{lem:face0}
Let
$(\dot\lambda_1+\bl_1\Lambda,\dot\lambda_2+\bl_2\Lambda,\dot\mu+(\bl_1+\bl_2)\Lambda)\in
(P_+)^3$.
Then $L(\dot\mu+(\bl_1+\bl_2)\Lambda)$ is contained in
$L(\dot\lambda_1+\bl_1\Lambda)\otimes L(\dot\lambda_2+\bl_2\Lambda)$ if
and only if $L_{\dot G}(\dot\mu)$ is contained in
$L_{\dot G}(\dot\lambda_1)\otimes L_{\dot G}(\dot\lambda_2)$.

In particular, $\Ac_{(e,e,e,0)}$ has nonempty interior in $\Ac$.
\end{lemma}

\begin{proof}
  The first assertion is certainly well known. It can also be obtained
  as a consequence of Theorem~\ref{th:egalite}. Indeed, in
  $\Ho^*(\Gr_{\dot G},\ZZ)$, we have $n_{e,e}^e=1$. 
For $\tau=\varpi_{\alpha_0^\vee}$,
$(\dot\lambda_1+\bl_1\Lambda,\dot\lambda_2+\bl_2\Lambda,\dot\mu+(\bl_1+\bl_2)\Lambda)$
satisfies equality~\eqref{eq:egalite}.
Corollary~\ref{cor:G2L} shows that the multiplicity of 
$L(\dot\mu+(\bl_1+\bl_2)\Lambda)$ in
$L(\dot\lambda_1+\bl_1\Lambda)\otimes L(\dot\lambda_2+\bl_2\Lambda)$
is equal to those of $L_{\dot G}(\dot\mu)$ in
$L_{\dot G}(\dot\lambda_1)\otimes L_{\dot G}(\dot\lambda_2)$. The first
assertion of the lemma follows.

It is well known (see e.g. \cite[Theorem~1.4]{PR:example}) that
$\Gamma(\dot\lg)$ has nonempty interior in $(X(\dot T)_\QQ)^3$. 
But for any given
$(\dot\lambda_1,\dot\lambda_2,\dot\mu)\in\Gamma(\dot\lg)$, 
$(\dot\lambda_1+\bl_1\Lambda,\dot\lambda_2+\bl_2\Lambda,\dot\mu+(\bl_1+\bl_2)\Lambda)\in\Gamma(\lg)$
for any $\bl_1\geq\langle\dot\lambda_1,\dot\theta^\vee\rangle$,
$\bl_2\geq\langle\dot\lambda_2,\dot\theta^\vee\rangle$ and $\bl_1+\bl_2\geq\langle\dot\mu,\dot\theta^\vee\rangle$.
The second assertion follows. 
\end{proof}

\subsection{The main result}

We can now restate our description $\Gamma^\circ(\lg)$.

\begin{theo}
\label{th:GammaC}
  With the above notation, we have
$$
\Gamma_{\rm red}^\circ(\lg)=\Cun.
$$
\end{theo}

\begin{proof}
The inclusion $\Gamma_{\rm red}^\circ(\lg)\subset \Cun$ is a direct
consequence of Proposition~\ref{prop:ineg}.
We have to prove that $\Cun$ is contained in $\Gamma(\lg)$.
Let $x_0\in\Ac$.
By Lemma~\ref{lem:Gammaepigraph}, to show that
$x_0+]-\infty,\varphi(x)](0,0,\delta)$ is contained in $\Gamma(\lg)$,
it is sufficient to prove that $x_0+\varphi(x)(0,0,\delta)$ belongs to
$\Gamma(\lg)$.

By Proposition~\ref{prop:locpol}, there exists $a_0\in\cI$ such that
$\varphi(x_0)=\varphi_{a_0}(x_0)$.
Let $U$ be an open neighborhood of $x_0$ in $\Ac$ such that for any
$y\in U$, $ \varphi_{a_0}(y)\leq \varphi(x_0)+1$. 
By Proposition~\ref{prop:locpol}, there exists a smaller neighborhood $V\subset U$
of $x_0$ and $J\subset\cI$ finite such that
$$
\forall y\in V\qquad\forall a\in\cI-J\qquad
 \varphi_{a}(y)>\varphi(x_0)+1\geq\varphi_{a_0}(y).
$$
Note that $a_0\in J$. Then, for any $y\in V$
$$
\varphi(y)=\min_{a\in J} \varphi_{a}(y).
$$
Choose a simplex ${\mathcal S}$ containing $x_0$ in its interior such that
${\mathcal S}\cap\Ac\subset V$. Up to replacing $V$ by ${\mathcal
  S}\cap\Ac$, one may assume that $V$ is a convex polytope. 
Then, the theory of polyhedrons implies that,
for any $y\in V$
$$
\varphi(y)=\min_{a\in J\cap\cI_1} \varphi_{a}(y).
$$
In particular, it is sufficient to prove that 
\begin{equation}
  \label{eq:facein}
\forall (u_1,u_2,v,i)\in \cI_1\qquad
\Face_{(u_1,u_2,v,i)}\subset\Gamma_{\rm red}^\circ(\lg).
\end{equation}
Let $\cI_1^0$ denote the set of elements in $\cI_1$ satisfying
inclusion~\eqref{eq:facein}.

\bigskip
Fix an element $ (u_1,u_2,v,i)\in \cI$. Set $C=L_iu_1\inv
\underline{o}^-\times L_iu_2\inv \underline{o}^-\times L_iv\inv \underline{o}$ the component of
$\varpi_{\alpha_i^\vee}$-fixed points in $\dX$.
Let $\Gamma(C,L_i)$ denote the set of
$(\lambda_1,\lambda_2,\bar\mu+b\delta)\in \Ac\times\QQ(0,0,\delta)$ such that
$C^{\rm
  ss}(\Li,L_i)\neq\emptyset$ (where $\Li$ is the $G$-linearized line
bundle on $\dX$ associated to $(\lambda_1,\lambda_2,\bar\mu+b\delta)$).
Observe that, since $\varpi_{\alpha_i^\vee}$ acts trivially on
$\Li_{|C}$, we have
$b=\varphi_{(u_1,u_2,v,i)}(\lambda_1,\lambda_2,\bar\mu)$ if $C^{\rm
  ss}(\Li,L_i)\neq\emptyset$.

\bigskip
Consider the following assumption
\begin{center}
  (H)\hspace{1cm} $\Face_{(u_1,u_2,v,i)}\cap(P^{++}_\QQ)^3\cap
  \Gamma(\lg)$ is not empty.
\end{center}

The proof proceed in three steps:
\begin{enumerate}
\item[Claim 1.] Under assumption $(H)$, we have $(u_1,u_2,v,i)\in \cI_1^0$.
\item[Claim 2.] The element $(e,e,e,0)$ belongs to $\cI_1$ and satisfies assumption $(H)$.
\item[Claim 3.] If one of two adjacent elements of $\cI_1$ satisfies assumption
  $(H)$ then the two ones satisfy assumption
  $(H)$.
\end{enumerate}

These claims are  sufficient. Indeed, for any $a\in\cI_1$, there exists a
sequence $a=a_0,\dots,a_n=(e,e,e,0)$ such that $\Ac_{a_i}$ and $\Ac_{a_{i+1}}$ are
adjacent for any $i$. By Claim~2, $a_n$ satisfies assumption~$(H)$. By an immediate
induction, Claim~3 implies that $a$ satisfies $(H)$. 
Now Claim~1 shows that $a\in \cI_1^0$.


\bigskip \noindent
{\it Proof of Claim 2.} 
It is a direct consequence of Lemma~\ref{lem:face0}.

\bigskip \noindent
{\it Proof of Claim 1$\;\Rightarrow\;$ Claim 3.}
Let $a$ and $a'$ in $\cI_1$ be such that $\Ac_a$ and $\Ac_{a'}$ are 
adjacent along the face $\Ac'$. Assume that $a$ satisfy assumption $(H)$.
Then the interior of $\Ac'$ is contained in the interior of
$\Ac_a\cup\Ac_{a'}$. In particular, 
 $\Ac'$  intersects $(P^{++}_\QQ)^3$.
Let $(\lambda_1,\lambda_2,\bar\mu)\in\Ac\cap P^{++}$. 
Then
$\varphi_a(\lambda_1,\lambda_2,\bar\mu)=\varphi_{a'}(\lambda_1,\lambda_2,\bar\mu)$. 
But, by Claim~1 for $a$,
$(\lambda_1,\lambda_2,\bar\mu)+\varphi_a(\lambda_1,\lambda_2,\bar\mu)\delta$
belongs to $\Gamma(\lg)$.
Then $(\lambda_1,\lambda_2,\bar\mu)+\varphi_{a'}(\lambda_1,\lambda_2,\bar\mu)\delta$
belongs to $\Gamma(\lg)\cap P^{++}_\QQ\cap\Face_{a'}$ and
$a'$ satisfies assumption $(H)$.

\bigskip \noindent
{\it Proof of Claim 1.} 
Let $(\bar u_1,\bar u_2,\bar v,i)\in \cI_1$ satisfying assumption
$(H)$.
Let $(\lambda_1,\lambda_2,\mu)\in \Face_{(\bar u_1,\bar u_2,\bar
  v,i)}$.
Then $(\lambda_1,\lambda_2,\bar\mu)\in\Ac$ and 
\begin{eqnarray}
  \label{eq:mudphi}
  \mu(d)=\varphi_{(\bar u_1,\bar u_2,\bar
  v,i)}(\lambda_1,\lambda_2,\bar \mu).
\end{eqnarray}
By Theorem~\ref{th:egalite}, to prove that $(\lambda_1,\lambda_2,\mu)$ belongs
to $\Gamma(\lg)$, it is sufficient  to prove that
$C^\rss(\Li,L_i)\neq\emptyset$ where $\Li$ is the line bundle on $\dX$
associated to $(\lambda_1,\lambda_2,\mu)$ and 
$C=L_iu_1\inv\underline o^-\times L_iu_2\inv\underline o^-\times
L_iv\inv\underline o$.
By \cite{BK}, this last assertion is equivalent to the fact that
$(\lambda_1,\lambda_2,\mu)$ satisfies a finite family of linear
inequalities. 
More precisely, $C^\rss(\Li,L_i)\neq\emptyset$ if and only if
\begin{enumerate}
\item \label{cond1}$u_1\inv\lambda_1$,$u_2\inv\lambda_2$ and $v\inv\mu$ are
  dominant for $L_i$;
\item \label{cond2} $Z(L_i)^\circ$ acts trivially on $\Li_{|C}$;
\item for any $j\in\{0,\dots,l\}-\{i\}$
  \begin{equation}
    \label{eq:412}
\langle\tilde v\varpi_{\alpha_j^\vee},\bar v\inv\mu\rangle\leq
\langle\tilde u_1\varpi_{\alpha_j^\vee},\bar u_1\inv\mu\rangle
+
\langle\tilde u_2\varpi_{\alpha_j^\vee},\bar u_2\inv\mu\rangle,
  \end{equation}
for any $(\tilde u_1,\tilde u_2,\tilde v)\in W_{L_i}^{P_j^i}$
such that $\epsilon_{\tilde v}(L_i/P_j^i)$ appears with
coefficient one in $\epsilon_{\tilde u_1}(L_i/P_j^i)\bkprod
\epsilon_{\tilde u_2}(L_i/P_j^i)$.
\end{enumerate}
Here $P_j^i$ is the maximal standard parabolic subgroup of $L_i$
associated to $j$. Then $P_j^i=L_i\cap P_j$. Set also $P_{i,j}=P_i\cap
P_j$.
One easily checks that $(\lambda_1,\lambda_2,\mu)$ satisfies
\eqref{cond1} and \eqref{cond2}.
It remains to prove that it satisfies inequalities~\eqref{eq:412}.

Denote by $p$ the projection that maps $(\lambda_1,\lambda_2,\mu)$ on
$(\lambda_1,\lambda_2,\bar\mu)$. 
By \cite{BK}, the image by $p$ of the set of triples satisfying
conditions \eqref{cond1}, \eqref{cond2} and inequalities~\eqref{eq:412} is
a closed convex polyhedral cone $\cone^{L_i}(C)$ of nonempty
interior. 
It remains to prove that $(\lambda_1,\lambda_2,\bar\mu)$ belongs to
$\cone^{L_i}(C)\cap \Ac$.
By assumption $(H)$, the cone $\cone^{L_i}(C)$ intersects the interior
of $\Ac$. 
In particular, $\cone^{L_i}(C)\cap \Ac$ has nonempty interior.
Then, the general theory of convex polyhedrons implies that it is
sufficient to check the conditions~\eqref{eq:412} such that the
associated face of $\cone^{L_i}(C)\cap \Ac$ has codimension one.
Consider such an inequality associated to $(\tilde u_1,\tilde u_2,\tilde
v,j)$ and
\begin{center}
\begin{tikzpicture}
  \matrix (m) [matrix of math nodes,row sep=2.5em,column sep=1.2em,minimum width=2em]
  {
     L_i/(P_j^i)&&G/(P_{ij})\\
&G/P_i&&G/P_j\\};
  \path[-stealth]
    (m-1-1) edge  (m-1-3)
    (m-1-3) edge  (m-2-2)
    (m-1-3) edge  (m-2-4);
\end{tikzpicture}  
\end{center} 
By Lemma~\ref{lem:WQP},  $(u_1=\bar u_1 \tilde u_1,u_2=\bar u_2\tilde
u_2,v=\bar v\tilde v)\in
(W^{P_{i,j}})^3$.
By Proposition~\ref{prop:multiplicative}, the
$\epsilon_v$ appears with multiplicity one in
$\epsilon_{u_1}.\epsilon_{u_2}$,  in $\Ho^*(G/P_{i,j},\ZZ)$.
Set $\tau=\varpi_{\alpha_i^\vee}+\varpi_{\alpha_j^\vee}$. 
Then
\begin{eqnarray}
  \label{eq:tau}
\langle v\tau,\mu'\rangle\leq
\langle u_1\tau,\lambda_1'\rangle+
\langle u_2\tau,\lambda_2'\rangle,
\end{eqnarray}
for any $(\lambda_1',\lambda_2',\mu')\in \Gamma(\lg)$.
Let  $(\lambda_1'',\lambda_2'',\mu'')$ in the interior of $\Ac$, in
$\cone^{L_i}(C)$ and such
that inequality~\eqref{eq:412} is an equality.
Then  inequality~\eqref{eq:tau} is an equality for 
$(\lambda_1'',\lambda_2'',\mu''+\varphi_{(\bar u_1,\bar u_2,\bar
  v,i)}\delta)$.
Now, Theorem~\ref{th:essineqBKBprod} implies that $\epsilon_v$ appears with multiplicity one in
$\epsilon_{u_1}\bkprod \epsilon_{u_2}$.

Consider now the three
elements $\bar u_1',\,\bar u'_2$ and $\bar v'$ in $W^{P_{j}}$ such
that $\bar u_1' W_{P_j}=u_1 W_{P_j}$,$\bar u'_2 W_{P_j} =u_2 W_{P_j}$
and $\bar v' W_{P_j} =v W_{P_j}$.
Lemma~\ref{lem:bkhyp} and
Proposition~\ref{prop:multiplicative} imply that, in $H^*(G/P_j,\ZZ)$, $\epsilon_{\bar v'}$ appears with multiplicity one in
$\epsilon_{\bar u_1'}\bkprod \epsilon_{\bar u'_2}$.
Since $(\lambda_1,\lambda_2,\mu)\in\Cun$, it satisfies
\begin{equation}
  \label{eq:410}
\langle\mu,v\varpi_{\alpha_j^\vee}\rangle\leq
\langle\lambda_1,u_1\varpi_{\alpha_j^\vee}\rangle+
\langle\lambda_2,u_2\varpi_{\alpha_j^\vee}\rangle.
\end{equation}
But, modulo equality~\eqref{eq:mudphi}, the inequality~\eqref{eq:412}
to prove is equivalent to 
inequality~\eqref{eq:410}.

We conclude that  $(\bar
  u_1,\bar u_2,\bar v,i)$ belongs to $\cI^0_1$.
\end{proof}

\section{Saturation factors}
\label{sec:saturation}

In this section, we prove Theorems~\ref{th:saturation} and
\ref{th:saturation2} of the introduction. Let us first check the
computation of the constants $k_s$. 

In the finite dimensional case,  known saturation factors are collected
in the following tabular. These results was obtained in
\cite{KT:saturation} for the
type $A$,  in \cite{KapMill:tssemigroup} for $B_2$ and $G_2$,
\cite{BK:satoddorthsymp} for types $B_\bl, C_\bl$, \cite{KKM} for 
$D4$ and in \cite{KM} in the remaning cases.

$$
\begin{array}{|c|c|c|c|c|c|}
\hline
  {\rm Type}&A_\bl&B_\bl(\bl\geq 2)&C_\bl(\bl\geq 3)&D_4&D_\bl(\bl\geq 5)\\[1.1ex]
\hline
{\rm Saturation\ factor}&1&2&2&1&4
\\
\hline\hline
 {\rm Type}&E_6&E_7&E_8&F_4&G_2 \\[1.1ex]
\hline
{\rm Saturation\ factor}&36&144&3\,600&144&2,3\\
\hline
\end{array}
$$

Using these datas,
one can easily check the computations of the $k_s$ given in the
introduction reading the  Dynkin diagrams.

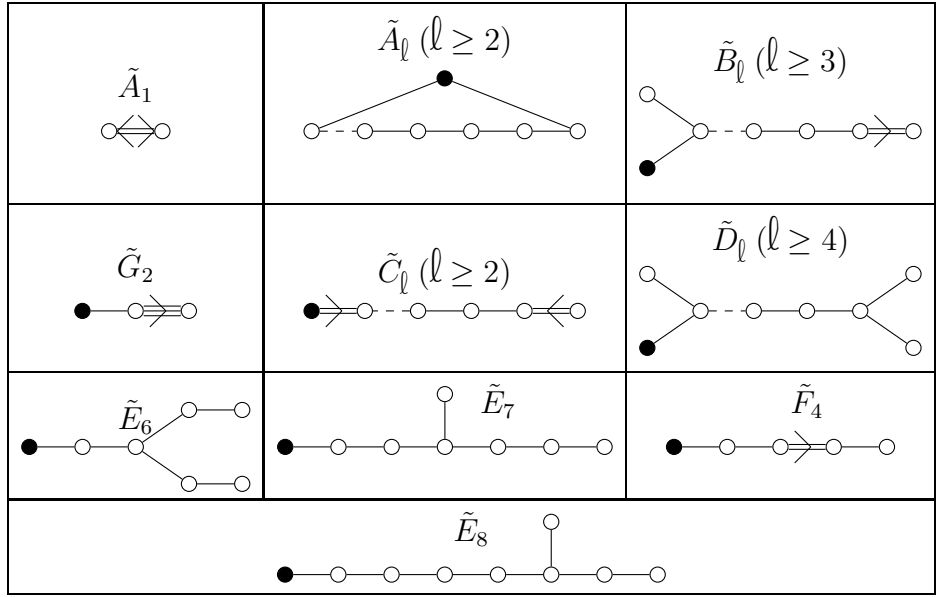
\begin{center}
  \begin{tabular}{|c|c|c|}
\hline
  \begin{tikzpicture}[scale=.7,baseline=(2.center)]
    \draw (0.5,1.4) node[anchor=north]  {$\tilde{A}_1$};

    \node[dnode] (1) at (0,0) {};
    \node[dnode] (2) at (1,0) {};

    \path (2) edge[infedge] (1)
          ;
\end{tikzpicture}&
\begin{tikzpicture}[scale=.7,baseline=(2.center)]
\draw (2.5,2.4) node[anchor=north]  {$\tilde{A}_\bl\;(\bl\geq 2)$};
    \node[dnode] (1) at (0,0) {};
    \node[dnode] (2) at (1,0) {};
    \node[dnode] (3) at (2,0) {};
    \node[dnode] (4) at (3,0) {};
    \node[dnode] (5) at (4,0) {};
    \node[dnode] (6) at (5,0) {};
    \node[dnode,fill=black] (0) at (2.5,1) {};
    \path (1) edge[sedge,dashed] (2) edge[sedge] (0);
    \path (3) edge[sedge] (2) edge[sedge] (4);
     \path (6) edge[sedge] (5) edge[sedge] (0);
\path (4) edge[sedge] (5);
  \end{tikzpicture}
  &
  
  \begin{tikzpicture}[scale=.7,baseline=(3.center)]
\draw (2.5,2) node[anchor=north]  {$\tilde{B}_\bl\;(\bl\geq 3)$};
    \node[dnode] (1) at (0,0.7) {};
    \node[dnode] (2) at (1,0) {};
    \node[dnode] (3) at (2,0) {};
    \node[dnode] (4) at (3,0) {};
    \node[dnode] (5) at (4,0) {};
    \node[dnode] (6) at (5,0) {};
    \node[dnode,fill=black] (0) at (0,-0.7) {};
    \path (2) edge[sedge,dashed] (3) edge[sedge] (0) edge[sedge] (1);
    \path (4) edge[sedge] (3) edge[sedge] (5);
     \path (5) edge[dedge] (6);
  \end{tikzpicture}\\[4.5ex]
\hline
\begin{tikzpicture} [scale=.7,baseline=(2.center)]
\draw (1,1.4) node[anchor=north]  {$\tilde{G}_2$};
    \node[dnode,fill=black] (1) at (0,0) {};
    \node[dnode] (2) at (1,0) {};
    \node[dnode] (3) at (2,0) {};
 
     \path (2) edge[tedge] (3);
\path (1) edge[sedge] (2);
  \end{tikzpicture}
   &

\begin{tikzpicture}[scale=.7,baseline=(2.center)]
    \draw (2.5,1.4) node[anchor=north]  {$\tilde{C}_\bl\;(\bl\geq 2)$};
\node[dnode,fill=black] (1) at (0,0) {};
    \node[dnode] (2) at (1,0) {};
    \node[dnode] (3) at (2,0) {};
    \node[dnode] (4) at (3,0) {};
    \node[dnode] (5) at (4,0) {};
    \node[dnode] (6) at (5,0) {};
    \path (1) edge[dedge] (2);
    \path (3) edge[sedge,dashed] (2) edge[sedge] (4);
     \path (6) edge[dedge] (5);
\path (4) edge[sedge] (5);
  \end{tikzpicture}
&
 \begin{tikzpicture}[scale=.7,baseline=(3.center)]
  \draw (2.5,2) node[anchor=north]  {$\tilde{D}_\bl\;(\bl\geq 4)$};
  \node[dnode] (1) at (0,0.7) {};
    \node[dnode] (2) at (1,0) {};
    \node[dnode] (3) at (2,0) {};
    \node[dnode] (4) at (3,0) {};
    \node[dnode] (5) at (4,0) {};
    \node[dnode] (6) at (5,0.7) {};
     \node[dnode] (7) at (5,-0.7) {};
    \node[dnode,fill=black] (0) at (0,-0.7) {};
    \path (2) edge[sedge,dashed] (3) edge[sedge] (0) edge[sedge] (1);
    \path (4) edge[sedge] (3) edge[sedge] (5);
     \path (5) edge[sedge] (6) edge[sedge] (7);
  \end{tikzpicture}
\\[3.6ex]
\hline
\begin{tikzpicture}[scale=.7,baseline=(1.center)]
\draw (2,1.1) node[anchor=north]  {$\tilde{E}_6$};
    \node[dnode,fill=black] (0) at (0,0) {};
    \node[dnode] (1) at (1,0) {};
    \node[dnode] (2) at (2,0) {};
\node[dnode] (3) at (3,-0.7) {};
    \node[dnode] (4) at (3,0.7) {};
    \node[dnode] (5) at (4,-0.7) {};
    \node[dnode] (6) at (4,0.7) {};

    \path (1) edge[sedge] (0) edge[sedge] (2);
    \path (2) edge[sedge] (3) edge[sedge] (4);
     \path (5) edge[sedge] (3);\path (6) edge[sedge] (4);
  \end{tikzpicture}
&
\begin{tikzpicture} [scale=.7,baseline=(4.center)]
\draw (4,1.4) node[anchor=north]  {$\tilde{E}_7$};
    \node[dnode,fill=black] (1) at (0,0) {};
    \node[dnode] (2) at (1,0) {};
    \node[dnode] (3) at (2,0) {};
    \node[dnode] (4) at (3,0) {};
    \node[dnode] (5) at (4,0) {};
    \node[dnode] (6) at (5,0) {};
    \node[dnode] (7) at (6,0) {};
    \node[dnode] (0) at (3,1) {};
    \path (4) edge[sedge] (0) edge[sedge] (5) edge[sedge] (3);
    \path (2) edge[sedge] (1) edge[sedge] (3);
     \path (6) edge[sedge] (5) edge[sedge] (7);
  \end{tikzpicture}
&
\begin{tikzpicture} [scale=.7,baseline=(3.center)]
\draw (2.5,1.4) node[anchor=north]  {$\tilde{F}_4$};
    \node[dnode,fill=black] (1) at (0,0) {};
    \node[dnode] (2) at (1,0) {};
    \node[dnode] (3) at (2,0) {};
    \node[dnode] (4) at (3,0) {};
    \node[dnode] (5) at (4,0) {};
    
    \path (2) edge[sedge] (1) edge[sedge] (3);
   
     \path (3) edge[dedge] (4);
\path (4) edge[sedge] (5);
  \end{tikzpicture}\\[3ex]
\hline

\multicolumn{3}{|c|}{
\begin{tikzpicture} [scale=.7]
\draw (2.5,1.4) node[anchor=north]  {$\tilde{E}_8$};
    \node[dnode,fill=black] (0) at (-1,0) {};
    \node[dnode] (1) at (0,0) {};
    \node[dnode] (2) at (1,0) {};
    \node[dnode] (3) at (2,0) {};
    \node[dnode] (4) at (3,0) {};
    \node[dnode] (5) at (4,0) {};
    \node[dnode] (6) at (5,0) {};
    \node[dnode] (7) at (6,0) {};
    \node[dnode] (8) at (4,1) {};
    \path (5) edge[sedge] (6) edge[sedge] (4) edge[sedge] (8);
    \path (3) edge[sedge] (4) edge[sedge] (2);
     \path (1) edge[sedge] (0) edge[sedge] (2);
     \path (6) edge[sedge] (7);
  \end{tikzpicture}}\\
\hline
\end{tabular}
  \end{center}

  \begin{proof}[Proof of Theorems~\ref{th:saturation} and \ref{th:saturation2}]
    Let $(\lambda_1,\lambda_2,\mu)\in (P_+)^3$ such that
    $\mu-\lambda_1-\lambda_2\in Q$ and there exists $N>0$ such
    that $(N\lambda_1,N\lambda_2,N \mu)\in\Gamma_\NN(\lg)$.
Up to tensoring with $L(\delta)$ one may assume that
$\lambda_1(d)=\lambda_2(d)=0$.
Write $\mu$ as $\bar\mu+n\delta$, with $n\in \ZZ$.

Set $b=\varphi(\lambda_1,\lambda_2,\bar\nu)$. 
By Proposition~\ref{prop:locpol}, there exists $(u_1,u_2,v,i)\in\cI$
such that $b=\varphi_{(u_1,u_2,v,i)}(\lambda_1,\lambda_2,\bar\nu)$.

We claim that $bk_\lg$ is an integer.
The norm on $\dot Q^\vee$ is normalized by $\Vert\dot\alpha^\vee\Vert^2=2$
for a short coroot $\dot\alpha^\vee\in\dot\Phi^\vee$. 
Then, for any $h\in \dot Q^\vee$, we have $\frac{\Vert h\Vert^2}2\in \ZZ$. 
This can be proved easily by a case by case consideration. 
Then, formula~\eqref{eq:3} shows that $b\in\ZZ$ if $i=0$. 
If $i>0$, formula~\eqref{eq:33} shows that $k_\lg b\in\ZZ$.

 Consider $(k_\lg\lambda_1,k_\lg\lambda_2,k_\lg\bar\nu+(k_\lg
 b)\delta)$. 
Observe the $Q=\dot Q+\ZZ\delta$. Then, since $
\nu-\lambda_1-\lambda_2\in Q$,
$k_\lg\bar\nu+(k_\lg
 b)\delta -k_\lg\lambda_1-k_\lg\lambda_2\in Q$. 
For any $\lambda\in \lh^*_\ZZ$ and $w\in W$, $\lambda-w\lambda\in
Q$. Hence, $v\inv k_\lg\bar\nu+(k_\lg
 b)\delta-u_1\inv k_\lg\lambda_1-u_2\inv k_\lg\lambda_2$ belongs to
 $Q$. 
Since $(N\lambda_1,N\lambda_2,N \nu)\in\Gamma_\NN(\lg)$,
Corollary~\ref{cor:G2L} implies that 
$(u_1\inv \lambda_1,u_2\inv \lambda_2,v\inv (\bar\nu+b\delta))$ belongs to
$\Gamma(L_i)$. 
But $k_s$ is a saturation factor for the group $L_i$. Then 
$(u_1\inv k_s k_\lg\lambda_1,u_2\inv k_s k_\lg\lambda_2,v\inv k_s k_\lg(\bar\nu+b\delta))$ belongs to
$\Gamma(L_i)$. 
Corollary~\ref{cor:G2L} implies that $(k_s k_\lg\lambda_1,k_s k_\lg\lambda_2,k_s k_\lg(\bar\nu+b\delta))$ belongs to
$\Gamma_\NN(\lg)$.

Proposition~\ref{prop:ineg} implies that $n\leq b$.
Then $k_\lg(b-n)\in \ZZ_{\geq 0}$. 

If $b=n$, we already proved that 
$(k_s k_\lg\lambda_1,k_s k_\lg\lambda_2,k_s k_\lg\nu)$ belongs to
$\Gamma_\NN(\lg)$.
Theorem~\ref{th:saturation} is proved in this case.

Moreover, Proposition~\ref{prop:ineg} implies that the integer $b_0$
of 
Lemma~\ref{lem:Gammaepigraph} for $(k_s k_\lg\lambda_1,k_s
k_\lg\lambda_2,k_s k_\lg\bar\nu)$ is equal to $k_s k_\lg b$. 
Then, Lemma~\ref{lem:Gammaepigraph} implies that 
$(k_s k_\lg\lambda_1,k_s
k_\lg\lambda_2,k_s k_\lg\bar\nu+m\delta)\in\Gamma_\NN(\lg)$, for any
\begin{equation}
  \label{eq:110}
  m\leq k_s k_\lg b-2.
\end{equation}
Assume that $k_\lg(b-n)\in \ZZ_{>0}$.
If $k_s>1$, $m=k_sk_\lg n$ satisfies condition~\eqref{eq:110}. 
Similarly, for any $d>1$, $k_sk_\lg n-d$ satisfies
condition~\eqref{eq:110}.
The theorems follow in these cases.

Assume now that  $k_s=1$ and fix $d>1$.
We may assume that $n\neq b$.
 Then, the integer $b_0$ of 
Lemma~\ref{lem:Gammaepigraph} for $(d k_\lg\lambda_1,d
k_\lg\lambda_2,d k_\lg\bar\nu)$ is equal to $dk_\lg b$.
Since  $m=dk_\lg n=d(k_\lg n-k_\lg b)+dk_\lg n$ satisfies 
$m\leq d k_\lg b-2$, Theorem~\ref{th:saturation} also holds in this case.
  \end{proof}

\section{Some technical lemmas}
\label{sec:tech}

\renewcommand{\thelemma}{\thesection.\arabic{lemma}}
\renewcommand{\theprop}{\thesection.\arabic{prop}}
\setcounter{prop}{0}
\setcounter{lemma}{0}

In this section we collect some technical results on Birkhoff and
Bruhat decompositions, on Geometric Invariant Theory\dots

\subsection{Bruhat and Birkhoff decompositions}

In this subsection, $G$ is the minimal Kac-Moody group associated to
any symmetrizable GCM. 
Fix $T$, $W$, $B$ and $B^-$ as usually.  
Let $P\supset B$ be a standard parabolic subgroup with standard Levi
subgroup $L$. 
Fix a one parameter subgroup $\tau$ of $T$ such that for all
$\beta\in\Phi$, 
$\beta\in\Phi(P)$ if and only if $\langle\beta,\tau\rangle\geq 0$.

\begin{lemma}
 \label{lem:flow1}
Let $u\in W$ and $v\in W^P$ such that  $u\neq v$. 
Let $x\in \cX^{u\inv}_\GBm$.

Then   $\lim_{t\to 0}\tau(t)x$ does not belong to $v\inv B\underline{o}^-$.
\end{lemma}

\begin{proof}
  Recall that $(G/B^-)^\tau$ can be decomposed in the two following
  ways:
$$
(G/B^-)^\tau=\sqcup_{w\in W^P}Lw\inv \underline{o}^-=\sqcup_{w\in W}(B\cap L)w\underline{o}^-.
$$
Moreover $\{w\underline{o}^-\,:\,w\in W\}$ are exactly the $T$-fixed points
in $G/B^-$. 
Hence, if $Y\subset G/B^-$ is $(B\cap L)-$stable then
$Y^\tau=\sqcup_{x\in Y^T}(B\cap L)x$. 
For $Y=Bu\inv \underline{o}^-$ we get 
$$
(Bu\inv \underline{o}^-)^\tau=(B\cap L)u\inv \underline{o}^-.
$$
If $v\in W^P$ then $(v\inv B v)\cap L=B\cap L$. Hence, for $Y=v\inv B\underline{o}^-$,
we get
$$
(v\inv B\underline{o}^-)^\tau=(B\cap L)v\inv \underline{o}^-.
$$
Since $\lim_{t\to 0}\tau(t)x$ belongs to $(Bu\inv \underline{o}^-)^\tau$, we
deduce that it does not belong to $v\inv B\underline{o}^-$.
\end{proof}

\begin{lemma}
  \label{lem:flow2}
Let $u,v\in W$ such that $l(v)=l(u)+1$.

\begin{enumerate}

\item Let $x_1,x_2\in \cX^u_\GBm$ such that
  $\lim_{t\to\infty}\tau(t)x_1=\lim_{t\to\infty}\tau(t)x_2$ belongs to
  $\cX^v_\GBm$. 

Then $\tau(\CC^*)x_1=\tau(\CC^*)x_2$.
\item Let $x_1,x_2\in \cX_v^\GB$ such that
  $\lim_{t\to\infty}\tau(t)x_1=\lim_{t\to\infty}\tau(t)x_2$ belongs to
  $\cX_u^\GB$. 

Then $\tau(\CC^*)x_1=\tau(\CC^*)x_2$.
\end{enumerate}
\end{lemma}

\begin{proof}
Let us prove the first assertion.
Set $y=\lim_{t\to\infty}\tau(t)x_1$.
Then $y\in (\cX^v_\GBm)^\tau =(B\cap L)v\underline{o}^-$.
Fix $l\in B\cap L$ such that $y=l v\underline{o}^-$.
Moreover, $Py=Px_1=Px_2$.
Note that for any $g^u\in P^{u,-}$, $l'\in L$ and $w\in W$, we have
$$\lim_{t\to\infty}\tau(t)g^ul'w\underline{o}^-=l'w\underline{o}^-.$$
 One deduces that there exist $g^u_i\in
P^{u,-}$ such that $x_i=g^u_iy$, for $i=1,2$.
Then  $l\inv x_i=(l\inv g^u_il)v \underline{o}^-$ belongs to
$\cX_v^\GBm$. 
Since $l\in B$, $l\inv x_i$ also belongs
to $\cX^u_\GBm$. 
Finally, $l\inv x_i$ belongs to  $\cX^u_\GBm\cap \cX_v^\GBm$.

But $l(v)=l(u)+1$ and  $\cX^u_\GBm\cap \cX_v^\GBm$ is isomorphic to
$\CC^*$. Since $l\inv x_1$ and $l\inv x_2$ are not fixed by
$\tau(\CC^*)$ they belong to the same $\tau$-orbit. 
Since the actions of $\tau$ and $L$ commute, one deduces that
$\tau(\CC^*)x_1=\tau(\CC^*)x_2$. 

\bigskip
The second assertion works similarly. Up to translating by an element
of $B\cap L$, one may assume that 
$\lim_{t\to\infty}\tau(t)x_1=\lim_{t\to\infty}\tau(t)x_2=u\underline
o$. Then $x_1$ and $x_2$ belong to 
$\cX^u_\GB\cap\cX_v^\GB$ that is isomorphic to $\CC^*$.
\end{proof}

\begin{lemma}
\label{lem:Cplushom}

Let $Q_1,Q_2$ be  two parabolic subgroups of $G$ containing $B^-$. Consider
$\dX=G/Q_1\times G/Q_2\times G/P$ with base point $(\underline
o_1,\underline o_2,\underline o)$. Let $P$ be a standard parabolic subgroup of finite
type and $L$ denote its  standard Levi subgroup. Fix $l\in L$. 
Let $u_1,u_2$, and $v$ in $W^P$. Set $\dx_0=(lu_1\inv \underline{o}_1,u_2\inv
\underline{o}_2,v\inv \underline{o})\in \dX$ and $\Orb=G.\dx_0$.

Then
$$
\{x\in\Orb\,:\,\lim_{t\to 0}\tau(t)x\in L.\dx_0\}=
P.\dx_0.
$$
\end{lemma}

\begin{proof}
  Consider first the anologous situation in $G/Q_2\times G/B$, with
  its two projections $p_1$ and $p_2$ on $G/Q_2$ and $G/B$. 
Set $\dx_1=(u_2\inv \underline{o}_2,v\inv \underline{o})$, $\Orb_1=G.\dx_1$ and
$\Orb_1^0=L.\dx_1$.
Set also 
$\Orb_1^+=\{x\in\Orb_1\,:\,\lim_{t\to 0}\tau(t)x\in L.\dx_1\}$.
Then $p_1(\Orb_1)=G/B^-$ and $p_1(\Orb_1^0)=Lu_2\inv \underline{o}_2$. Moreover, 
$$
\{x\in G/Q_2\,:\,\lim_{t\to 0}\tau(t)x\in L.u_2\inv \underline{o}\}=
P.u_2\inv \underline{o}.
$$
Since $\Orb_1^+$ is stable by $P$, it follows that 
$$
\Orb_1^+=P.\cI\mbox{ where } \cI=
(\{u_2\inv \underline{o}_2\}\times G/B)\cap \Orb_1^+.
$$
Set $\dx_2=v\inv \underline{o}$. Then, 
$p_2(\cI)$ is the set of points $x\in (u_2\inv Q_2u_2).\dx_2$ such that
$\lim_{t\to 0}\tau(t)x\in (L\cap u_2^{-1}Q_2u_2)\dx_2$.
In particular $p_2(\cI)$ is contained in $P\dx_2$. In particular
the weights of $\tau$ on $T_{\dx_2}p_2(\cI)$ are nonnegative. On the other
hand they are contained in $T_{\dx_2}(u_2^{-1}Q_2u_2)\dx_2$. 
It follows
that $T_{\dx_2}p_2(\cI)$ is contained in $T_{\dx_2}(P\cap
u_2^{-1}Q_2u_2)\dx_2$. 
Note that, since $P$ has finite type, $P\cap u_2\inv Q_2u_2$ is finite
dimensional. Moreover, the dimension of $\cI$ (at $\dx_2$) is at most equal to $\dim((P\cap
u_2^{-1}Q_2u_2)\dx_2)$. It follows that $(P\cap u_2^{-1}Q_2u_2)\dx_2$
is open in $p_2(\cI)$.
Since $(P\cap u_2^{-1}Q_2u_2)\dx_2$ contains $(L\cap
u_2^{-1}Q_2u_2)\dx_2$, we deduce that 
$$
p_2(\cI)=(P\cap u_2^{-1}Q_2u_2)\dx_2,
$$
and
\begin{equation}
  \label{eq:11}
  \Orb_1^+=P.\dx_1.
\end{equation}

Consider now 
$$
\pi_1\,:\,\dX\longto G/Q_2\times G/B,\,(x_1,x_2,x_3)\longmapsto
(x_2,x_3).
$$
Set $\Orb^+=\{x\in\Orb\,:\,\lim_{t\to 0}\tau(t)x\in L.\dx_0\}$.
Equality~\eqref{eq:11} and the fact that $\Orb^+$ is $P$-stable imply that 
\begin{equation}
  \label{eq:102}
\Orb^+=P\bigg(
(G/Q_1\times \{\dx_1\})\cap \Orb^+
\bigg).
\end{equation}
Note that 
$$
\begin{array}{ll}
  (G/Q_1\times \{\dx_1\})\cap \Orb=(u_2\inv Q_2u_2\cap v\inv
  Bv).\dx_0&\mbox{and}\\
(G/Q_1\times \{\dx_1\})\cap \Orb^0=(u_2\inv Q_2u_2\cap v\inv
  Bv\cap L).\dx_0
\end{array}
$$
Then, since $u_2\inv Q_2u_2\cap v\inv Bv$ is finite dimensional,
\cite[Lemma~12]{GITEigen} shows that 
$$
(G/Q_1\times \{\dx_1\})\cap \Orb^+=(P\cap u_2\inv Q_2u_2\cap v\inv
  Bv).\dx_0.
$$
With equality~\eqref{eq:102} this ends the proof of the lemma.
\end{proof}

\subsection{Affine root systems}

In this subsection, we consider an untwisted affine root system and
use the notations of Section~\ref{sec:defaff}. 
Recall in particular, that $\dot\lh^*_\RR$ is endowed with $\dot
W$-Euclidean norm $\Vert\,\cdot\,\Vert$ such that
$\Vert\theta\Vert^2=2$. 

\begin{lemma}
  \label{lem:lvsnorm}
Consider an affine Weyl group $W=\dot Q^\vee.\dot W$. Set $N=\sharp \dot \Phi^+$.

There exists a positive real constant $K$  such that
for any $h\in \dot Q^\vee$ and $\dot w\in\dot W$, we have
$$
K \Vert h\Vert-N\leq l(h\dot w)\leq N+\sqrt 2 N \Vert h\Vert.
$$
\end{lemma}

\begin{proof}
  Set $w=h\dot w$. The length of $w$ is the
cardinality of $w\inv\Phi^+\cap\Phi^-$. One can deduce (see
e.g. \cite{IM}) that:
\begin{equation}
  \label{eq:416}
l(h\dot w)=\sum_{\dot \alpha\in\dot \Phi^+,\,\dot w\inv \dot \alpha\in\dot \Phi^+}|\langle
h,\dot \alpha\rangle|+
\sum_{\dot \alpha\in\dot \Phi^+,\,\dot w\inv \dot \alpha\in\dot \Phi^-}|\langle h,\dot \alpha\rangle-1|.
\end{equation}
The inequality on the right just follows from 
$$
\begin{array}{l}
|\langle h,\dot \alpha\rangle-1|\leq |\langle h,\dot \alpha\rangle|+1\\
  |\langle
h,\dot \alpha\rangle|\leq \Vert h\Vert \Vert \dot\alpha\Vert \leq \sqrt 2\Vert
  h\Vert.
\end{array}
$$

Moreover,
$$
\begin{array}{ll}
  l(h\dot w)&\geq l(h)-l(\dot w)\\
&\geq\sum_{\dot \alpha\in\dot\Phi^+}|\langle
h,\dot \alpha\rangle|-N.
\end{array}
$$
Since $\dot\Phi^+$ spans $\dot\lh^*_\RR$, the map $h\mapsto \sum_{\dot \alpha\in\dot \Phi^+}|\langle
h,\dot \alpha\rangle|$ is a norm on the real vector space $\dot\lh^*_\RR$. This norm
is equivalent to $\Vert\cdot\Vert$, and there exists $K$ such that $K \Vert h\Vert\leq \sum_{\dot \alpha\in\dot \Phi^+}|\langle
h,\dot \alpha\rangle|$. The lemma follows.
\end{proof}

\subsection{Jacobson-Morozov's theorem}

Let $\lg$ be an untwisted affine Kac-Moody Lie algebra and $\lp$ be a
standard parabolic subalgebra. 
Let $G$ be the minimal Kac-Moody group associated to $\lg$ and $P$ be
the parabolic subgroup corresponding to $\lp$.
Fix $\tau$ a one parameter subgroup of $T$ in
$\oplus_{\alpha_j\not\in\Delta(P)}\ZZ_{>0}\varpi_{\alpha_j^\vee}$.
Consider the action of $\tau$ on $\lg$ and the corresponding weight
space decompositions
$$
\lg=\oplus_{n\in\ZZ}\lg_n\qquad\lp=\oplus_{n\in\ZZ_{\geq 0}}\lg_n.
$$

In $\sl_2(\CC)$, we denote by $(E,H,F)$ the standard triple
$$
E=
\begin{pmatrix}
  0&1\\0&0
\end{pmatrix}
\qquad
H=
\begin{pmatrix}
  1&0\\0&-1
\end{pmatrix}
\qquad
F=
\begin{pmatrix}
  0&0\\1&0
\end{pmatrix}
$$
satisfying
$$
[E,F]=H\qquad [H,E]=2E\qquad [H,F]=-2F.
$$

\begin{prop}
  \label{prop:JM}
Fix $n\in\ZZ_{>0}$. Let $w\in W$.
Let $\xi$ be a nonzero vector in $\lg_n\cap w\lu^-w\inv$.

Then there exists a morphism $\phi\,:\,\SL_2(\CC)\longto G$ of
ind-groups such that $T_e\phi(E)=\xi$. 
\end{prop}

\begin{proof}
  Observe that $\lg_n$ is contained in $\lu$. Then $\xi\in\lu\cap
  w\lu^-w\inv$ and by \cite[Theorem~10.2.5]{Kumar:KacMoody},
  $\ad\xi\in\End(\lg)$ is locally nilpotent. 

Set $\KK=\CC((t))=\CC[t\inv][[t]]$ and
$\Rr=\CC[t,t\inv]\subset\KK$. Consider the Lie algebras $\dlg\otimes
\Rr$ and $\dlg\otimes\KK$. Recall that $\CC d\oplus\dlg\otimes
\Rr$ is a semi-direct product and that 
\begin{center}
\begin{tikzpicture}
  \matrix (m) [matrix of math nodes,row sep=3em,column sep=3em,minimum width=2em]
  {
     0&\CC c&\lg& \CC d\oplus\dlg\otimes
\Rr&0\\
   };
\draw[->] (m-1-4-|m-1-3.east) -- (m-1-4-|m-1-4.west);
\draw[->] (m-1-4-|m-1-1.east) -- (m-1-4-|m-1-2.west);
\draw[->] (m-1-4-|m-1-2.east) -- (m-1-4-|m-1-3.west);
\draw[->] (m-1-4-|m-1-4.east) -- (m-1-4-|m-1-5.west);
\end{tikzpicture}  
\end{center}
is a central extension. Consider also the canonical $\CC$-linear
embedding $\iota\,:\,\dlg\otimes\Rr\longto\lg$. It is not an
homomorphism of Lie
algebras.

Note that the one parameter subgroup $\tau$ is equal to $\dot\tau+md$
for some one parameter subgroup $\dot\tau$ of $\dot T$ and some
positive integer $m$. 
Then, $\tau$ acts on $\dlg\otimes \Rr$ and
$\dlg\otimes \KK$ by $\CC$-linear automorphisms.
 Consider the decomposition 
$$
\dlg\otimes \Rr=\oplus_{k\in\ZZ}( \dlg\otimes\Rr)_k
$$
in $\tau$-eigenspaces. Since each $( \dlg\otimes\Rr)_k$ is finite
dimensional and $m$ is positive,
we have
\begin{equation}
  \label{eq:decgKK}
\dlg\otimes \KK=\oplus_{k\in\ZZ_{< 0}}(
\dlg\otimes\Rr)_k\oplus\prod_{k\in\ZZ_{\geq 0}} (
\dlg\otimes\Rr)_k.
 \end{equation}

Observe that, for any nonzero $k\in\ZZ$, $\lg_k=(
\dlg\otimes\Rr)_k$. In particular, $\xi$ belongs to $\dlg\otimes\Rr$. We
denote by $\bar\xi$ (resp. $\tilde\xi$) the element $\xi$ considered
as an element of the Lie algebra $\dlg\otimes\Rr$
(resp. $\dlg\otimes\KK$). 

Since $\ad\xi$ is locally nilpotent and $\ad\tilde\xi$ is
$\KK$-linear, $\ad\tilde\xi$ is nilpotent. 
Applying Jacobson-Morozov's theorem (see
e.g. \cite[VIII--\S 11~Proposition~2]{Bourb:Lie79}) to
the Lie algebra $\dlg\otimes\KK$ over the field $\KK$ of
characteristic zero, we get an $\sl_2$-triple $(X,H,Y)$ in
$\dlg\otimes\KK$ such that $X=\tilde\xi$.

Write $Y=\sum _{k\in\ZZ} Y_k$ according to the
decomposition~\eqref{eq:decgKK}. Since the Lie bracket is graded, we have
in $\dlg\otimes\KK$
$$
[X,[X,Y_{-n}]]=-2X.
$$
Set $H_0=[X,Y_{-n}]$ and $\lln=\Ker(\ad\tilde\xi)$.
Since $X$ is homogeneous, $\lln$ decomposes 
as $\oplus_{k\in\ZZ_{< 0}}\lln_k\oplus\prod_{k\in\ZZ_{\geq 0}}
\lln_k$, where $\lln_k=\lln\cap(\dot\lg\otimes\KK)_k$.
Note that $[X,Y_{-n}]+2Y_{-n}$ belongs to $\lln_{-n}$.
By \cite[Corollary~3.4]{Kostant:sl2triple}, $\ad H_0+2\Id_{\dlg\otimes\KK}$ is
injective and stabilizes each $\lln_n$. Moreover, $\lln_{-n}$ is $(\ad
H_0-2\Id_{\dlg\otimes\KK})$-stable and finite dimensional as a
complex vector space. Then there exists $Y'\in\lln_{-n}$ such that 
$$
[X,Y_{-n}]+2Y_{-n}=[X,Y'_{-n}]+2Y'_{-n}.
$$
Then, $(X,H_0,Y_{-n}-Y_{-n}')$ is an $\sl_2$-triple
contained in $(\dlg\otimes\KK)_{n}\times(\dlg\otimes\KK)_0\times (\dlg\otimes\KK)_{-n}$.
In particular, this $\sl_2$-triple is contained in
$\dlg\otimes\Rr$. Hence,  we get an $\Rr$-linear Lie algebra homomorphism
$$
\phi\,:\,\sl_2(\Rr)\longto \dlg\otimes\Rr
$$
such that $\phi(E)=\xi$. Since $\SL_2$ is simply connected and $\Rr$
contains $\QQ$, \cite[Expos\'e~XXIV, Proposition~7.3.1]{SGA3} implies
that there exists a morphism
$$
\Phi\,:\,\SL_2\longto \dot G
$$
of $\Rr$-group schemes with $\phi$ as differential map at the
identity. In particular, we get a morphsm of ind-groups
$$
\bar\Phi\,:\,\SL_2(\CC)\longto \dot G\otimes\Rr
$$
such that $T_e\bar\Phi(E)=\xi$. 

Consider now the semidirect product $\CC^*\ltimes \dot G\otimes\Rr$
associated to the derivation $d$, and the central extension

\begin{center}
\begin{tikzpicture}
  \matrix (m) [matrix of math nodes,row sep=3em,column sep=3em,minimum width=2em]
  {
     \{1\}&\CC^*&G& \CC^*\ltimes \dot G(\Rr)&\{1\}.\\
   };
  \path[-stealth]
    (m-1-1) edge[->] (m-1-2)
    (m-1-2)  edge[->]  (m-1-3)
     (m-1-3)  edge[->]  node [above]{$\pi$}(m-1-4)
 (m-1-4)  edge[->]  (m-1-5);
\end{tikzpicture}  
\end{center}
Then $\pi\inv(\bar\Phi(\SL_2(\CC)))$ is a central extension of
$(P)\SL_2(\CC)$. Hence, it is isomorphic to either $\CC^*\times
(P)\SL_2(\CC)$ or  $\GL_2(\CC)$. 
In each case, $\bar\Phi$ can be lift to a morphism to
$\pi\inv(\bar\Phi(\SL_2(\CC)))$. This concludes the proof of the
proposition.
\end{proof}

\subsection{Geometric Invariant Theory}

For a given $\CC^*$-variety $X$ and a given integer $k$, we denote by
$\CC[X]^{(k)}$ the set of regular functions $f$ on $X$ such that
$(t.f)(x)=f(t\inv x)=t^kf(x)$, for any $t\in\CC^*$ and $x\in X$.

\begin{lemma}\label{lem:GIT}
  Let $X$ be a normal affine $\CC^*$-variety. 
Let $D$ be a $\CC^*$-stable irreducible divisor and set $U=X-D$.
We assume that
\begin{enumerate}
\item \label{hyp1}
$\forall x\in U\quad\lim_{t\to 0} tx$ does not exist in $X$.
\item \label{hyp2}$\forall x_1,x_2\in U\qquad\lim_{t\to \infty} tx_1=\lim_{t\to
    \infty} tx_2\in D\,\Longrightarrow\,
\CC^*x_1=\CC^*x_2$.
\item \label{hyp3}$\forall x\in D\quad\lim_{t\to 0} tx$ does exist in $D$.
\item \label{hyp4}$\forall y\in D^{\CC^*}\qquad\exists x\in U\qquad \lim_{t\to \infty} tx=y$.
\end{enumerate}

Then, for any nonnegative integer $k$, the restriction map induces an
isomorphism $\CC[X]^{(k)}\simeq \CC[U]^{(k)}$.
\end{lemma}

\begin{NB}
  In assumption $(ii)$, $\lim_{t\to \infty} tx_1=\lim_{t\to
    \infty} tx_2\in D$ means that the limits exist and belong to $D$.
\end{NB}
\begin{proof}
  Set $\tilde X=X\times \CC$, $\tilde U=U\times \CC$ and $\tilde
  D=D\times \CC$. We define an action of $\CC^*$ on $\tilde X$ by
  $t.(x,z)=(t.x,tz)$ for any $t\in\CC^*$, $x\in X$ and $z\in\CC$.
Observe that $\CC[\tilde X]^{\CC^*}=\oplus_{k\in \NN} \CC[X]^{(k)} z^k$ and
$\CC[\tilde U]^{\CC^*}=\oplus_{k\in \NN} \CC[U]^{(k)} z^k$.  
Then, it is sufficient to prove that $\CC[\tilde X]^{\CC^*}=\CC[\tilde
U]^{\CC^*}$.

But, one can easily check that $\tilde X$ satisfy all the assumptions
of the lemma. As a consequence, it is sufficient to prove the lemma
for $k=0$.

Consider the commutative diagram

\begin{center}
\begin{tikzpicture}
  \matrix (m) [matrix of math nodes,row sep=3em,column sep=4em,minimum width=2em]
  {
     U&X \\
     U\quot\CC^* & X\quot\CC^*,  \\};
  \path[-stealth]
    (m-1-1) edge node [left] {$\pi_U$} (m-2-1)
            edge  (m-1-2)
    (m-2-1) edge node [above] {$\theta$}(m-2-2)
    (m-1-2) edge node [right] {$\pi_X$} (m-2-2);
\end{tikzpicture}  
\end{center}
where $\quot\CC^*$ denotes the GIT-quotient.
It remains to prove that $\theta$ is an isomorphism. 

We first prove the surjectivity of $\theta$.
Let $\xi\in X\quot\CC^*$ and $\Orb\subset X$ be the unique closed $\CC^*$-orbit in
$\pi_X\inv(\xi)$. It $\Orb\subset U$, it is clear that
$\theta(\pi_U(\Orb))=\xi$. 
Otherwise $\Orb\subset D$.
The orbit $\Orb$ being closed, assumption~\eqref{hyp3} implies that
$\Orb$ is a fixed point.
 By assumption~\eqref{hyp4}, there exists
$\Orb'\subset U$ such that $\overline{\Orb'}\supset\Orb$. Then 
$\theta(\pi_U(\Orb'))=\xi$. We conclude that $\theta$ is surjective.

Let us prove now that $\theta$ is injective. Assume that
$\xi_1\neq\xi_2\in U\quot\CC^*$ satisfy
$\theta(\xi_1)=\theta(\xi_2)=:\xi$. 
Let $\Orb_1,\Orb_2\subset U$ and $\Orb\subset X$ be the closed $\CC^*$-orbit in
$\pi_U\inv(\xi_1)$, $\pi_U\inv(\xi_2)$ and $\pi_X\inv(\xi)$
respectively.
Since $\pi_X(\Orb_1)=\pi_X(\Orb_2)=\xi$, we have $\Orb\subset
\overline{\Orb_1}\cap \overline{\Orb_2}$.
In particular, $\Orb$ is a $\CC^*$-fixed point and $\Orb_1$ and
$\Orb_2$ are one dimensional. 
Pick $x_1\in\Orb_1$,  $x_2\in\Orb_2$ and $y\in\Orb$.
By assumption~\eqref{hyp1}, the limit $\lim_{t\to 0}t.x_1$ does not
exist. But $y\in \overline{\Orb_1}-\Orb_1$, so
$\lim_{t\to \infty}t.x_1=y$. Similarly $\lim_{t\to \infty}t.x_2=y$. Now,
assumption~\eqref{hyp2}, implies that $\Orb_1=\Orb_2$. Hence $\theta$ is
injective.

Since we work over complex numbers, the fact that $\theta$ is bijective
implies that it is birational. 
By assumption $X$ is normal. Thus $X\quot \CC^*$ is normal.
Then Zariski's main theorem (see
e.g. \cite[Theorem~A.11]{Kumar:KacMoody}) implies
that   $\theta$ is an isomorphism.
\end{proof}

\bibliographystyle{amsalpha}
\bibliography{tensorKacMoody}

\newcommand{\etalchar}[1]{$^{#1}$}
\providecommand{\bysame}{\leavevmode\hbox to3em{\hrulefill}\thinspace}
\providecommand{\MR}{\relax\ifhmode\unskip\space\fi MR }
\providecommand{\MRhref}[2]{%
  \href{http://www.ams.org/mathscinet-getitem?mr=#1}{#2}
}
\providecommand{\href}[2]{#2}
\begin{thebibliography}{{Kum}12}

\bibitem[ABD{\etalchar{+}}66]{SGA3}
Michael Artin, Jean-Etienne Bertin, Michel Demazure, Alexander Grothendieck,
  Pierre Gabriel, Michel Raynaud, and Jean-Pierre Serre, \emph{Sch\'emas en
  groupes}, S\'eminaire de G\'eom\'etrie Alg\'ebrique de l'Institut des Hautes
  \'Etudes Scientifiques, Institut des Hautes \'Etudes Scientifiques, Paris,
  1963/1966.

\bibitem[BB05]{BjoBre:Coxgrp}
Anders Bj{{\"o}}rner and Francesco Brenti, \emph{Combinatorics of {C}oxeter
  groups}, Graduate Texts in Mathematics, vol. 231, Springer, New York, 2005.

\bibitem[BK06]{BK}
Prakash Belkale and Shrawan Kumar, \emph{Eigenvalue problem and a new product
  in cohomology of flag varieties}, Invent. Math. \textbf{166} (2006), no.~1,
  185--228.

\bibitem[BK10]{BK:satoddorthsymp}
\bysame, \emph{Eigencone, saturation and {H}orn problems for symplectic and odd
  orthogonal groups}, J. Algebraic Geom. \textbf{19} (2010), no.~2, 199--242.

\bibitem[BK14]{BrownKumar}
Merrick Brown and Shrawan Kumar, \emph{A study of saturated tensor cone for
  symmetrizable {K}ac-{M}oody algebras}, Math. Ann. \textbf{360} (2014),
  no.~3-4, 901--936.

\bibitem[Bou05]{Bourb:Lie79}
Nicolas Bourbaki, \emph{Lie groups and {L}ie algebras. {C}hapters 7--9},
  Elements of Mathematics (Berlin), Springer-Verlag, Berlin, 2005, Translated
  from the 1975 and 1982 French originals by Andrew Pressley.

\bibitem[DH98]{DH}
Igor~V. Dolgachev and Yi~Hu, \emph{Variation of geometric invariant theory
  quotients}, Inst. Hautes \'Etudes Sci. Publ. Math. \textbf{87} (1998), 5--56,
  With an appendix by Nicolas Ressayre.

\bibitem[GKO85]{GKO}
P.~Goddard, A.~Kent, and D.~Olive, \emph{Virasoro algebras and coset space
  models}, Phys. Lett. B \textbf{152} (1985), no.~1-2, 88--92.

\bibitem[IM65]{IM}
N.~Iwahori and H.~Matsumoto, \emph{On some {B}ruhat decomposition and the
  structure of the {H}ecke rings of {${\germ p}$}-adic {C}hevalley groups},
  Inst. Hautes {\'E}tudes Sci. Publ. Math. (1965), no.~25, 5--48.

\bibitem[Kam96]{Kambayashi96:JofAlg}
T.~Kambayashi, \emph{Pro-affine algebras, {I}nd-affine groups and the
  {J}acobian problem}, J. Algebra \textbf{185} (1996), no.~2, 481--501.

\bibitem[KKM09]{KKM}
Michael Kapovich, Shrawan Kumar, and John~J. Millson, \emph{The eigencone and
  saturation for {S}pin(8)}, Pure Appl. Math. Q. \textbf{5} (2009), no.~2,
  Special Issue: In honor of Friedrich Herzebruch. Part 1, 755--780.

\bibitem[KM06]{KapMill:tssemigroup}
Misha Kapovich and John~J. Millson, \emph{Structure of the tensor product
  semigroup}, Asian J. Math. \textbf{10} (2006), no.~3, 493--539.

\bibitem[KM08]{KM}
Michael Kapovich and John~J. Millson, \emph{A path model for geodesics in
  {E}uclidean buildings and its applications to representation theory}, Groups
  Geom. Dyn. \textbf{2} (2008), no.~3, 405--480.

\bibitem[KN98]{KumarNori}
Shrawan Kumar and Madhav~V. Nori, \emph{Positivity of the cup product in
  cohomology of flag varieties associated to {K}ac-{M}oody groups}, Internat.
  Math. Res. Notices (1998), no.~14, 757--763.

\bibitem[Kos59]{Kostant:sl2triple}
Bertram Kostant, \emph{The principal three-dimensional subgroup and the {B}etti
  numbers of a complex simple {L}ie group}, Amer. J. Math. \textbf{81} (1959),
  973--1032.

\bibitem[KS09]{KS:affflagman}
Masaki Kashiwara and Mark Shimozono, \emph{Equivariant {$K$}-theory of affine
  flag manifolds and affine {G}rothendieck polynomials}, Duke Math. J.
  \textbf{148} (2009), no.~3, 501--538.

\bibitem[KT99]{KT:saturation}
Allen Knutson and Terence Tao, \emph{The honeycomb model of {${\rm GL}\sb
  n({\Bbb C})$} tensor products. {I}. {P}roof of the saturation conjecture}, J.
  Amer. Math. Soc. \textbf{12} (1999), no.~4, 1055--1090.

\bibitem[Kum02]{Kumar:KacMoody}
Shrawan Kumar, \emph{Kac-{M}oody groups, their flag varieties and
  representation theory}, Progress in Mathematics, vol. 204, Birkh\"auser
  Boston Inc., Boston, MA, 2002.

\bibitem[{Kum}12]{Kumar:positivity}
Shrawan {Kumar}, \emph{{Positivity in T-Equivariant K-theory of flag varieties
  associated to Kac-Moody groups}}, ArXiv e-prints (2012).

\bibitem[Kum14]{Kumar:survey}
Shrawan Kumar, \emph{A survey of the additive eigenvalue problem}, Transform.
  Groups \textbf{19} (2014), no.~4, 1051--1148, With an appendix by M.
  Kapovich.

\bibitem[Kum15]{Kumar:surveyEMS}
\bysame, \emph{Additive eigenvalue problem}, Eur. Math. Soc. Newsl. (2015),
  no.~98, 20--27.

\bibitem[KW88]{KW}
Victor~G. Kac and Minoru Wakimoto, \emph{Modular and conformal invariance
  constraints in representation theory of affine algebras}, Adv. in Math.
  \textbf{70} (1988), no.~2, 156--236.

\bibitem[MFK94]{GIT}
David Mumford, John Fogarty, and Frances Kirwan, \emph{Geometric invariant
  theory}, 3d ed., Springer Verlag, New York, 1994.

\bibitem[{Pet}04]{Peters:coursAG}
Chris {Peters}, \emph{An introduction to complex algebraic geometry with
  emphasis on the theory of surfaces}, Course notes available at \url{
  https://www-fourier.ujf-grenoble.fr/~peters/ConfsAndSchools/surface.f/surfcourse.pd},
  2004.

\bibitem[PR13]{PR:example}
B.~Pasquier and N.~Ressayre, \emph{The saturation property for branching
  rules---examples}, Exp. Math. \textbf{22} (2013), no.~3, 299--312.

\bibitem[Res10]{GITEigen}
Nicolas Ressayre, \emph{Geometric invariant theory and generalized eigenvalue
  problem}, Invent. Math. \textbf{180} (2010), 389--441.

\bibitem[Res11]{multi}
\bysame, \emph{Multiplicative formulas in {S}chubert calculus and quiver
  representation}, Indag. Math. (N.S.) \textbf{22} (2011), no.~1-2, 87--102.

\bibitem[Ric12]{Rich:mult}
Edward Richmond, \emph{A multiplicative formula for structure constants in the
  cohomology of flag varieties}, Michigan Math. J. \textbf{61} (2012), no.~1,
  3--17.

\bibitem[Sha81]{Sha81}
I.~R. Shafarevich, \emph{On some infinite-dimensional groups. {II}}, Izv. Akad.
  Nauk SSSR Ser. Mat. \textbf{45} (1981), no.~1, 214--226, 240.

\bibitem[Sta12]{Stampfli:JofAlge}
Immanuel Stampfli, \emph{On the topologies on ind-varieties and related
  irreducibility questions}, J. Algebra \textbf{372} (2012), 531--541.

\end{thebibliography}

\begin{center}
  -\hspace{1em}$\diamondsuit$\hspace{1em}-
\end{center}
\end{document}